\documentclass[10pt]{amsart}
\usepackage{times,amsmath,amsbsy,amssymb,amscd,mathrsfs}
\usepackage{graphicx,subfigure,epstopdf,wrapfig,chemarrow}

\usepackage{algorithm2e} 
\usepackage{multicol,multirow}
\usepackage{array,booktabs,tabularx}
\usepackage{mathtools}
\usepackage[usenames,dvipsnames,svgnames,table]{xcolor}
\usepackage[all]{xy}
\usepackage{wrapfig}
\usepackage{tcolorbox}
\usepackage{stmaryrd}

\usepackage{tikz,tikz-cd}
\usepackage[utf8]{inputenc}
\usepackage{pgfplots} 
\usepackage{pgfgantt}
\usepackage{pdflscape}
\pgfplotsset{compat=newest} 
\pgfplotsset{plot coordinates/math parser=false}
\newlength\fwidth

\definecolor{myBlue}{rgb}{0.0,0.0,0.55}
\usepackage[pdftex,colorlinks=true,citecolor=myBlue,linkcolor=myBlue]{hyperref}

\usepackage[hyperpageref]{backref}

\usepackage{comment,enumerate,multicol,xspace}
\usepackage{enumitem}
\usepackage{adjustbox}

  \newcounter{mnote}
  \let\oldmarginpar\marginpar
    \renewcommand\marginpar[1]{\-\oldmarginpar[\raggedleft\footnotesize #1]%
    {\raggedright\footnotesize #1}}

\newtheorem{theorem}{Theorem}[section]
\newtheorem{lemma}[theorem]{Lemma}

\newtheorem{example}[theorem]{Example}

\newtheorem{remark}[theorem]{Remark}

\newcommand{\dx}{\,{\rm d}x}
\newcommand{\dd}{\,{\rm d}}

\newcommand{\bs}{\boldsymbol}

\DeclareMathOperator*{\spa}{span}

\newcommand{\curl}{\operatorname{curl}}
\renewcommand{\div}{\operatorname{div}}
\newcommand{\grad}{\operatorname{grad}}
\newcommand{\rot}{\operatorname{rot}}
\newcommand{\tr}{\operatorname{tr}}
\newcommand{\dev}{\operatorname{dev}}
\newcommand{\sym}{\operatorname{sym}}
\newcommand{\skw}{\operatorname{skw}}

\newcommand{\mskw}{\operatorname{mskw}}
\newcommand{\vskw}{\operatorname{vskw}}
\newcommand{\defm}{\operatorname{def}}
\newcommand{\hess}{\operatorname{hess}}
\newcommand{\inc}{\operatorname{inc}}

\newcommand{\cott}{\operatorname{cott}}
\newcommand{\Oplus}{\ensuremath{\vcenter{\hbox{\scalebox{1.5}{$\oplus$}}}}}

\newcommand{\step}[1]{\noindent\raisebox{1.5pt}[10pt][0pt]{\tiny\framebox{$#1$}}\xspace}

\newcommand{\vertiii}[1]{{\left\vert\kern-0.25ex\left\vert\kern-0.25ex\left\vert #1 
    \right\vert\kern-0.25ex\right\vert\kern-0.25ex\right\vert}}

\allowdisplaybreaks[4]

\begin{document}
\title[Finite Element Conformal Complex]{Finite element conformal complexes in three dimensions}

 \author{Xuehai Huang}%
 \address{
School of Mathematics, Shanghai University of Finance and Economics, Shanghai 200433, China}%
 \email{huang.xuehai@sufe.edu.cn}%

 \thanks{This work was supported by the National Natural Science Foundation of China Project 12671432.}

\makeatletter
\@namedef{subjclassname@2020}{\textup{2020} Mathematics Subject Classification}
\makeatother
\subjclass[2020]{
65N30;   
58J10;   
65N12;   
}

\begin{abstract}
This paper develops three-dimensional finite element conformal Hessian and conformal elasticity complexes through a local-bubble Bernstein-Gelfand-Gelfand (BGG) construction parameterized by smoothness vectors. These complexes incorporate higher-order differential operators, including the linearized Cotton-York operator, and require conformal tensor spaces (i.e., symmetric and traceless tensor spaces) with nontrivial smoothness and trace conditions.
We apply the discrete BGG framework at the level of local bubble complexes parameterized by vertex, edge, and face smoothness vectors, together with geometric bubble decompositions and a reduction operation. This yields exact bubble complexes over explicit admissible regularity ranges, together with the face-bubble exactness needed to assemble families of global conforming complexes with varying vertex, edge, and face smoothness. The construction recovers known conformal Hessian and conformal elasticity complexes as particular choices, while producing broader smoothness-parameterized families, including new exact conformal elasticity complexes with strictly weaker interelement continuity.
The resulting complexes provide structure-preserving finite element building blocks for applications in relativity, Cosserat elasticity, and fluid mechanics.
 \end{abstract}
\keywords{Conformal Hessian complex, Conformal elasticity complex, Finite element conformal complexes, bubble finite element complexes, Bernstein-Gelfand-Gelfand construction}

\maketitle


\section{Introduction}

Hilbert complexes play a fundamental role in the theoretical analysis and the development of stable and robust numerical methods for solving partial differential equations~\cite{ArnoldFalkWinther2006,ArnoldFalkWinther2010,Arnold2018,ChenHuang2018}. To extend this framework, Arnold and Hu~\cite{ArnoldHu2021} introduced a systematic methodology for generating new complexes by applying the Bernstein-Gelfand-Gelfand (BGG) construction to existing Hilbert complexes, such as the de Rham complex.

Using the smooth finite elements developed in \cite{HuLinWu2024,ChenHuang2021,ChenChenGaoHuangEtAl2025},
we constructed finite element de Rham and Stokes complexes with varying degrees of smoothness in two and three dimensions in \cite{ChenHuang2024,ChenHuang2024a}.
Recently, in~\cite{ChenHuang2025,ChenHuang2024a}, we applied the BGG framework to these smooth finite element de Rham complexes. By further incorporating techniques such as the 
$t$-$n$ decomposition \cite{ChenHuang2024b,ChenChenHuangWei2024,ChenHuang2024c} and trace complexes, we systematically derived finite element Hessian, elasticity, and div-div complexes. These constructions unify and extend various tensor finite element complexes studied in~\cite{ChenHuang2020,ChenHuang2022b,ChenHuang2022,ChristiansenHuHu2018,HuLiang2021,HuMaZhang2021,HuLiangMa2022}. The discrete BGG construction of such tensor finite element complexes on cubical meshes was developed in~\cite{BonizzoniHuKanschatSap2025}.
For additional developments of conforming tensor finite element complexes, we refer to~\cite{ChenHuang2025a,ChenHuang2022d,ChristiansenHu2023,HuLiangMaZhang2024,HuLiangLin2024,ChristiansenGopalakrishnanGuzmanHu2024,GuzmanHuang2026}. Distributional finite element complexes can be found in~\cite{ChenHuangZhang2025,ChenHuang2025a,ChenHuHuang2018,HuLinZhang2025,HuLin2025,Christiansen2011}.
Further tensor finite elements are discussed in~\cite{ChenHuang2022a,ChenHuang2024b,HuangZhangZhouZhu2024,HuangTang2025,ChenHuang2025b,BerchenkoKoganGawlik2025,HuangTang2026,HuangZhao2026,ChenHuangSunTian2026} and the references therein.

In this paper, we will extend the discrete BGG construction in \cite{ChenHuang2025,ChenHuang2024a} to
the conformal Hessian complex 
\begin{equation}\label{intro:conformalHesscomplex3d}
\resizebox{.935\hsize}{!}{$
{\rm CH}\xrightarrow{\subset} H^2\xrightarrow{\dev\hess} H(\sym\curl; \mathbb{S}\cap\mathbb{T})\xrightarrow{\sym\curl} H(\div\div; \mathbb{S}\cap\mathbb{T}) \xrightarrow{\div{\div}} L^2\xrightarrow{}0,
$}
\end{equation}
and
the conformal elasticity complex
\begin{equation}\label{intro:conformalElascomplex3d}
{\rm CK}\xrightarrow{\subset} H^1(\mathbb R^3)\xrightarrow{\dev\defm} H(\cott; \mathbb{S}\cap\mathbb{T})\xrightarrow{\cott} H(\div; \mathbb{S}\cap\mathbb{T}) \xrightarrow{\div} L^2(\mathbb R^3)\xrightarrow{}0,
\end{equation}
where ${\rm CH}:= \mathbb P_1+{\rm span}\{\boldsymbol{x}^{\intercal}\boldsymbol{x}\}$, and
the space of conformal Killing fields \cite{Dain2006,LewintanMuellerNeff2021,WilliamsHong2024}
\begin{equation*}
{\rm CK}:= \{(\boldsymbol{x}\cdot\boldsymbol{x})\boldsymbol{a}-2(\boldsymbol{a}\cdot\boldsymbol{x})\boldsymbol{x}+b\boldsymbol{x}+\boldsymbol{c}\times\boldsymbol{x}+\boldsymbol{d}: \boldsymbol{a},\boldsymbol{c},\boldsymbol{d}\in\mathbb{R}^3, b\in\mathbb{R} \}.
\end{equation*}
The Sobolev spaces appearing in the conformal complexes \eqref{intro:conformalHesscomplex3d}-\eqref{intro:conformalElascomplex3d} are defined in Section~\ref{subsec:spaces}.
The conformal Hessian complex~\eqref{intro:conformalHesscomplex3d} arises in general relativity~\cite{QuennevilleBelair2015,BeigChrusciel2020}, and the conformal elasticity complex~\eqref{intro:conformalElascomplex3d} has applications in general relativity~\cite{Dain2006}, Cosserat elasticity~\cite{NeffJeong2009}, and fluid mechanics~\cite{GopalakrishnanLedererSchoeberl2020,FeireislNovotny2017}.
The $Q$-tensor, which characterizes uniaxial and biaxial nematic liquid crystals and vanishes in the isotropic phase~\cite{GennesProst1993}, serves as a prototypical example of a conformal tensor—a tensor that is both symmetric and traceless.

The construction of finite element discretizations with varying degrees of smoothness for the conformal complexes~\eqref{intro:conformalHesscomplex3d}-\eqref{intro:conformalElascomplex3d} is significantly more challenging than those in~\cite{ChenHuang2025,ChenHuang2024a}, as the involved tensors are conformal—that is, simultaneously symmetric and traceless—rather than merely symmetric or traceless. Moreover, the linearized Cotton-York operator $\cott := \curl S^{-1} \inc = \inc S^{-1} \curl$ is a third-order differential operator, further complicating the discretization. 
To the best of our knowledge, the only known finite element conformal Hessian complex was recently constructed in~\cite{GuoHuLin2025,GuoHuLin2025a}, and the only known finite element conformal elasticity complex was introduced in~\cite{HuLinShi2023}.

As a first step, we combine the polynomial elasticity complex \cite[(2.6)]{ArnoldAwanouWinther2008}, the polynomial Hessian complex \cite[(3.4.9)]{ChenHuang2022d}, and the polynomial de Rham complex \cite[(3.3)]{ArnoldFalkWinther2006} to form the following diagram:
\begin{equation*}
\begin{tikzcd}
{\rm RM} \longrightarrow \mathbb P_{k+4}(\mathbb R^3)
\arrow{r}{\defm}
&
\mathbb P_{k+3}(\mathbb S)
 \arrow{r}{\inc}
 &
\mathbb P_{k+1}(\mathbb S)
 \arrow{r}{\div}
 & \mathbb P_{k}(\mathbb R^3) \longrightarrow 0
 \\
\mathbb{P}_1 \longrightarrow \mathbb P_{k+3}
 \arrow[ur,swap,"\iota"'] \arrow{r}{\hess}
 & 
\mathbb P_{k+1}(\mathbb S) 
 \arrow[ur,swap,"S"'] \arrow{r}{\curl}
 & 
\mathbb P_{k}(\mathbb T) 
 \arrow[ur,swap,"-2\vskw"'] \arrow{r}{\div} \arrow[r] 
 & \mathbb P_{k-1}(\mathbb R^3) \longrightarrow 0 
 \\
\mathbb{R} \longrightarrow \mathbb P_{k+1}
 \arrow[ur,swap,"\iota"'] \arrow{r}{\grad}
 & 
\mathbb P_{k}(\mathbb R^3) 
 \arrow[ur,swap,"\mskw"'] \arrow{r}{\curl}
 & 
\mathbb P_{k-1}(\mathbb R^3) 
 \arrow[ur,swap,"{\rm id}"'] \arrow{r}{\div} \arrow[r] 
 & \mathbb P_{k-2} \longrightarrow 0. 
\end{tikzcd}
\end{equation*}
Applying the BGG framework to this diagram yields the following exact polynomial conformal elasticity complex from the first two rows:
\begin{equation}\label{intro:polyconformalElascomplex3d}
{\rm CK}\xrightarrow{\subset} \mathbb P_{k+4}(\mathbb R^3)\xrightarrow{\dev\defm} \mathbb P_{k+3}(\mathbb{S}\cap\mathbb{T})\xrightarrow{\cott} \mathbb P_{k}(\mathbb{S}\cap\mathbb{T}) \xrightarrow{\div} \mathbb P_{k-1}(\mathbb R^3)\xrightarrow{}0.
\end{equation}
The last two rows similarly yield the following exact polynomial conformal Hessian complex:
\begin{equation}\label{intro:polyconformalHesscomplex3d}
{\rm CH}\xrightarrow{\subset} \mathbb P_{k+3}\xrightarrow{\dev\hess} \mathbb P_{k+1}(\mathbb{S}\cap\mathbb{T})\xrightarrow{\sym\curl} \mathbb P_{k}(\mathbb{S}\cap\mathbb{T}) \xrightarrow{\div{\div}} \mathbb P_{k-2}\xrightarrow{}0.
\end{equation}
These polynomial conformal complexes~\eqref{intro:polyconformalElascomplex3d}-\eqref{intro:polyconformalHesscomplex3d} are used to perform dimension counts in establishing the exactness of the corresponding finite element conformal complexes.

In contrast to applying the BGG framework directly to global finite element complexes as in~\cite{ChenHuang2025,ChenHuang2024a,BonizzoniHuKanschatSap2025}, we apply it to smoothness-parameterized local bubble complexes. This makes the smoothness shifts induced by the differential operators explicit and allows the vertex, edge, and face regularities to be varied without constructing a new global BGG diagram for each choice.

For the overview below, a smoothness vector $\boldsymbol r=(r^{\texttt{v}},r^e,r^f)^{\intercal}$ records the orders at vertices, edges, and faces, with $-1$ meaning that no corresponding smoothness is imposed; $\boldsymbol r\ominus m:=\max\{\boldsymbol r-m,-1\}$ componentwise. The notation $\mathbb B_k^{\mathrm d}(\boldsymbol r;\mathbb X)$ denotes the local polynomial bubble space with target $\mathbb X$ and the homogeneous traces associated with the operator $\mathrm d$. Precise definitions are given in Section~\ref{subsec:bubblespaces}.

Let polynomial degree $k\geq 2r_2^{\texttt{v}}+4$ and smoothness vectors
\begin{equation*}
\boldsymbol r_1 \geq (1,0,-1)^{\intercal},
\quad
\boldsymbol{r}_0= \boldsymbol r_1+2, \quad
 \boldsymbol{r}_2=\boldsymbol{r}_1\ominus1,
\quad
 \boldsymbol{r}_3=\boldsymbol{r}_2\ominus2.
\end{equation*}
By combining the bubble de Rham complex and the bubble div\,div complex, and applying the geometric decomposition of the bubble space $\mathbb B_{k+1}^{\sym\curl}(\boldsymbol{r}_1;\mathbb T)$, we construct the diagram
\begin{equation*}
\begin{tikzcd}[column sep=normal, row sep=normal]
&
\mathbb B_{k+1}^{\div}(\boldsymbol{r}_1)
  \arrow{r}{\div}
 &
\mathbb B_{k}(\boldsymbol{r}_2)
\arrow{r}{}
& \mathbb R \\
\mathbb B_{k+1}^{\sym\curl}(\boldsymbol{r}_1;\mathbb T)
 \arrow[ur,swap,"- 2\vskw"'] \arrow{r}{\sym\curl}
 & 
\mathbb B_{k}^{\div\div}(\boldsymbol{r}_2;\mathbb S)
 \arrow[ur,swap,"\tr"'] \arrow{r}{\div\div}
 & 
\mathbb B_{k-2}(\boldsymbol{r}_3)/\mathbb P_{1}(T)
\arrow[r] 
& 0,
\end{tikzcd}
\end{equation*}
and the diagram
\begin{equation*}
\adjustbox{scale=0.9,center}{%
\begin{tikzcd}[column sep=normal, row sep=normal]
\mathbb B_{k+3}(\boldsymbol{r}_0)\arrow{r}{\grad} 
&
\mathbb B_{k+2} (\boldsymbol{r}_1+1;\mathbb R^3)
  \arrow{r}{\curl}
 &
\mathbb B_{k+1}^{\div}(\boldsymbol{r}_1)
\arrow{r}{\div}
& \mathbb B_{k}(\boldsymbol{r}_2)/\mathbb R \\
\mathbb B_{k+2} (\boldsymbol{r}_1+1;\mathbb R^3)
 \arrow[ur,swap,"{\rm id}"'] \arrow{r}{\dev\grad}
 & 
\mathbb B_{k+1}^{\sym\curl}(\boldsymbol{r}_1;\mathbb T)
 \arrow[ur,swap,"- 2\vskw"'] \arrow{r}{\sym\curl}
 & 
\mathbb B_{k}^{\div\div}(\boldsymbol{r}_2;\mathbb S)\cap\ker(\div\div)
\arrow[r]
& 0.
\end{tikzcd}
}
\end{equation*}
Then apply the BGG framework to derive the bubble conformal Hessian complex 
\begin{align*}
0\xrightarrow{\subset} \mathbb B_{k+3}(\boldsymbol{r}_0)&\xrightarrow{\dev\hess} \mathbb B_{k+1}^{\sym\curl}(\boldsymbol{r}_1;\mathbb S\cap\mathbb T)\xrightarrow{\sym\curl} \mathbb B_{k}^{\div\div}(\boldsymbol{r}_2;\mathbb S\cap\mathbb T) \\
\notag
&\xrightarrow{\div{\div}} \mathbb B_{k-2}(\boldsymbol{r}_3)/{\rm CH}\xrightarrow{}0.
\end{align*}
Under these canonical smoothness shifts, the bubble and face-bubble complexes yield the following finite element conformal Hessian complex:
\begin{align*}
{\rm CH}\xrightarrow{\subset} \mathbb V_{k+3}(\boldsymbol{r}_0)&\xrightarrow{\dev\hess} \mathbb V_{k+1}^{\sym\curl}(\boldsymbol{r}_1;\mathbb S\cap\mathbb T)\xrightarrow{\sym\curl} \mathbb V_{k}^{\div\div}(\boldsymbol{r}_2;\mathbb S\cap\mathbb T) \\
\notag
&\xrightarrow{\div{\div}} \mathbb V_{k-2}(\boldsymbol{r}_3)\xrightarrow{}0.
\end{align*}
Remark~\ref{rem:femCHtilde} gives a two-parameter extension for $\boldsymbol r_2\geq\boldsymbol r_1\ominus1$ and $\boldsymbol r_3\geq\boldsymbol r_2\ominus2$, under the terminal stability condition stated there.
Taking $\boldsymbol{r}_0=(4,2,1)^{\intercal}$, $\boldsymbol{r}_1=\boldsymbol{r}_0-2$, $\boldsymbol{r}_2=\boldsymbol{r}_1\ominus1$, $\boldsymbol{r}_3=\boldsymbol{r}_2\ominus2$, and $k\geq6$, we obtain the finite element conformal Hessian complex
\begin{equation*}
{\rm CH} \hookrightarrow
\begin{pmatrix}
	 4\\
	 2\\
	 1
\end{pmatrix}
\xrightarrow{\dev\hess}
\begin{pmatrix}
 2\\
 0\\
 -1
\end{pmatrix}
\xrightarrow{\sym\curl}
\begin{pmatrix}
	1\\
	-1\\
	-1
\end{pmatrix}
\xrightarrow{\div\div} 
\begin{pmatrix}
	 -1\\
	 -1\\
	 -1
\end{pmatrix} \to 0,
\end{equation*}
which was recently constructed in~\cite{GuoHuLin2025,GuoHuLin2025a}.

Let polynomial degree $k\geq 2r_2^{\texttt{v}}+2$ and smoothness vectors 
\begin{equation*}
\boldsymbol{r}_0\geq (7,1,0)^{\intercal},  \quad 
\boldsymbol{r}_1=\boldsymbol{r}_0-1\geq (6,0,-1)^{\intercal},\quad \boldsymbol{r}_2=\boldsymbol{r}_1\ominus3,\quad \boldsymbol{r}_3=\boldsymbol{r}_2\ominus1.
\end{equation*}
For the higher-smoothness branch satisfying the input bubble-complex hypotheses in Section~\ref{subsec:bubbleCEcomplex}, combining the bubble elasticity and Hessian complexes and applying the tilde operation in~\cite{ChenHuang2025} gives the diagram
\begin{equation*}
\begin{tikzcd}[column sep=small, row sep=normal]
&
\mathbb B^{\div}_{k+1}(\boldsymbol{r}_2+1; \mathbb S)
  \arrow{r}{\div}
 &
\mathbb B_{k}(\boldsymbol{r}_2; \mathbb R^3)
\arrow{r}{}
& {\rm RM} \\
{\mathbb B}^{\curl}_{k+1}(\boldsymbol{r}_2+1;\mathbb S)
 \arrow[ur,swap,"S"'] \arrow{r}{\curl}
 & 
\mathbb B^{\div}_{k}(\boldsymbol{r}_2;\mathbb T)
 \arrow[ur,swap,"-2\vskw "'] \arrow{r}{\div}
 & 
\mathbb B_{k-1}(\boldsymbol{r}_3;\mathbb R^3)/{\rm RT}
\arrow[r] 
&0,
\end{tikzcd}
\end{equation*}
and the diagram
\begin{equation*}
\begin{tikzcd}
\mathbb B_{k+4}(\boldsymbol{r}_0;\mathbb R^3)
\arrow{r}{\defm}
&
\widetilde{\mathbb B}_{k+3}^{\inc} (\boldsymbol{r}_1; \mathbb S)
 \arrow{r}{\inc}
 &
\widetilde{\mathbb B}^{\div}_{k+1}(\boldsymbol{r}_1-2; \mathbb S)
 \arrow{r}{\div}
 & 0
 \\
\mathbb B_{k+3}(\boldsymbol{r}_1)
 \arrow[ur,swap,"\iota"'] \arrow{r}{\hess}
 & 
\widetilde{\mathbb B}_{k+1}^{\curl}(\boldsymbol{r}_1-2;\mathbb S) 
 \arrow[ur,swap,"S"'] \arrow{r}{\curl}
 & 
\mathbb B^{\div}_{k}(\boldsymbol{r}_2; \mathbb S\cap \mathbb T)\cap\ker(\div). 
 \arrow[ur,swap,"-2\vskw"'] 
\end{tikzcd}
\end{equation*}
The remaining admissible parameters are covered by the reductions and the auxiliary endpoint argument in Section~\ref{subsec:bubbleCEcomplex}. Together these arguments give the bubble conformal elasticity complex
\begin{align*}
0\xrightarrow{\subset}\mathbb B_{k+4}(\boldsymbol{r}_0;\mathbb R^3) &\xrightarrow{\dev\defm} \mathbb B_{k+3}^{\cott}(\boldsymbol{r}_1;\mathbb{S}\cap\mathbb{T}) \xrightarrow{\cott}  \mathbb B^{\div}_{k}(\boldsymbol{r}_2; \mathbb S\cap \mathbb T) \\
&\xrightarrow{\div} \mathbb B_{k-1}(\boldsymbol{r}_3;\mathbb R^3)/{\rm CK}\rightarrow 0. \notag
\end{align*}
Under these canonical smoothness shifts, now with the global degree condition $k\geq2r_2^{\texttt{v}}+5$, the bubble and face-bubble complexes yield the following finite element conformal elasticity complex:
\begin{align*}
{\rm CK}\xrightarrow{\subset} \mathbb V_{k+4}(\boldsymbol{r}_0;\mathbb R^3)&\xrightarrow{\dev\defm} \mathbb V_{k+3}^{\cott}(\boldsymbol{r}_1;\mathbb S\cap\mathbb T)\xrightarrow{\cott} \mathbb V_{k}^{\div}(\boldsymbol{r}_2;\mathbb S\cap\mathbb T) \\
\notag
&\xrightarrow{\div} \mathbb V_{k-1}(\boldsymbol{r}_3;\mathbb R^3)\xrightarrow{}0.
\end{align*}
Remark~\ref{rem:femCEtilde} gives a two-parameter extension for $\boldsymbol r_2\geq\boldsymbol r_1\ominus3$ and $\boldsymbol r_3\geq\boldsymbol r_2\ominus1$, under the terminal surjectivity assumption stated there.
Taking $\boldsymbol{r}_0=(7,1,0)^{\intercal}$, $\boldsymbol{r}_1=\boldsymbol{r}_0-1$, $\boldsymbol{r}_2=(3,0,-1)^{\intercal}$, $\boldsymbol{r}_3=\boldsymbol{r}_2\ominus1$, and $k\geq11$, we recover the complex in \cite{HuLinShi2023}.
Alternatively, taking $\boldsymbol{r}_0=(7,1,0)^{\intercal}$, $\boldsymbol{r}_1=\boldsymbol{r}_0-1$, $\boldsymbol{r}_2=\boldsymbol{r}_1\ominus3$, $\boldsymbol{r}_3=\boldsymbol{r}_2\ominus1$, and $k\geq11$, we obtain the finite element conformal elasticity complex
\begin{equation*}
{\rm CK} \hookrightarrow
\begin{pmatrix}
	 7\\
	 1\\
	 0
\end{pmatrix}
\xrightarrow{\dev\defm}
\begin{pmatrix}
 6\\
 0\\
 -1
\end{pmatrix}
\xrightarrow{\cott}
\begin{pmatrix}
	3\\
	-1\\
	-1
\end{pmatrix}
\xrightarrow{\div} 
\begin{pmatrix}
	 2\\
	 -1\\
	 -1
\end{pmatrix} \to 0,
\end{equation*}
which is less smooth than the complex in \cite{HuLinShi2023}.

\subsection*{Comparison with previous constructions}

The use of BGG is common to all three constructions, including those in~\cite{GuoHuLin2025,HuLinShi2023}, and is not, by itself, the distinction claimed here. The difference is the use of smoothness-parameterized local bubble input complexes and the way these local constructions are assembled into families parameterized by vertex, edge, and face regularities across the full sequence. Table~\ref{tab:comparison-related} summarizes the technical comparison.

For the conformal Hessian complex, our BGG construction was developed independently and contemporaneously with~\cite{GuoHuLin2025}. At the fixed smoothness pattern considered there, the two constructions use the same underlying BGG structure, and our parameter choice recovers the conformal Hessian complex of~\cite{GuoHuLin2025}. The distinction is that Lemmas~4.8--4.11 establish the geometric decomposition, the skew-part and trace images, the $\div\div$ stability, and the $\sym\curl$ image uniformly in the smoothness vector. Consequently, the same local-bubble BGG construction yields exact bubble and global complexes throughout the admissible smoothness range, rather than requiring a separate construction for a prescribed continuity pattern. At the $H(\div\div;\mathbb S\cap\mathbb T)$-space level, Example~\ref{ex:lower-vertex-divdiv} further lowers the vertex smoothness from $\boldsymbol r_2=(1,-1,-1)^{\intercal}$ in~\cite{GuoHuLin2025} to $\boldsymbol r_2=(0,-1,-1)^{\intercal}$.

For the conformal elasticity complex, the bubble Hessian complex and reduction operation close the parameterized BGG diagram under the canonical shifts $\boldsymbol r_1=\boldsymbol r_0-1$, $\boldsymbol r_2=\boldsymbol r_1\ominus3$, and $\boldsymbol r_3=\boldsymbol r_2\ominus1$, yielding the exact bubble and global complexes in Theorems~\ref{thm:bubbleCEcomplex3d} and~\ref{thm:femCEcomplex3d}. More importantly, this construction is compatible with the smoothness-vector framework of the input complexes: the differential orders are reflected directly in the shifts of vertex, edge, and face regularities, so the same construction extends systematically to different continuity patterns. Remark~\ref{rem:femCEtilde} gives the inequality extension for sufficiently large $k$ under its terminal surjectivity assumption. The choice $\boldsymbol r_2=(3,0,-1)^{\intercal}$ recovers~\cite{HuLinShi2023}, whereas $\boldsymbol r_2=(3,-1,-1)^{\intercal}$ yields a new exact global complex whose $H(\div;\mathbb S\cap\mathbb T)$ space drops the full $C^0$ edge continuity imposed in~\cite{HuLinShi2023}; among the edge components, only the normal--normal block is required to be single-valued, while the face-normal continuity required for $H(\div)$-conformity is retained. We also note that~\cite{HuLinShi2023} develops lower-smoothness bubble complexes and a family of $H(\cott)$ spaces. Its intrinsic supersmoothness and inf--sup stability results are complementary to the smoothness-parameterized complex construction developed here.
\begin{table}[htbp]
\centering
\footnotesize
\caption{Technical comparison with~\cite{GuoHuLin2025,HuLinShi2023}.}
\label{tab:comparison-related}
\setlength{\tabcolsep}{4pt}
\renewcommand{\arraystretch}{1.05}
\makeatletter
\newcommand{\tablecite}[1]{{\def\@citestyle{\m@th\upshape\bfseries}\cite{#1}}}
\makeatother
\begin{tabularx}{\textwidth}{>{\raggedright\arraybackslash}p{0.14\textwidth}>{\raggedright\arraybackslash}X>{\raggedright\arraybackslash}X}
\toprule
\textbf{Aspect}
& \textbf{Conformal Hessian: Present vs.~\tablecite{GuoHuLin2025}}
& \textbf{Conformal elasticity: Present vs.~\tablecite{HuLinShi2023}} \\
\midrule
\textbf{Regularity}
& \textbf{Present:} smoothness-vector family over explicit admissible ranges; at the $H(\div\div;\mathbb S\cap\mathbb T)$-space level, Example~\ref{ex:lower-vertex-divdiv} further allows $\boldsymbol r_2=(0,-1,-1)^{\intercal}$. \par\noindent\textbf{\tablecite{GuoHuLin2025}:} its exact global complex, for $k\geq6$, uses the prescribed smoothness pattern $(4,2,1)^{\intercal}\to(2,0,-1)^{\intercal}\to(1,-1,-1)^{\intercal}\to(-1,-1,-1)^{\intercal}$.
& \textbf{Present:} smoothness-vector family across the full sequence, including the new exact choice $\boldsymbol r_2=(3,-1,-1)^{\intercal}$. \par\noindent\textbf{\tablecite{HuLinShi2023}:} lower-smoothness bubble complexes and a family of $H(\cott)$-conforming spaces; the exact global complex displayed in Section~8 uses $(7,1,0)^{\intercal}\to(6,0,-1)^{\intercal}\to(3,0,-1)^{\intercal}\to(2,-1,-1)^{\intercal}$. \\
\midrule
\textbf{Bubble BGG input}
& \textbf{Present:} parameterized bubble de Rham and div\,div complexes. \par\noindent\textbf{\tablecite{GuoHuLin2025}:} the same underlying BGG ingredients at its fixed smoothness pattern.
& \textbf{Present:} bubble elasticity and Hessian complexes with reduction. \par\noindent\textbf{\tablecite{HuLinShi2023}:} bubble elasticity and div\,div complexes for its smoothest bubble BGG construction. \\
\midrule
\textbf{Bubble structure}
& \textbf{Present:} geometric decompositions and skew/trace image identities uniform in the smoothness vector. \par\noindent\textbf{\tablecite{GuoHuLin2025}:} explicit $H(\sym\curl)$ and complement-based $H(\div\div)$ bubble characterizations.
& \textbf{Present:} $H(\cott)$ bubble characterizations and reduction-space decompositions used in the smoothness-parameterized BGG construction. \par\noindent\textbf{\tablecite{HuLinShi2023}:} nested $H(\cott)$ bubble spaces and intrinsic supersmoothness of the divergence-free $H(\div)$ bubble subspace. \\
\midrule
\textbf{DoFs and traces}
& \textbf{Present:} entitywise DoFs vary with the smoothness entries and use parameterized face-bubble complexes. \par\noindent\textbf{\tablecite{GuoHuLin2025}:} trace-complex-based DoFs for the fixed smoothness pattern.
& \textbf{Present:} entitywise parameterized DoFs; the choice $\boldsymbol r_2=(3,-1,-1)^{\intercal}$ drops full $C^0$ edge continuity while retaining the single-valued $H(\div)$ face trace $\boldsymbol{\tau}\boldsymbol n$. \par\noindent\textbf{\tablecite{HuLinShi2023}:} DoFs enforce edge continuity and the three Cotton traces, with auxiliary face complexes used to determine $\tr_2^{\cott}$ and $\tr_3^{\cott}$. \\
\midrule
\textbf{Global exactness}
& \textbf{Present:} parameterized bubble and face-bubble exactness and stability, followed by an Euler dimension argument. \par\noindent\textbf{\tablecite{GuoHuLin2025}:} fixed bubble exactness, bubble-corrected interpolation, and a global dimension count.
& \textbf{Present:} parameterized bubble exactness and stability, followed by an Euler dimension argument. \par\noindent\textbf{\tablecite{HuLinShi2023}:} bubble exactness, div stability, supersmoothness-based kernel identification, and a global dimension count. \\
\bottomrule
\end{tabularx}
\end{table}

The remainder of this paper is organized as follows. Section~\ref{sec:preliminary} introduces the notation used throughout the paper, including tangential and normal vectors, smoothness vectors, tensors, differential operators, polynomial and Sobolev spaces, as well as triangulations. Section~\ref{sec:conformalpolydiffcomplex} develops conformal polynomial and differential complexes via the BGG framework. Section~\ref{sec:bubblecomplex} presents the construction of the bubble conformal Hessian complex and the bubble conformal elasticity complex, also based on the BGG framework. The finite element conformal Hessian complex is discussed in Section~\ref{sec:FEconformalHesscomplex}, and the finite element conformal elasticity complex in Section~\ref{sec:FEconformalelascomplex}. Appendix~\ref{apdx:continuous-exactness} gives detailed proofs of Lemmas~\ref{lem:reducedelasticityexact} and~\ref{lem:continuousCEexact}, which establish the exactness of the modified elasticity complex and the continuous conformal elasticity complex, respectively, and Appendix~\ref{apdx:uinsol} addresses the number of degrees of freedom in \eqref{eq:cottSTdofs} for the $H(\cott;\mathbb S\cap\mathbb T)$-conforming finite element.

\section{Preliminaries}\label{sec:preliminary}

We use the notation of~\cite{ChenHuang2025,ChenHuang2024,ChenHuang2024a} for smoothness vectors, tensor calculus, function spaces, and triangulations.


\subsection{Notation}
Let $\Omega\subset\mathbb R^3$ be a bounded contractible polyhedron.
For each two-dimensional polygon $f$, we select a unit normal $\boldsymbol{n}_f$ and an oriented orthonormal tangential basis $(\boldsymbol{t}_1^f,\boldsymbol{t}_2^f)$ satisfying $\boldsymbol{t}_1^f\times\boldsymbol{t}_2^f=\boldsymbol{n}_f$; write $\boldsymbol{n}$, $\boldsymbol{t}_1$, and $\boldsymbol{t}_2$ when $f$ is clear.
For each one-dimensional edge $e$, we choose a unit tangential vector $\boldsymbol{t}_e$ and two orthonormal normal vectors $\boldsymbol{n}_1^e$ and $\boldsymbol{n}_2^e$ such that $\boldsymbol n_{1}^e\times\boldsymbol n_{2}^e=\boldsymbol t_{e}$, abbreviated as $\boldsymbol{t}$, $\boldsymbol{n}_1$, and $\boldsymbol{n}_2$.
For an edge $e$ of polygon $f$, denote by $\boldsymbol n_{f,e}$ the unit vector being parallel to $f$ and outward normal to $\partial f$. Set $\boldsymbol t_{f,e}:=\boldsymbol n_{f}\times \boldsymbol n_{f,e}$.
For a given plane $f$ with a normal vector $\boldsymbol n$, we define the tangential projection matrix as
$
\Pi_f: = \boldsymbol I - \boldsymbol n\boldsymbol n^{\intercal}.
$

An integer vector $\boldsymbol r= (r^{\texttt{v}}, r^e)^{\intercal}$ is called a two-dimensional smoothness vector if $r^e\geq -1$ and $r^{\texttt{v}}\geq \max\{2r^e,-1\}$.
An integer vector $\boldsymbol r= (r^{\texttt{v}}, r^e, r^f)^{\intercal}$ is called a three-dimensional smoothness vector if $r^f\geq-1$, $r^e\geq \max\{2r^f,-1\}$ and $r^{\texttt{v}}\geq \max\{2r^e,-1\}$. Its restriction $(r^{\texttt{v}}, r^e)^{\intercal}$ is a two-dimensional smoothness vector. For a smoothness vector $\boldsymbol r$ and an integer $m\geq1$, define $\boldsymbol r\ominus m: = \max \{\boldsymbol r-m, -1\}$ and $\boldsymbol r_+ = \max\{\boldsymbol r, 0\}$. 

We will use Iverson bracket  $[\cdot]$~\cite{Iverson1962}, which extends the Kronecker delta function to a statement $P$
$$
[P]= 
\begin{cases}1 & \text { if } P \text { is true}, \\ 0 & \text { otherwise}.\end{cases}
$$

\subsection{Tensors}
Set $\mathbb M:=\mathbb R^{d\times d}$ for $d=2,3$.
Denote by $\mathbb{S}$, $\mathbb{K}$, and $\mathbb{T}$ the subspaces of symmetric, skew-symmetric, and traceless matrices in $\mathbb{M}$, respectively. 
A matrix $\boldsymbol{\tau}$ can be decomposed into symmetric and skew-symmetric parts as  
\[
\boldsymbol{\tau} = \sym\boldsymbol{\tau} + \skw\boldsymbol{\tau},
\]
where  
\[
\sym\boldsymbol{\tau} := \frac{1}{2}(\boldsymbol{\tau} + \boldsymbol{\tau}^{\intercal}), \quad  
\skw\boldsymbol{\tau} := \frac{1}{2}(\boldsymbol{\tau} - \boldsymbol{\tau}^{\intercal}).
\]  
Alternatively, $\boldsymbol{\tau}$ admits the decomposition  
\[
\boldsymbol{\tau} = \dev\boldsymbol{\tau} + \frac{1}{d} \iota(\tr\boldsymbol{\tau}),
\]
where the deviatoric part is given by $\dev\boldsymbol{\tau} := \boldsymbol{\tau} - \frac{1}{d} (\tr\boldsymbol{\tau}) \boldsymbol{I}$, the trace by $\tr\boldsymbol{\tau} := \boldsymbol{\tau} : \boldsymbol{I}$, and the operator $\iota(v)$ is defined as $\iota(v) := v\boldsymbol{I}$.  
Here, $\boldsymbol{\sigma}:\boldsymbol{\tau}$ means the Frobenius inner product of tensors $\boldsymbol{\sigma}$ and $\boldsymbol{\tau}$.

Introduce an isomorphism from $\mathbb R^3$ to the space of skew-symmetric matrices $\mathbb K$ as follows: for a vector $\boldsymbol  \omega =
( \omega_1, \omega_2, \omega_3)^{\intercal}
\in \mathbb R^3,$
\begin{equation*}
\mskw \boldsymbol  \omega :=
\begin{pmatrix}
 0 & -\omega_3 & \omega_2 \\
\omega_3 & 0 & - \omega_1\\
-\omega_2 & \omega_1 & 0
\end{pmatrix}. 
\end{equation*}
Obviously $\mskw: \mathbb R^3 \to \mathbb K$ is a bijection. Define $\vskw: \mathbb M\to \mathbb R^3$ by $\vskw := \mskw^{-1}\circ \skw$.

We define the dot product and the cross product from the left, denoted as $\boldsymbol b\cdot \boldsymbol A$ and $\boldsymbol b\times \boldsymbol A$ respectively. These operations are applied column-wise to the matrix $\boldsymbol A$. Conversely, when the vector is on the right of the matrix, i.e., $\boldsymbol A\cdot \boldsymbol b$ and $\boldsymbol A\times \boldsymbol b$, the operations are defined row-wise. The order of performing the row and column products is interchangeable, resulting in the associative rule of triple products:
$$
\boldsymbol b\times \boldsymbol A\times \boldsymbol c := (\boldsymbol b\times \boldsymbol A)\times \boldsymbol c = \boldsymbol b\times (\boldsymbol A\times \boldsymbol c).
$$
Similar rules apply for $\boldsymbol b\cdot \boldsymbol A\cdot \boldsymbol c$ and $\boldsymbol b\cdot \boldsymbol A\times \boldsymbol c$, allowing parentheses to be omitted.

\subsection{Differential operators}
Let $\nabla = (\partial_1, \partial_2, \partial_3)^{\intercal}$ be the gradient operator. For a vector $\boldsymbol{v}$, define the gradient, curl, and divergence operators as follows:
\begin{equation*}
\nabla\boldsymbol{v}:=\nabla\otimes\boldsymbol{v},\;\; \grad\boldsymbol{v}:=(\nabla\boldsymbol{v})^{\intercal}=\boldsymbol{v}\otimes\nabla,\;\; \curl\boldsymbol{v}:=\nabla\times\boldsymbol{v},\;\; \div\boldsymbol{v}:=\nabla\cdot\boldsymbol{v}.
\end{equation*}
Set deformation tensor $\defm(\boldsymbol{v}):=\sym(\grad\boldsymbol{v})$.
For a matrix-valued function, the differential operators $\curl$ and $\div$ are applied row-wise; that is,
\begin{equation*}
\curl\boldsymbol{\tau}=-\boldsymbol{\tau}\times\nabla,\quad \div\boldsymbol{\tau}=\boldsymbol{\tau}\cdot\nabla.
\end{equation*} 

We further introduce the incompatibility operator $\inc: \mathbb{M}\to\mathbb{S}$ and the linearized Cotton-York operator $\cott: \mathbb{M}\to\mathbb{S}\cap\mathbb{T}$ as
\begin{align*}
\inc\boldsymbol{\tau}:=\curl S^{-1}(\curl\boldsymbol{\tau}),\quad\cott\boldsymbol{\tau}:=\curl S^{-1}(\inc\boldsymbol{\tau})=\inc S^{-1}(\curl\boldsymbol{\tau}),
\end{align*}
where $S\boldsymbol{\sigma}:=\boldsymbol{\sigma}^{\intercal}-(\tr\boldsymbol{\sigma})\boldsymbol{I}$ and $S^{-1}\boldsymbol{\sigma}:=\boldsymbol{\sigma}^{\intercal}-\frac{1}{2}(\tr\boldsymbol{\sigma})\boldsymbol{I}$. It can be verified that $\inc\boldsymbol{\tau}=0$ for skew-symmetric $\boldsymbol{\tau}$.
When $\boldsymbol{\tau}$ is symmetric, we have
\begin{equation}\label{eq:incandcott}
\inc\boldsymbol{\tau}=\curl(\curl\boldsymbol{\tau})^{\intercal}=-\nabla\times\boldsymbol{\tau}\times\nabla,\quad \cott\boldsymbol{\tau}=\inc(\curl\boldsymbol{\tau})=\inc(\sym\curl\boldsymbol{\tau}).
\end{equation}
Hence, for general tensor $\boldsymbol{\tau}$,
\begin{align*}
\inc\boldsymbol{\tau}&=\inc(\sym\boldsymbol{\tau})=-\nabla\times(\sym\boldsymbol{\tau})\times\nabla, \\  
\cott\boldsymbol{\tau}&=\inc(S^{-1}(\sym\curl \boldsymbol{\tau}))=\sym\curl(\inc\boldsymbol{\tau})=\cott(\dev\sym\boldsymbol{\tau}).
\end{align*}

For face $f$, introduce the following definitions:
$$
\nabla_f: = \Pi_f \nabla, \quad
\nabla_f^{\bot} := \boldsymbol n \times \nabla.
$$
For a scalar function $v$, we have:
\begin{align*}
\grad_f v = \nabla_f v &= \Pi_f (\nabla v) = - \boldsymbol n \times (\boldsymbol n\times \nabla v), \\
\curl_f v = \nabla_f^{\bot} v &= \boldsymbol n \times \nabla v = \boldsymbol n\times \nabla_f v,
\end{align*}
where $\grad_f v$ is the surface gradient of $v$, and $\curl_f v$ is the surface $\curl$. 
For a vector function $\boldsymbol v$, the surface divergence is defined as:
$$
\div_f\boldsymbol v: = \nabla_f\cdot \boldsymbol v = \nabla_f\cdot (\Pi_f\boldsymbol v).
$$
Furthermore, the surface rot operator is given by:
\begin{equation*}
{\rm rot}_f \boldsymbol v := \nabla_f^{\bot}\cdot \boldsymbol v = (\boldsymbol n\times \nabla)\cdot \boldsymbol v = \boldsymbol n\cdot (\nabla \times \boldsymbol v),
\end{equation*}
representing the normal component of $\curl\boldsymbol{v}=\nabla \times \boldsymbol v$. 
For a tensor $\boldsymbol{\tau}$, define the surface trace and the surface deviation as follows:
\begin{equation*}
\tr_f(\boldsymbol{\tau}):=\tr(\Pi_f\boldsymbol{\tau}\Pi_f),\quad \dev_f(\boldsymbol{\tau}):=\Pi_f\boldsymbol{\tau}\Pi_f-\frac{1}{2}\tr_f(\boldsymbol{\tau})I_f\quad\textrm{ with } I_f=\Pi_fI\Pi_f.
\end{equation*}

\subsection{Polynomial and Sobolev spaces}\label{subsec:spaces}
For a bounded domain $D\subset\mathbb{R}^{d}$ and 
a non-negative integer $k$, 
let $\mathbb P_k(D)$ stand for the set of all polynomials over $D$ with the total degree no more than $k$. When $k<0$, set $\mathbb P_k(D) := \{0\}.$
In the binomial coefficient notation ${d\choose k}$, if $d\geq 0$ and $k<0$, we set ${d\choose k} := 0$.

Given a bounded domain $D\subset\mathbb{R}^{d}$ with $d=2,3$ and a
real number $s$, let $H^s(D)$ be the usual Sobolev space of functions
over $D$, whose norm is denoted by
$\Vert\cdot\Vert_{s,D}$. 
Define $H_0^k(D)$ as the closure of $C_{0}^{\infty}(D)$ with
respect to the norm $\Vert\cdot\Vert_{k,D}$.
Let $(\cdot, \cdot)_D$ be the standard inner product on $L^2(D)$. If $D$ is $\Omega$, 
we abbreviate $(\cdot, \cdot)_D$ by $(\cdot, \cdot)$.
For finite-dimensional polynomial spaces $V$ and $W$ contained in the same scalar-, vector-, or tensor-valued $L^2(D)$ space, we use the notation
\[
V/W:=V\cap W^{\perp_D}
=\{v\in V:(v,w)_D=0\quad\text{for all }w\in W\},
\]
where $W^{\perp_D}$ denotes the orthogonal complement with respect to the (componentwise, when appropriate) $L^2(D)$ inner product. We do not require $W\subseteq V$; thus $V/W$ denotes the $L^2(D)$-orthogonal complement of $W$ within $V$, rather than a quotient space.
We use $\bs n_{\partial D}$ to denote the unit outward normal vector of $\partial D$, which will be abbreviated as $\bs n$ if not causing any confusion. 
For a space $B(D)$ defined on $D$,
let $B(D; \mathbb{X}):=B(D)\otimes\mathbb{X}$ be its vector or tensor version for $\mathbb{X}$ being $\mathbb{R}^d$, $\mathbb{M}$, $\mathbb{S}$, $\mathbb{K}$ and $\mathbb{T}$.

Define Sobolev spaces
\begin{align*}
H(\curl,D)&:=\{\boldsymbol{v}\in L^2(D;\mathbb R^3): \curl\boldsymbol{v}\in L^2(D;\mathbb R^3)\},\\
H(\div,D)&:=\{\boldsymbol{v}\in L^2(D;\mathbb R^3): \div\boldsymbol{v}\in L^2(D)\},\\
H(\curl,D; \mathbb X)&:=\{\boldsymbol{\tau}\in L^{2}(D; \mathbb X): \curl\boldsymbol{\tau}\in L^{2}(D; \mathbb{M})\}, \\
H(\div,D; \mathbb X)&:=\{\boldsymbol{\tau}\in L^{2}(D; \mathbb X): \div\boldsymbol{\tau}\in L^{2}(D; \mathbb{R}^3)\}, \\
H(\sym\curl,D; \mathbb X)&:=\{\boldsymbol{\tau}\in L^{2}(D; \mathbb X): \sym\curl\boldsymbol{\tau}\in L^{2}(D; \mathbb{S})\}, \\
H(\div\div,D; \mathbb X)&:=\{\boldsymbol{\tau}\in L^{2}(D; \mathbb X): \div {\div}\boldsymbol{\tau}\in L^{2}(D)\}, \\
H(\inc,D; \mathbb X)&:=\{\boldsymbol{\tau}\in L^{2}(D; \mathbb X): \inc\boldsymbol{\tau}\in L^{2}(D; \mathbb{S})\}, \\
H(\cott,D; \mathbb{S}\cap\mathbb{T})&:=\{\boldsymbol{\tau}\in L^{2}(D;\mathbb{S}\cap\mathbb{T}): \cott\boldsymbol{\tau}\in L^{2}(D;\mathbb{S}\cap\mathbb{T})\},
\end{align*}
where $\mathbb X=\mathbb{S}, \mathbb{T}, \mathbb{S}\cap\mathbb{T}$.
Hereafter, we will skip $D$ in the space notation to save space if not causing any confusion.
For example, we abbreviate $H^s(D)$ as $H^s$, and $\mathbb P_k(D)$ as $\mathbb P_k$.

\subsection{Triangulation}
Denote by $\mathcal{T}_h$ a conforming triangulation of $\Omega$ with each element being a simplex, where $h:=\max_{T\in\mathcal{T}_h}h_T$.
Let $\Delta_{\ell}(\mathcal{T}_h)$ denote the set of all $\ell$-dimensional faces of $\mathcal{T}_h$ for $0\leq\ell\leq3$, and let $\Delta_{\ell}(\mathring{\mathcal{T}}_h)$ denote the set of all interior $\ell$-dimensional faces for $0\leq\ell\leq2$.
For a $d$-dimensional simplex $T$ with $d=2,3$, we let 
$\Delta_{\ell}(T)$ denote the set of all $\ell$-dimensional faces for $0\leq \ell \leq d$. 
Normal vectors \(\boldsymbol{n}_1^e\), \(\boldsymbol{n}_2^e\), or \(\boldsymbol{n}_f\) depend on the edge \(e\in\Delta_{1}(\mathcal{T}_h)\) or the face \(f\in\Delta_{2}(\mathcal{T}_h)\), not the element containing them.

Euler's formula holds~\cite{Hatcher2002}:
\begin{equation}\label{eulerformula}
|\Delta_0(\mathcal{T}_h)|-|\Delta_1(\mathcal{T}_h)|+|\Delta_2(\mathcal{T}_h)|-|\Delta_3(\mathcal{T}_h)|=1.
\end{equation}


\section{Conformal Polynomial and Differential Complexes}\label{sec:conformalpolydiffcomplex}

In this section, we will apply the BGG framework to derive conformal polynomial complexes and conformal differential complexes in three dimensions.
We also present traces for operators $\sym\curl$, $\div\div$, $\inc$ and $\cott$, and their properties for later use.

\subsection{Conformal polynomial complexes in three dimensions}
Let
\begin{itemize}
  \item ${\rm RT}:=\mathbb R^3\oplus\spa\{\boldsymbol{x}\}=\{\boldsymbol{a}+b\boldsymbol{x}: \boldsymbol{a}\in\mathbb{R}^3, b\in\mathbb{R}\}$;
  \item ${\rm RM}:=(\mathbb{R}^3\times\boldsymbol{x})\oplus\mathbb{R}^3=\{\boldsymbol{a}\times\boldsymbol{x}+\boldsymbol{b}: \boldsymbol{a},\boldsymbol{b}\in\mathbb{R}^3\}$;
  \item ${\rm CH}:= \mathbb P_1+{\rm span}\{\boldsymbol{x}^{\intercal}\boldsymbol{x}\}$.
\end{itemize}
Here ${\rm RT}$ is the shape function space of the lowest order Raviart-Thomas element \cite{RaviartThomas1977,Nedelec1980}, ${\rm RM}$ is the rigid motion space \cite{Sokolnikoff1946}, and ${\rm CH}$ is the kernel space of the operator $\dev\hess$ with the Hessian operator $\hess:=\nabla^2$.
 
Define the conformal Killing field \cite{Dain2006,LewintanMuellerNeff2021,WilliamsHong2024}
\begin{equation*}
{\rm CK}:= \{(\boldsymbol{x}\cdot\boldsymbol{x})\boldsymbol{a}-2(\boldsymbol{a}\cdot\boldsymbol{x})\boldsymbol{x}+b\boldsymbol{x}+\boldsymbol{c}\times\boldsymbol{x}+\boldsymbol{d}: \boldsymbol{a},\boldsymbol{c},\boldsymbol{d}\in\mathbb{R}^3, b\in\mathbb{R} \}.
\end{equation*}
Clearly,
\begin{align*}
{\rm CK}&= {\rm RM} \oplus {\rm span}\{(\boldsymbol{x}\cdot\boldsymbol{x})\boldsymbol{a}-2(\boldsymbol{a}\cdot\boldsymbol{x})\boldsymbol{x}+b\boldsymbol{x}: \boldsymbol{a}\in\mathbb{R}^3, b\in\mathbb{R} \} \\
&= {\rm RT} \oplus {\rm span}\{(\boldsymbol{x}\cdot\boldsymbol{x})\boldsymbol{a}-2(\boldsymbol{a}\cdot\boldsymbol{x})\boldsymbol{x}+\boldsymbol{c}\times\boldsymbol{x}: \boldsymbol{a}, \boldsymbol{c}\in\mathbb{R}^3\}.
\end{align*}

\begin{lemma}
We have
\begin{equation}\label{eq:CKprop1}
\dev\grad{\rm CK}=\mskw{\rm RM},\quad \defm{\rm CK}=\mathbb P_1 I,\quad \div{\rm CK}=\mathbb P_1,
\end{equation}
and
\begin{equation}\label{eq:CKprop2}
\curl{\rm CK}={\rm RM}, \quad {\rm CK}\cap\ker(\curl)={\rm RT}.
\end{equation}
\end{lemma}
\begin{proof}
By a direct computation, 
\begin{equation*}
\grad\big((\boldsymbol{x}\cdot\boldsymbol{x})\boldsymbol{a}-2(\boldsymbol{a}\cdot\boldsymbol{x})\boldsymbol{x}\big)=2\mskw(\boldsymbol{x}\times\boldsymbol{a})-2(\boldsymbol{a}\cdot\boldsymbol{x})I, 
\end{equation*}
\begin{equation*}
\grad(b\boldsymbol{x})=bI, \quad \grad(\boldsymbol{c}\times\boldsymbol{x})=\mskw\boldsymbol{c}.
\end{equation*}
Then \eqref{eq:CKprop1} holds.
By $\curl=2\vskw\grad=2\mskw^{-1}\skw\grad$, we have 
\begin{equation*}
\curl((\boldsymbol{x}\cdot\boldsymbol{x})\boldsymbol{a}-2(\boldsymbol{a}\cdot\boldsymbol{x})\boldsymbol{x})=4\boldsymbol{x}\times\boldsymbol{a},\quad \curl(\boldsymbol{c}\times\boldsymbol{x})=2\boldsymbol{c}.
\end{equation*}
Thus, \eqref{eq:CKprop2} is true.
\end{proof}

Combine the polynomial elasticity complex \cite[(2.6)]{ArnoldAwanouWinther2008}, the polynomial Hessian complex \cite[(3.4.9)]{ChenHuang2022d}, and the polynomial de Rham complex \cite[(3.3)]{ArnoldFalkWinther2006} to have the diagram
\begin{equation*}
\begin{tikzcd}
{\rm RM} \longrightarrow \mathbb P_{k+4}(\mathbb R^3)
\arrow{r}{\defm}
&
\mathbb P_{k+3}(\mathbb S)
 \arrow{r}{\inc}
 &
\mathbb P_{k+1}(\mathbb S)
 \arrow{r}{\div}
 & \mathbb P_{k}(\mathbb R^3) \longrightarrow 0
 \\
\mathbb{P}_1 \longrightarrow \mathbb P_{k+3}
 \arrow[ur,swap,"\iota"'] \arrow{r}{\hess}
 & 
\mathbb P_{k+1}(\mathbb S) 
 \arrow[ur,swap,"S"'] \arrow{r}{\curl}
 & 
\mathbb P_{k}(\mathbb T) 
 \arrow[ur,swap,"-2\vskw"'] \arrow{r}{\div} \arrow[r] 
 & \mathbb P_{k-1}(\mathbb R^3) \longrightarrow 0 
 \\
\mathbb{R} \longrightarrow \mathbb P_{k+1}
 \arrow[ur,swap,"\iota"'] \arrow{r}{\grad}
 & 
\mathbb P_{k}(\mathbb R^3) 
 \arrow[ur,swap,"\mskw"'] \arrow{r}{\curl}
 & 
\mathbb P_{k-1}(\mathbb R^3) 
 \arrow[ur,swap,"{\rm id}"'] \arrow{r}{\div} \arrow[r] 
 & \mathbb P_{k-2} \longrightarrow 0. 
\end{tikzcd}
\end{equation*}
Applying the BGG framework to this diagram yields the following exact polynomial conformal elasticity complex from the first two rows:
\begin{equation}\label{polyconformalElascomplex3d}
{\rm CK}\xrightarrow{\subset} \mathbb P_{k+4}(\mathbb R^3)\xrightarrow{\dev\defm} \mathbb P_{k+3}(\mathbb{S}\cap\mathbb{T})\xrightarrow{\cott} \mathbb P_{k}(\mathbb{S}\cap\mathbb{T}) \xrightarrow{\div} \mathbb P_{k-1}(\mathbb R^3)\xrightarrow{}0.
\end{equation}
The last two rows similarly yield the following exact polynomial conformal Hessian complex:
\begin{equation}\label{polyconformalHesscomplex3d}
{\rm CH}\xrightarrow{\subset} \mathbb P_{k+3}\xrightarrow{\dev\hess} \mathbb P_{k+1}(\mathbb{S}\cap\mathbb{T})\xrightarrow{\sym\curl} \mathbb P_{k}(\mathbb{S}\cap\mathbb{T}) \xrightarrow{\div{\div}} \mathbb P_{k-2}\xrightarrow{}0.
\end{equation}

Combine the polynomial de Rham complex and the polynomial  div\,div complex \cite[(25)]{ChenHuang2022} to have the diagram
\begin{equation*}
\begin{tikzcd}[column sep=1cm, row sep=normal]
\mathbb{R} \longrightarrow \mathbb P_{k+3} \arrow{r}{\grad}&
\mathbb P_{k+2}(\mathbb R^3)
  \arrow{r}{\curl}
 &
\mathbb P_{k+1}(\mathbb R^3)
\arrow{r}{\div}
& \mathbb P_{k} \longrightarrow 0 \\
{\rm RT} \longrightarrow \mathbb P_{k+2}(\mathbb R^3)
 \arrow[ur,swap,"{\rm id}"'] \arrow{r}{\dev\grad}
 & 
\mathbb P_{k+1}(\mathbb T)
 \arrow[ur,swap,"- 2\vskw"'] \arrow{r}{\sym\curl}
 & 
\mathbb P_k(\mathbb S)
\arrow[ur,swap,"\tr"']\arrow{r}{\div\div}
& \mathbb P_{k-2} \longrightarrow 0,
\end{tikzcd}
\end{equation*}
which also implies the polynomial conformal Hessian complex \eqref{polyconformalHesscomplex3d}. 


\subsection{Conformal complexes in three dimensions}
In this subsection, we review the conformal Hessian complex and the conformal elasticity complex; see, e.g., \cite{ArnoldHu2021,BeigChrusciel2020,Beig1997}.

Combine the de Rham complex \cite[(2.6)]{ArnoldFalkWinther2006} and the div\,div complex \cite{PaulyZulehner2020,ArnoldHu2021} to arrive at the diagram
\begin{equation*}
\begin{tikzcd}[column sep=1cm, row sep=normal]
\mathbb{R} \longrightarrow H^2 \arrow{r}{\grad}&
H^1(\mathbb R^3)
  \arrow{r}{\curl}
 &
H(\div)
\arrow{r}{\div}
& L^2 \longrightarrow 0 \\
{\rm RT} \longrightarrow H^1(\mathbb R^3)
 \arrow[ur,swap,"{\rm id}"'] \arrow{r}{\dev\grad}
 & 
H(\sym\curl;\mathbb T)
 \arrow[ur,swap,"- 2\vskw"'] \arrow{r}{\sym\curl}
 & 
H(\div\div;\mathbb S)
\arrow[ur,swap,"\tr"']\arrow{r}{\div\div}
& L^2 \longrightarrow 0.
\end{tikzcd}
\end{equation*}
Applying the BGG framework to this diagram will induce
the conformal Hessian complex 
\begin{equation}\label{conformalHesscomplex3d}
\resizebox{.935\hsize}{!}{$
{\rm CH}\xrightarrow{\subset} H^2\xrightarrow{\dev\hess} H(\sym\curl; \mathbb{S}\cap\mathbb{T})\xrightarrow{\sym\curl} H(\div\div; \mathbb{S}\cap\mathbb{T}) \xrightarrow{\div{\div}} L^2\xrightarrow{}0,
$}
\end{equation}
which is exact on bounded Lipschitz contractible domains \cite{ArnoldHu2021,GuoHuLin2025}.
The conformal Hessian complex \eqref{conformalHesscomplex3d} has applications in general relativity \cite{QuennevilleBelair2015,BeigChrusciel2020}.

Next we derive the following conformal elasticity complex using the BGG framework, and the tilde and hat operations in \cite{ChenHuang2025} 
\begin{equation}\label{conformalElascomplex3d}
{\rm CK}\xrightarrow{\subset} H^1(\mathbb R^3)\xrightarrow{\dev\defm} H(\cott; \mathbb{S}\cap\mathbb{T})\xrightarrow{\cott} H(\div; \mathbb{S}\cap\mathbb{T}) \xrightarrow{\div} L^2(\mathbb R^3)\xrightarrow{}0.
\end{equation}
The conformal elasticity complex \eqref{conformalElascomplex3d} has applications in general relativity \cite{Dain2006}, Cosserat elasticity \cite{NeffJeong2009} and fluid mechanics problems \cite{GopalakrishnanLedererSchoeberl2020,FeireislNovotny2017}.

\begin{lemma}\label{lem:reducedelasticityexact}
The elasticity complex 
\begin{equation}\label{elascomplex3dtildehat}
{\rm RM}\xrightarrow{\subset} H^{1,2}(\div)\xrightarrow{\defm} \widehat{\widetilde{H}}(\inc; \mathbb S)\xrightarrow{\inc} \widetilde{H}(\div; \mathbb S) \xrightarrow{\div} L^2(\mathbb R^3)\xrightarrow{}0
\end{equation}
is exact, where
\begin{align*}
H^{1,2}(\div)&:=\{\boldsymbol{v}\in H^{1}(\mathbb R^3): \div\boldsymbol{v}\in H^2\},\\
\widehat{\widetilde{H}}(\inc; \mathbb S)&:=\{\boldsymbol{\tau}\in H(\inc; \mathbb S): \tr\boldsymbol{\tau}\in H^2, \inc\boldsymbol{\tau}\in \widetilde{H}(\div; \mathbb S)\},\\
\widetilde{H}(\div; \mathbb S)&:=S H(\curl;\mathbb S).
\end{align*}
\end{lemma}
\begin{proof}
The tilde and hat operations on the standard elasticity complex give this result; Appendix~\ref{apdx:continuous-exactness} verifies the nonstandard intermediate spaces.
\end{proof}

\begin{lemma}\label{lem:continuousCEexact}
The conformal elasticity complex \eqref{conformalElascomplex3d} is exact.
\end{lemma}
\begin{proof}
The BGG framework~\cite{ArnoldHu2021} and Sobolev cohomology~\cite{CapHu2024} give abstract exactness; Appendix~\ref{apdx:continuous-exactness} verifies the reduction-space identifications in this diagram.
\end{proof}

\subsection{Traces and continuity}
We collect traces for operators $\sym\curl$, $\div\div$, $\inc$ and $\cott$ in \cite{ChenHuang2020,ChenHuang2022,ChenHuang2022b,HuLinShi2023}. 

For a smooth tensor $\boldsymbol{\tau}$, define the following trace operators on face $f$
\begin{itemize}
\item $\tr_1^{\sym\curl}(\boldsymbol{\tau}):=\sym(\Pi_f\boldsymbol{\tau}\times\boldsymbol{n})=\Pi_f\sym(\boldsymbol{\tau}\times\boldsymbol{n})\Pi_f$,
\smallskip
\item $\tr_2^{\sym\curl}(\boldsymbol{\tau}):=\boldsymbol n\cdot\boldsymbol \tau\times\boldsymbol n$,
\smallskip
\item $\tr_1^{\div\div}(\boldsymbol{\tau}):=\boldsymbol{n}\cdot\boldsymbol{\tau}\cdot\boldsymbol{n}$,
\smallskip
\item $\tr_2^{\div\div}(\boldsymbol{\tau}):=\boldsymbol n^{\intercal}\div \boldsymbol \tau +  \div_f(\boldsymbol\tau \boldsymbol n)$.
\end{itemize}
For a smooth symmetric tensor $\boldsymbol{\tau}$, define the following trace operators on face $f$
\begin{itemize}
\item $\tr_1^{\inc}(\boldsymbol{\tau}):=\boldsymbol{n}\times\boldsymbol{\tau}\times\boldsymbol{n}$,
\smallskip
\item $\tr_2^{\inc}(\boldsymbol{\tau}):=2\defm_f(\boldsymbol{n}\cdot\boldsymbol{\tau}\Pi_f)-\Pi_f\partial_n\boldsymbol{\tau}\Pi_f$,
\smallskip
\item $\tr_1^{\cott}(\boldsymbol{\tau}):=\tr_1^{\sym\curl}(\Pi_f\boldsymbol{\tau}\Pi_f)=\tr_1^{\sym\curl}(\boldsymbol{\tau})=\sym(\Pi_f\boldsymbol{\tau}\times\boldsymbol{n})$, 
\smallskip
\item $\tr_2^{\cott}(\boldsymbol{\tau}):=\tr_1^{\sym\curl}(\tr_2^{\inc}(\boldsymbol{\tau}))=\sym(\tr_2^{\inc}(\boldsymbol{\tau})\times\boldsymbol{n})$, 
\smallskip
\item $\tr_3^{\cott}(\boldsymbol{\tau}):=\dev_f(\tr_2^{\inc}(\sym\curl\boldsymbol{\tau}))$.
\end{itemize}
Since $\tr_f(\tr_2^{\inc}(\sym\curl\boldsymbol{\tau}))=\div_f\div_f\tr_1^{\cott}(\boldsymbol{\tau})$ for any smooth symmetric tensor $\boldsymbol{\tau}$~\cite[Lemma 4.1]{HuLinShi2023}, we have
\begin{equation}\label{eq:trtrcott3}
\tr_2^{\inc}(\sym\curl\boldsymbol{\tau})=\tr_3^{\cott}(\boldsymbol{\tau})+\frac{1}{2}\div_f\div_f\tr_1^{\cott}(\boldsymbol{\tau}) I_f. 
\end{equation}
The third Cotton trace here is the tangential deviatoric part of that in~\cite{HuLinShi2023}. By~\eqref{eq:trtrcott3}, the conformity conditions are equivalent when the first Cotton trace is single-valued with a fixed face orientation.

For a symmetric tensor $\boldsymbol{\tau}$, it follows \cite[Lemma 4.1]{ChenHuang2022b}
\begin{equation}\label{eq:20240319tr2inc}
\tr_2^{\inc}(\boldsymbol{\tau})=\boldsymbol{n}\times(\curl\boldsymbol{\tau})^{\intercal}\Pi_f+\grad_f(\Pi_f\boldsymbol{\tau}\boldsymbol{n})=-\Pi_f(\curl\boldsymbol{\tau})\times\boldsymbol{n}+\nabla_f(\boldsymbol{n}\cdot\boldsymbol{\tau}\Pi_f),
\end{equation}
which implies \cite[Lemma 4.2]{HuLinShi2023}
\begin{equation}\label{eq:20240319tr2cott}
\tr_2^{\cott}(\boldsymbol{\tau})=\tr_1^{\inc}(\sym\curl\boldsymbol{\tau})-\sym(\curl_f(\Pi_f\boldsymbol{\tau}\boldsymbol{n})).
\end{equation}
Here, $\curl_f(\Pi_f\boldsymbol{\tau}\boldsymbol{n})$ means the tensor $(\Pi_f\boldsymbol{\tau}\boldsymbol{n})\otimes \curl_f$.

For a smooth tensor $\boldsymbol{\tau}$ and a face $f$, define the following trace operators on edge $e$
\begin{itemize}
\item $\tr_1^{\div_f\div_f}(\boldsymbol{\tau}):=\boldsymbol{n}_{f,e}\cdot\boldsymbol{\tau}\cdot\boldsymbol{n}_{f,e}$,
\smallskip
\item $\tr_2^{\div_f\div_f}(\bs \tau):=\partial_{\boldsymbol{t}_e}(\boldsymbol{t}_e^{\intercal}\bs \tau\boldsymbol{n}_{f,e})+\boldsymbol{n}_{f,e}^{\intercal}\div_f\bs \tau$,
\smallskip
\item $\tr_{e}^{\sym\curl}(\boldsymbol{\tau}):=\boldsymbol{n}_2^{\intercal}(\curl\boldsymbol{\tau})\boldsymbol{n}_1+\partial_{\boldsymbol{t}_e}(\boldsymbol{t}_e^{\intercal}\boldsymbol{\tau}\boldsymbol{t}_e)$,
\smallskip
\item $\tr_{e,1}^{\cott}(\boldsymbol{\tau}):=\boldsymbol{t}_e^{\intercal}(\sym\curl\boldsymbol{\tau})\boldsymbol{t}_e=\boldsymbol{t}_e^{\intercal}(\curl\boldsymbol{\tau})\boldsymbol{t}_e$,
\smallskip
\item $\tr_{e,2}^{\cott}(\boldsymbol{\tau}):=2\partial_{\boldsymbol{t}_e}(\sym\curl\boldsymbol{\tau})\boldsymbol{t}_e-\nabla(\tr_{e,1}^{\cott}(\boldsymbol{\tau}))$,
\smallskip
\item $\tr_{e,3}^{\cott}(\boldsymbol{\tau}):=-\boldsymbol{t}_e\cdot\nabla\times(\sym\curl\boldsymbol{\tau})\cdot\boldsymbol{t}_e-\frac{1}{2}\partial_{\boldsymbol{t}_e}(\boldsymbol{t}_e\cdot\div\boldsymbol{\tau})$.
\end{itemize}
It follows from Lemma 6.1 in \cite{ChenHuang2022} that on edge $e$, 
\begin{align}
\notag
&\tr_2^{\div_f\div_f}(\tr_1^{\sym\curl}(\boldsymbol{\tau}))=\boldsymbol{n}_{f,e}^{\intercal}(\curl\boldsymbol{\tau})\boldsymbol{n}_f+\partial_{\boldsymbol{t}_{f,e}}(\boldsymbol{t}_e^{\intercal}\boldsymbol{\tau}\boldsymbol{t}_e),\\
\label{eq:edgedofprop1}
&\tr_{e}^{\sym\curl}(\boldsymbol{\tau})
=\boldsymbol{t}_e\cdot\boldsymbol{t}_{f,e}\tr_2^{\div_f\div_f}(\tr_1^{\sym\curl}(\boldsymbol{\tau})) \\
\notag
&\qquad\qquad\qquad\;\; + 2(\boldsymbol n_f\cdot\boldsymbol n_2)^2\boldsymbol n_{1}^{\intercal}(\sym\curl\boldsymbol \tau)\boldsymbol n_2 \\
\notag
&\qquad\qquad\qquad\;\;  +(\boldsymbol n_f\cdot\boldsymbol n_1) (\boldsymbol n_f\cdot\boldsymbol n_2) \left [\boldsymbol n_{1}^{\intercal}(\sym\curl\boldsymbol \tau)\boldsymbol n_1-\boldsymbol n_{2}^{\intercal}(\sym\curl\boldsymbol \tau)\boldsymbol n_2\right ]. 
\end{align}

For a symmetric tensor $\boldsymbol{\tau}$, by \eqref{eq:trtrcott3}-\eqref{eq:20240319tr2cott}, we obtain \cite[Lemma~4.4]{HuLinShi2023}
\begin{align}
\label{eq:20240328}
\boldsymbol{n}\times(\cott\boldsymbol{\tau})\boldsymbol{n}&=\rot_f(\boldsymbol{n}\times(\curl(\sym\curl\boldsymbol{\tau}))^{\intercal}\Pi_f)=\rot_f(\tr_2^{\inc}(\sym\curl\boldsymbol{\tau}))\\
\notag
&=\rot_f(\tr_3^{\cott}(\boldsymbol{\tau})+\frac{1}{2}\div_f\div_f\tr_1^{\cott}(\boldsymbol{\tau}) I_f),\\
\notag
\boldsymbol{n}^{\intercal}(\cott\boldsymbol{\tau})\boldsymbol{n}&=-\div_f\div_f\tr_1^{\inc}(\sym\curl\boldsymbol{\tau})=-\div_f\div_f\tr_2^{\cott}(\boldsymbol{\tau}).
\end{align}
These two identities motivate the following trace identities on edges.

\begin{lemma}  
For a symmetric tensor $\boldsymbol{\tau}$ and a face $f$,
we have on edge $e\in\partial f$ that
\begin{align}
\label{eq:202403201}
\tr_1^{\div_f\div_f}(\tr_2^{\cott}(\boldsymbol{\tau}))&=-\tr_{e,1}^{\cott}(\boldsymbol{\tau}) + \partial_{\boldsymbol{t}_{f,e}}(\boldsymbol{n}_{f,e}^{\intercal}\boldsymbol{\tau}\boldsymbol{n}), \\
\label{eq:202403202}
\tr_2^{\div_f\div_f}(\tr_2^{\cott}(\boldsymbol{\tau}))&=\boldsymbol{n}_{f,e}\cdot\tr_{e,2}^{\cott}(\boldsymbol{\tau}) + \partial_{\boldsymbol{t}_e}^2(\boldsymbol{t}_{f,e}^{\intercal}\boldsymbol{\tau}\boldsymbol{n}), \\
\label{eq:202403203}
\boldsymbol{t}_e\cdot\tr_3^{\cott}(\boldsymbol{\tau})\cdot\boldsymbol{t}_e&=\boldsymbol{n}_{f}\cdot\tr_{e,2}^{\cott}(\boldsymbol{\tau}) - \frac{1}{2}\div_f\div_f\tr_1^{\cott}(\boldsymbol{\tau}), \\
\label{eq:202403204}
\boldsymbol{n}_{f,e}\cdot\tr_3^{\cott}(\boldsymbol{\tau})\cdot\boldsymbol{t}_{f,e}&=\tr_{e,3}^{\cott}(\boldsymbol{\tau})+\frac{1}{2}\partial_{\boldsymbol{t}_e}^2(\boldsymbol{n}^{\intercal}\boldsymbol{\tau}\boldsymbol{n} + 2\,\boldsymbol{t}_{f,e}\cdot\tr_1^{\cott}(\boldsymbol{\tau})\cdot\boldsymbol{n}_{f,e}) \\
\notag
&\quad - \partial_{\boldsymbol{t}_{f,e}}(\partial_{\boldsymbol{n}_{f,e}}(\boldsymbol{t}_{e}\cdot\tr_1^{\cott}(\boldsymbol{\tau})\cdot\boldsymbol{t}_{e})), \\
\label{eq:nfecottntr3cott}
\boldsymbol{n}_{f,e}^{\intercal}(\cott\boldsymbol{\tau})\boldsymbol{n}&=\boldsymbol{t}_{f,e}^{\intercal}\rot_f(\tr_3^{\cott}(\boldsymbol{\tau})+\frac{1}{2}\div_f\div_f\tr_1^{\cott}(\boldsymbol{\tau}) I_f).
\end{align}
\end{lemma}
\begin{proof}
Let $\boldsymbol{\sigma}=\sym\curl\boldsymbol{\tau}$ and $\boldsymbol{\sigma}_f=\sym\curl_f(\Pi_f\boldsymbol{\tau}\boldsymbol{n})$ for simplicity.
First, the identity~\eqref{eq:202403201} follows from \eqref{eq:20240319tr2cott} and $\boldsymbol{n}_{f,e}^{\intercal}\boldsymbol{\sigma}_f\boldsymbol{n}_{f,e}=-\partial_{\boldsymbol{t}_{f,e}}(\boldsymbol{n}_{f,e}^{\intercal}\boldsymbol{\tau}\boldsymbol{n})$.
Since
\begin{align*}
\tr_2^{\div_f\div_f}(\tr_1^{\inc}(\boldsymbol{\sigma}))&=\partial_{\boldsymbol{t}_e}(\boldsymbol{n}_{f,e}\cdot\boldsymbol{\sigma}\cdot\boldsymbol{t}_{e})-(\boldsymbol{n}\times\nabla)\cdot\boldsymbol{\sigma}\cdot\boldsymbol{t}_{f,e}=\boldsymbol{n}_{f,e}\cdot\tr_{e,2}^{\cott}(\boldsymbol{\tau}), \\
\tr_2^{\div_f\div_f}(\boldsymbol{\sigma}_f) & = \partial_{\boldsymbol{t}_e}(\boldsymbol{t}_e\cdot\boldsymbol{\sigma}_f\cdot\boldsymbol{n}_{f,e}) -\frac{1}{2}\partial_{\boldsymbol{t}_{f,e}}(\div_f(\boldsymbol{\tau}\boldsymbol{n}))  = -\partial_{\boldsymbol{t}_e}^2(\boldsymbol{t}_{f,e}^{\intercal}\boldsymbol{\tau}\boldsymbol{n}),
\end{align*}
we acquire \eqref{eq:202403202} from \eqref{eq:20240319tr2cott}.

By \eqref{eq:20240319tr2inc},
\begin{equation*}
\tr_2^{\inc}(\boldsymbol{\sigma})\cdot\boldsymbol{t}_{f,e} = \boldsymbol{n}\times(\nabla\times\boldsymbol{\sigma})\boldsymbol{t}_{f,e}+\partial_{\boldsymbol{t}_{f,e}}(\Pi_f\boldsymbol{\sigma}\boldsymbol{n}).
\end{equation*}
Then
\begin{equation*}
\boldsymbol{t}_{e}\cdot\tr_2^{\inc}(\boldsymbol{\sigma})\cdot\boldsymbol{t}_{e} = \boldsymbol{n}_{f,e}^{\intercal}(\nabla\times\boldsymbol{\sigma})\boldsymbol{t}_{f,e}+\partial_{\boldsymbol{t}_{e}}(\boldsymbol{n}^{\intercal}\boldsymbol{\sigma}\boldsymbol{t}_{e})=\boldsymbol{n}_{f}\cdot\tr_{e,2}^{\cott}(\boldsymbol{\tau}),
\end{equation*}
\begin{equation*}
\boldsymbol{n}_{f,e}\cdot\tr_2^{\inc}(\boldsymbol{\sigma})\cdot\boldsymbol{t}_{f,e} = -\boldsymbol{t}_{e}^{\intercal}(\nabla\times\boldsymbol{\sigma})\boldsymbol{t}_{e}+\partial_{\boldsymbol{t}_{f,e}}(\boldsymbol{n}^{\intercal}\boldsymbol{\sigma}\boldsymbol{n}_{f,e}).
\end{equation*}
Hence, \eqref{eq:202403203} follows from \eqref{eq:trtrcott3}, and 
\begin{equation*}
\boldsymbol{n}_{f,e}\cdot\tr_3^{\cott}(\boldsymbol{\tau})\cdot\boldsymbol{t}_{f,e} =\tr_{e,3}^{\cott}(\boldsymbol{\tau})+\frac{1}{2}\partial_{\boldsymbol{t}_e}(\boldsymbol{t}_e\cdot\div\boldsymbol{\tau})+\partial_{\boldsymbol{t}_{f,e}}(\boldsymbol{n}^{\intercal}\boldsymbol{\sigma}\boldsymbol{n}_{f,e}).
\end{equation*}
By a direct computation,
\begin{align*}
\frac{1}{2}\partial_{\boldsymbol{t}_e}(\boldsymbol{t}_e\cdot\div\boldsymbol{\tau})+\partial_{\boldsymbol{t}_{f,e}}(\boldsymbol{n}^{\intercal}\boldsymbol{\sigma}\boldsymbol{n}_{f,e})&=\frac{1}{2}\partial_{\boldsymbol{t}_e}^2(\boldsymbol{n}^{\intercal}\boldsymbol{\tau}\boldsymbol{n}+\boldsymbol{t}_{f,e}^{\intercal}\boldsymbol{\tau}\boldsymbol{t}_{f,e}-\boldsymbol{n}_{f,e}^{\intercal}\boldsymbol{\tau}\boldsymbol{n}_{f,e}) \\
&\quad +\partial_{\boldsymbol{t}_{f,e}}(\partial_{\boldsymbol{n}_{f,e}}(\boldsymbol{n}_{f,e}^{\intercal}\boldsymbol{\tau}\boldsymbol{t}_{f,e})).
\end{align*}
Therefore,
the identity \eqref{eq:202403204} follows from $2\,\boldsymbol{t}_{f,e}\cdot\tr_1^{\cott}(\boldsymbol{\tau})\cdot\boldsymbol{n}_{f,e}=\boldsymbol{t}_{f,e}^{\intercal}\boldsymbol{\tau}\boldsymbol{t}_{f,e}-\boldsymbol{n}_{f,e}^{\intercal}\boldsymbol{\tau}\boldsymbol{n}_{f,e}$ and $\boldsymbol{t}_{e}\cdot\tr_1^{\cott}(\boldsymbol{\tau})\cdot\boldsymbol{t}_{e}=-\boldsymbol{n}_{f,e}^{\intercal}\boldsymbol{\tau}\boldsymbol{t}_{f,e}$. Finally, the identity \eqref{eq:nfecottntr3cott} holds from \eqref{eq:20240328}.
\end{proof}

The identities \eqref{eq:202403201}-\eqref{eq:202403204} have been proved in \cite[Lemma~7.7]{HuLinShi2023} for $\boldsymbol{\tau}$ satisfying $\boldsymbol{\tau}|_e=0$ and $\tr_1^{\cott}(\boldsymbol{\tau})=0$. Here, they hold for general symmetric tensors.

Next we consider the continuity of a piecewise smooth tensor function to be in the Sobolev space $H(\div{\div },\Omega; \mathbb{S}\cap \mathbb{T})$.

The trace $\tr^{\div\div}\bs \sigma$, as a distribution, is defined as the difference 
\begin{equation*}
\langle \tr^{\div\div}\bs \sigma, \tr^{\dev\hess} v \rangle_{\partial T}:= (\div\div\boldsymbol \sigma, v)_T - (\boldsymbol \sigma, \dev\hess v)_T.
\end{equation*}
We decompose $\tr^{\div\div}\bs \sigma$ and $\tr^{\dev\hess} v$ into two face-wise trace operators and one edge trace operator.

Applying integration by parts as in Lemma 5.2 of~\cite{ChenHuang2022a}, we obtain the following Green's identity.
\begin{lemma}
For a polytope $T$,
we have for any $\boldsymbol \sigma\in H^2(T; \mathbb S\cap \mathbb T)$ and $v\in H^2(T)$ that
\begin{align}
\label{eq:greenidentitydivdiv} 
 (\div\div\boldsymbol \sigma, v)_T&=(\boldsymbol \sigma, \dev\hess v)_T -\sum_{e\in\Delta_{1}(T)}(\tr_e^{\div\div}(\bs \sigma), v)_e \\
\notag 
& \quad - \sum_{f\in\partial T}\left[( \tr_1^{\div\div}(\bs \sigma), \partial_nv)_{f} -  ( \tr_2^{\div\div}(\bs \sigma), v)_f\right] , 
\end{align}
where $\boldsymbol n=\boldsymbol n_{\partial T}$ on each face and $\tr_e^{\div\div}(\bs \sigma) = \sum_{f\in\partial T,e\in \partial f}\boldsymbol n_{f,e}^{\intercal}\boldsymbol \sigma \boldsymbol n_{\partial T}$.
\end{lemma}


Applying the same argument as Proposition 3.6 in~\cite{FuehrerHeuerNiemi2019}, we can present an equivalent condition for a piecewise smooth tensor belonging to $H(\div\div,\Omega;\mathbb S\cap\mathbb T)$.
\begin{lemma}\label{lm:divdivconforming}
Let $\bs \sigma \in L^2(\Omega;\mathbb S\cap\mathbb T)$ and $\bs \sigma|_T\in H^{2}(T;\mathbb S\cap\mathbb T)$ for each $T\in \mathcal T_h$. Then $\bs \sigma \in H(\div\div,\Omega;\mathbb S\cap\mathbb T)$ if and only if 
\begin{enumerate}
\item $[\tr_1^{\div\div}(\bs \sigma)]_f = 0$ and $[\tr_2^{\div\div}(\bs \sigma)]_f = 0$ for all $f\in \  \Delta_{2}(\mathring{\mathcal{T}}_h)$;
\smallskip
\item $[ \tr_e^{\div\div}(\bs \sigma)]|_e = 0$ for all $e\in \ \Delta_{1}(\mathring{\mathcal{T}}_h)$. 
\end{enumerate} 
Here $[v]|_f$ means the jump of $v$ across $f$, and
\begin{equation*}
[ \tr_e^{\div\div}(\bs \sigma)]|_e := \sum_{T\in \omega_e} \sum_{f\in\partial T, e\in\partial f} (\boldsymbol n_{f,e}^{\intercal}\boldsymbol \sigma \boldsymbol n_{\partial T})|_e.
\end{equation*}
Here, $\omega_e = \{T\in \mathcal T_h: e\in\Delta_1(T)\}$ is the set of all simplices containing $e$.
\end{lemma}

We recall the following $H(\cott)$-conformity criterion.
\begin{lemma}[Theorem 4.2 in \cite{HuLinShi2023}]\label{lm:cottconforming}
Let $\bs \sigma \in L^2(\Omega;\mathbb S\cap\mathbb T)$ be polynomial on each $T\in \mathcal T_h$. Assume
\begin{enumerate}
\item $[\tr_1^{\cott}(\bs \sigma)]_f = 0$, $[\tr_2^{\cott}(\bs \sigma)]_f = 0$, and $[\tr_3^{\cott}(\bs \sigma)]_f = 0$ for all $f\in \  \Delta_{2}(\mathring{\mathcal{T}}_h)$;
\smallskip
\item $\bs \sigma|_e$ is continuous for all $e\in \ \Delta_{1}(\mathring{\mathcal{T}}_h)$. 
\end{enumerate} 
Then $\bs \sigma \in H(\cott,\Omega;\mathbb S\cap\mathbb T)$.
\end{lemma}

\section{Bubble polynomial complexes}\label{sec:bubblecomplex}

In this section, we will apply the BGG framework to derive the bubble conformal Hessian complex 
\begin{align}\label{polybubbleCHcomplex3d}
0\xrightarrow{\subset} \mathbb B_{k+3}(\boldsymbol{r}_0)&\xrightarrow{\dev\hess} \mathbb B_{k+1}^{\sym\curl}(\boldsymbol{r}_1;\mathbb S\cap\mathbb T)\xrightarrow{\sym\curl} \mathbb B_{k}^{\div\div}(\boldsymbol{r}_2;\mathbb S\cap\mathbb T) \\
\notag
&\xrightarrow{\div{\div}} \mathbb B_{k-2}(\boldsymbol{r}_3)/{\rm CH}\xrightarrow{}0
\end{align}
with $k\geq 2r_2^{\texttt{v}}+4$ and smoothness vectors
\begin{equation}\label{eq:rCHbubble}
\boldsymbol r_1 \geq (1,0,-1)^{\intercal},
\quad
\boldsymbol{r}_0= \boldsymbol r_1+2, \quad
 \boldsymbol{r}_2=\boldsymbol{r}_1\ominus1,
\quad
 \boldsymbol{r}_3=\boldsymbol{r}_2\ominus2,
\end{equation}
and the bubble conformal elasticity complex
\begin{align}\label{polybubbleCEcomplex3d}
0\xrightarrow{\subset}\mathbb B_{k+4}(\boldsymbol{r}_0;\mathbb R^3) &\xrightarrow{\dev\defm} \mathbb B_{k+3}^{\cott}(\boldsymbol{r}_1;\mathbb{S}\cap\mathbb{T}) \xrightarrow{\cott}  \mathbb B^{\div}_{k}(\boldsymbol{r}_2; \mathbb S\cap \mathbb T) \\
&\xrightarrow{\div} \mathbb B_{k-1}(\boldsymbol{r}_3;\mathbb R^3)/{\rm CK}\rightarrow 0 \notag
\end{align}
with $k\geq 2r_2^{\texttt{v}}+2$ and smoothness vectors 
\begin{equation}\label{eq:rCE}
\boldsymbol{r}_0\geq (7,1,0)^{\intercal},  \quad 
\boldsymbol{r}_1=\boldsymbol{r}_0-1\geq (6,0,-1)^{\intercal},\quad \boldsymbol{r}_2=\boldsymbol{r}_1\ominus3,\quad \boldsymbol{r}_3=\boldsymbol{r}_2\ominus1.
\end{equation}
The bubble spaces involved in complexes \eqref{polybubbleCHcomplex3d} and \eqref{polybubbleCEcomplex3d}
are defined in Section~\ref{subsec:bubblespaces}.

\subsection{Bubble polynomial complexes in two dimensions}
For an edge $e$, let $r^{\texttt{v}}\geq -1$ and $k\geq 2r^{\texttt{v}}+1$, and define the edge bubble polynomial space
\begin{equation*}
\mathbb B_k(e; r^{\texttt{v}}) : = \{ u\in\mathbb P_k(e): \, \partial_t^ju \text{ vanishes at all vertices of } e \text{ for } j=0,\ldots, r^{\texttt{v}} \},
\end{equation*}
where $\partial_t$ is the tangential derivative along $e$. This bubble space can be easily characterized as
$\mathbb B_{k}(e; r^{\texttt{v}}) =  b_e^{r^{\texttt{v}}+1} \mathbb P_{k-2(r^{\texttt{v}}+1)}(e),$ where $b_e\in \mathbb P_2(e)$ vanishes at two vertices of $e$. 

For a triangle $f$ and a smoothness vector $\boldsymbol r= (r^{\texttt{v}}, r^e)^{\intercal}$, define face bubble polynomial space
\begin{align*}
\mathbb B_{k}(f; (r^{\texttt{v}},r^e)^{\intercal})&:=\{u\in\mathbb P_k(f): \nabla_f^ju \textrm{ vanishes at all vertices of $f$ for $j=0,\ldots, r^{\texttt{v}}$}, \\
&\qquad\qquad\qquad \textrm{ and $\nabla_f^ju$ vanishes on all edges of $f$ for $j=0,\ldots, r^{e}$}\},
\end{align*}
where $\nabla_f$ is the surface gradient on $f$. All polynomials defined on $e$ and $f$ can be naturally extended to the whole tetrahedron using the Bernstein representation in barycentric coordinates.

For $f\in\Delta_2(T)$ and a smoothness vector $\boldsymbol{r}$,
introduce the following bubble spaces on face $f$:
\begin{align*}
\mathbb B^{\div_f}_{k}(\boldsymbol{r}; \mathbb R^2):={}&\{\boldsymbol{\tau}\in \mathbb B_{k}(f;\boldsymbol{r})\otimes \mathbb R^2: (\boldsymbol{\tau}\cdot\boldsymbol{n}_{f,e})|_e=0,\ e\in\Delta_1(f)\},\\
\mathbb B^{\rot_f}_{k}(\boldsymbol{r}; \mathbb R^2):={}&\{\boldsymbol{\tau}\in \mathbb B_{k}(f;\boldsymbol{r})\otimes \mathbb R^2: (\boldsymbol{\tau}\cdot\boldsymbol{t}_e)|_e=0,\ e\in\Delta_1(f)\},\\
\mathbb B^{\div_f\div_f}_{k}(\boldsymbol{r}; \mathbb S\cap\mathbb T):={}&\{\boldsymbol{\tau}\in \mathbb B_{k}(f;\boldsymbol{r})\otimes (\mathbb S(f)\cap\mathbb T(f)): (\boldsymbol n_{f,e}^{\intercal}\boldsymbol \tau\boldsymbol n_{f,e})|_{\partial f}=0, \\
&\quad\quad\quad\;\;\tr_2^{\div_f\div_f}(\bs \tau)=0,\; \boldsymbol \tau(\texttt{v})=0 \textrm{ for } \texttt{v}\in\Delta_{0}(f)\},
\end{align*}
where
\begin{equation*}
\mathbb S(f):=\textrm{span}\{\boldsymbol{t}_1\boldsymbol{t}_1^{\intercal}, \boldsymbol{t}_2\boldsymbol{t}_2^{\intercal}, \sym(\boldsymbol{t}_1\otimes\boldsymbol{t}_2)\},
\quad \mathbb T(f):=\textrm{span}\{\boldsymbol{t}_1\boldsymbol{t}_2^{\intercal}, \boldsymbol{t}_2\boldsymbol{t}_1^{\intercal}, \boldsymbol{t}_1\boldsymbol{t}_1^{\intercal}-\boldsymbol{t}_2\boldsymbol{t}_2^{\intercal}\}.
\end{equation*}

We first recall the bubble de Rham complex \eqref{eq:femderhambubblecomplex2d} proved in \cite[Lemma A.1]{ChenHuang2025}.
\begin{lemma}
Let smoothness vectors $\boldsymbol{r}_1\geq-1$ and $\boldsymbol{r}_2=\boldsymbol{r}_1\ominus1$.
 Assume also that $\boldsymbol r_1+1$ is a smoothness vector. 
Let $f\in\Delta_2(T)$ and $k\geq\max\{2r_1^{\texttt{v}}+1,0\}$.
Then the bubble de Rham complex
\begin{equation}\label{eq:femderhambubblecomplex2d}
0\xrightarrow{\subset}\mathbb B_{k+2}(f;\boldsymbol{r}_1+1)\xrightarrow{\curl_f}\mathbb B^{\div_f}_{k+1}(\boldsymbol{r}_1;\mathbb R^2)\xrightarrow{\div_f}\mathbb B_{k}(f;\boldsymbol{r}_2)/\mathbb R\to0
\end{equation}
is exact.
\end{lemma}

\begin{lemma}
Let $\boldsymbol{r}_1=(1,-1)^{\intercal}$. Let $f\in\Delta_2(T)$ and $k\geq4$.
The space $\mathbb P_{k+1}(f)\otimes(\mathbb S(f)\cap\mathbb T(f))$ is uniquely determined by the degrees of freedom (DoFs)
\begin{subequations}\label{eq:divdivST2D}
\begin{align}
\label{eq:divdivST2D0}
\nabla^j\boldsymbol{\tau} (\texttt{v}), & \quad j=0,1, \texttt{v}\in \Delta_0(f), \\
\label{eq:divdivST2D1}
\int_e \tr_1^{\div_f\div_f}(\boldsymbol{\tau}) \, q \dd s, & \quad q \in \mathbb P_{k-3}(e), e\in \Delta_1(f), \\
\label{eq:divdivST2D2}
\int_e \tr_2^{\div_f\div_f}(\boldsymbol{\tau}) \, q \dd s, & \quad q \in \mathbb P_{k-2}(e), e\in \Delta_1(f), \\
\label{eq:divdivST2D3}
\int_f \div_f\div_f\boldsymbol{\tau} \, q \dd S, & \quad q \in \mathbb P_{k-1}(f)/{\rm CH}(f), \\
\label{eq:divdivST2D4}
\int_f \boldsymbol{\tau} : \sym\curl_f\grad_fq \dd S, & \quad q \in \mathbb B_{k+3}(f;\boldsymbol r_1+2),
\end{align}
\end{subequations}
where
$
{\rm CH}(f):=\{(\boldsymbol{v}\cdot\boldsymbol{n})|_f: \boldsymbol{v}\in{\rm CK}\}=\mathbb P_1(f)\oplus {\rm span}\{|\Pi_f\boldsymbol{x}|^2\}$.
Then we have
\begin{equation}\label{eq:dimdivdivSTbubble2D}
\dim \mathbb B^{\div_f\div_f}_{k+1}(\boldsymbol r_1;\mathbb S\cap\mathbb T)= k^2 -k-3.
\end{equation}
\end{lemma}
\begin{proof}
The number of DoFs in \eqref{eq:divdivST2D} is
\[
18+3(2k-3)+\frac{1}{2}k(k+1)-4+\frac{1}{2}(k-1)(k-2)=k^2+5k+6,
\]
which equals the dimension of $\mathbb P_{k+1}(f)\otimes(\mathbb S(f)\cap\mathbb T(f))$.

Assume $\boldsymbol{\tau}\in \mathbb P_{k+1}(f)\otimes(\mathbb S(f)\cap\mathbb T(f))$ such that all the DoFs in \eqref{eq:divdivST2D} vanish.
The vanishing DoFs in \eqref{eq:divdivST2D0}-\eqref{eq:divdivST2D2} imply $\boldsymbol{\tau}\in\mathbb B^{\div_f\div_f}_{k+1}(\boldsymbol r_1;\mathbb S\cap\mathbb T)$. This together with the vanishing DoFs in \eqref{eq:divdivST2D3} yields $\div_f\div_f\boldsymbol{\tau}=0$. By the exactness of the bubble divdiv complex on faces in \cite[page 1124]{ChenHuang2025}, there exists a $\boldsymbol{\phi}\in \mathbb B_{k+2}(f; \boldsymbol r_1+1)$ such that $\boldsymbol{\tau}=\sym\curl_f\boldsymbol{\phi}$.
Since $\boldsymbol{\tau}$ is traceless, the rotation of $\boldsymbol{\phi}$ vanishes. Then $\boldsymbol{\phi}=\grad_f w$ for some $w\in \mathbb B_{k+3}(f;\boldsymbol r_1+2)$, and $\boldsymbol{\tau}=\sym\curl_f\grad_f w$. The vanishing DoFs in \eqref{eq:divdivST2D4} imply $\boldsymbol{\tau}=0$.

By the unisolvence of the DoFs in \eqref{eq:divdivST2D} for $\mathbb P_{k+1}(f)\otimes(\mathbb S(f)\cap\mathbb T(f))$, we have
\begin{equation*}
\dim \mathbb B^{\div_f\div_f}_{k+1}(\boldsymbol r_1;\mathbb S\cap\mathbb T)=2\dim \mathbb P_{k+1}(f)-18-3(2k-3)=k^2-k-3.
\end{equation*}
Thus, \eqref{eq:dimdivdivSTbubble2D} follows.
\end{proof}

With the aid of the bubble de Rham complex \eqref{eq:femderhambubblecomplex2d}, we derive the bubble conformal div\,div complex \eqref{eq:femCdivdivbubblecomplex2d}.
\begin{lemma}
Let smoothness vectors $\boldsymbol{r}_0\geq(2,1)^{\intercal}$, $\boldsymbol r_1=\boldsymbol r_0-2\geq(0,-1)^{\intercal}$ and $\boldsymbol{r}_2=\boldsymbol{r}_1\ominus1$. Let $f\in\Delta_2(T)$ and $k\geq\max\{2r_2^{\texttt{v}}+4,4\}$.
Then the bubble conformal div\,div complex 
\begin{equation}\label{eq:femCdivdivbubblecomplex2d}
\resizebox{.925\hsize}{!}{$
\mathbb B_{k+3}(f;\boldsymbol{r}_0)\xrightarrow{\sym\curl_f\grad_f}\mathbb B^{\div_f\div_f}_{k+1}(\boldsymbol{r}_1;\mathbb S\cap \mathbb T)\xrightarrow{\div_f\div_f}\mathbb B_{k-1}(f;\boldsymbol{r}_2\ominus1)/{\rm CH}(f)\to0
$}
\end{equation}
is exact.
\end{lemma}
\begin{proof}
It is easy to verify that \eqref{eq:femCdivdivbubblecomplex2d} is a complex, and $\mathbb B^{\div_f\div_f}_{k+1}(\boldsymbol{r}_1;\mathbb S\cap \mathbb T)\cap\ker(\div_f\div_f)=\sym\curl_f\grad_f\mathbb B_{k+3}(f;\boldsymbol{r}_0)$. 
We will finish the proof of the exactness of complex \eqref{eq:femCdivdivbubblecomplex2d} by checking the dimension identity 
\begin{equation}\label{dimdivdivconformalbubble2d}
\dim\mathbb B^{\div_f\div_f}_{k+1}(\boldsymbol{r}_1;\mathbb S\cap \mathbb T)-\dim\mathbb B_{k+3}(f;\boldsymbol{r}_0)=\dim\mathbb B_{k-1}(f;\boldsymbol{r}_2\ominus1)-4.
\end{equation}

\step 1  First consider case $\boldsymbol{r}_1\geq(2,0)^{\intercal}$.
Since $\boldsymbol{r}_1\geq(2,0)^{\intercal}$, we have
\begin{equation*}
\mathbb B^{\div_f\div_f}_{k+1}(\boldsymbol{r}_1;\mathbb S\cap \mathbb T)=\{\boldsymbol{\tau}\in\mathbb B_{k+1}(\boldsymbol{r}_1;\mathbb S\cap \mathbb T):\partial_n(\boldsymbol{n}^{\intercal}\boldsymbol{\tau}\boldsymbol{n})|_{\partial f}=0\}.
\end{equation*}
We justify that the boundary constraints subtracted below are independent. Only the case $r_1^e=0$, equivalently $r_2^e=-1$, contributes such constraints. Consider the trace map
\[
\begin{aligned}
\mathcal T_f:\mathbb B_{k+1}(f;\boldsymbol r_1)\otimes(\mathbb S(f)\cap\mathbb T(f))
&\longrightarrow\bigoplus_{e\in\Delta_1(f)}\mathbb B_k(e;r_2^{\texttt{v}}),\\
\mathcal T_f\boldsymbol\tau
&=\big(\partial_{\boldsymbol n_{f,e}}(\boldsymbol n_{f,e}^{\intercal}\boldsymbol\tau\boldsymbol n_{f,e})|_e\big)_e .
\end{aligned}
\]
This map is onto. Indeed, for $q\in\mathbb B_k(e;r_2^{\texttt{v}})$, let $\lambda_{e^*}$ be the barycentric coordinate vanishing on $e$. Then
\[
\boldsymbol L_eq:=
\frac{\lambda_{e^*}q}{\partial_{\boldsymbol n_{f,e}}\lambda_{e^*}}
\big(\boldsymbol n_{f,e}\boldsymbol n_{f,e}^{\intercal}
-\boldsymbol t_e\boldsymbol t_e^{\intercal}\big)
\in\mathbb B_{k+1}(f;\boldsymbol r_1)\otimes(\mathbb S(f)\cap\mathbb T(f))
\]
satisfies $\mathcal T_f\boldsymbol L_eq=(0,\ldots,q,\ldots,0)$. Here the endpoint vanishing of $q$, together with $r_1^{\texttt{v}}\geq2$, makes the traces on the other two edges zero. Thus $\mathcal T_f$ has full row rank, and the codimension is the sum of the three edge-space dimensions.
This means
\begin{equation}\label{eq:20240317}
\dim\mathbb B^{\div_f\div_f}_{k+1}(\boldsymbol{r}_1;\mathbb S\cap \mathbb T)=2\dim\mathbb B_{k+1}(f;\boldsymbol{r}_1)-3[r_2^e=-1]\dim\mathbb B_{k}(e;  r_2^{\texttt{v}}).
\end{equation}
On the other side, apply bubble complex \eqref{eq:femderhambubblecomplex2d} twice to get
\begin{align*}
\dim\mathbb B_{k+3}(f;\boldsymbol{r}_0)&=2\dim\mathbb B_{k+2}(f;\boldsymbol{r}_0-1)-\dim\mathbb B_{k+1}(f;\boldsymbol{r}_1)+1 \\
&=3\dim\mathbb B_{k+1}(f;\boldsymbol{r}_1)-2\dim\mathbb B_{k}(f;\boldsymbol{r}_2)+3.
\end{align*}
Combining the last two equations yields
\begin{equation*}
\dim\mathbb B^{\div_f\div_f}_{k+1}(\boldsymbol{r}_1;\mathbb S\cap \mathbb T)-\dim\mathbb B_{k+3}(f;\boldsymbol{r}_0)=\dim\mathbb B^{\div_f}_{k}(f;\boldsymbol{r}_2)-\dim\mathbb B_{k+1}(f;\boldsymbol{r}_1)-3.
\end{equation*}
Employ bubble complex \eqref{eq:femderhambubblecomplex2d} again to acquire \eqref{dimdivdivconformalbubble2d}.

\step 2 Then consider case $\boldsymbol{r}_1\geq(2,-1)^{\intercal}$ with $r_1^e=-1$.
As $(\boldsymbol{r}_1)_+\geq(2,0)^{\intercal}$, we have $\div_f\div_f\mathbb B^{\div_f\div_f}_{k+1}((\boldsymbol{r}_1)_+;\mathbb S\cap \mathbb T)=\mathbb B_{k-1}(f;\boldsymbol{r}_2\ominus1)/{\rm CH}(f)$. Noting that 
\begin{equation*}
\mathbb B^{\div_f\div_f}_{k+1}((\boldsymbol{r}_1)_+;\mathbb S\cap \mathbb T)\subseteq \mathbb B^{\div_f\div_f}_{k+1}(\boldsymbol{r}_1;\mathbb S\cap \mathbb T),
\end{equation*}
it holds $\div_f\div_f\mathbb B^{\div_f\div_f}_{k+1}(\boldsymbol{r}_1;\mathbb S\cap \mathbb T)=\mathbb B_{k-1}(f;\boldsymbol{r}_2\ominus1)/{\rm CH}(f)$.

\step 3 Next consider case $\boldsymbol{r}_1=(1,-1)^{\intercal}$. In this case, $\boldsymbol{r}_0=(3,1)^{\intercal}$ and $\boldsymbol{r}_2\ominus1=-1$.
By \eqref{eq:dimdivdivSTbubble2D},
we have
\begin{align*}
\dim\mathbb B_{k+3}(f;\boldsymbol{r}_0)&=\dim\mathbb P_{k+3}(f)-6k-9=\frac{1}{2}(k^2-3k+2), \\
\dim\mathbb B^{\div_f\div_f}_{k+1}(\boldsymbol{r}_1;\mathbb S\cap \mathbb T)&=k^2-k-3.
\end{align*}
Hence, \eqref{dimdivdivconformalbubble2d} is true.

\step 4 Finally, consider case $\boldsymbol{r}_1=(0,-1)^{\intercal}$.
Following the argument in Step 2, we can use the result in Step 3 to arrive at $\div_f\div_f\mathbb B^{\div_f\div_f}_{k+1}(\boldsymbol{r}_1;\mathbb S\cap \mathbb T)=\mathbb B_{k-1}(f;\boldsymbol{r}_2\ominus1)/{\rm CH}(f)$.
\end{proof}

\subsection{Bubble functions and smooth finite elements in three dimensions}\label{subsec:bubblespaces}
For a tetrahedron $T$ with vertices $\texttt{v}_0,\ldots,\texttt{v}_3$ and barycentric coordinates $\lambda_0,\ldots,\lambda_3$, set $b_T:=\lambda_0\lambda_1\lambda_2\lambda_3$.
For tetrahedron $T$ and a smoothness vector $\boldsymbol r= (r^{\texttt{v}}, r^e, r^f)^{\intercal}$, define bubble polynomial space
\begin{align*}
\mathbb B_{k}(T;\boldsymbol r):=\{u\in\mathbb P_k(T):&\, \nabla^ju \textrm{ vanishes at all vertices of $T$ for $j=0,\ldots, r^{\texttt{v}}$}, \\
& \textrm{ $\nabla^ju$ vanishes on all edges of $T$ for $j=0,\ldots, r^{e}$}, \\
& \textrm{ and $\nabla^ju$ vanishes on all faces of $T$ for $j=0,\ldots, r^{f}$}\}.
\end{align*}
Notice that when $r^f = -1$, the bubble function may not vanish on the boundary of $T$. 
Precise characterization of bubble polynomial spaces $\mathbb B_{k}(f;\boldsymbol r)$ and $\mathbb B_{k}(T;\boldsymbol r)$ can be given by decompositions of simplicial lattice points; see~\cite{ChenHuang2024a,ChenHuang2024} for details. 

To simplify notation, for a three-dimensional smoothness vector $\boldsymbol r = (r^{\texttt{v}}, r^e, r^f)^{\intercal}$, $\mathbb B_{k}(f;\boldsymbol r) := \mathbb B_{k}(f; (r^{\texttt{v}}, r^e)^{\intercal})$ is the face bubble using the restriction of $\boldsymbol r$ on $f$. Similarly, $\mathbb B_{k}(e;\boldsymbol r) := \mathbb B_{k}(e; r^{\texttt{v}})$. 

For a vector space $V$, we abbreviate $V\otimes \mathbb R^3$ as $V^3$. Define bubble spaces
\begin{align*}
 \mathbb B^{\curl}_{k}(T; \boldsymbol{r}):={}&\{\boldsymbol{v}\in \mathbb B^{3}_{k}(T; \boldsymbol{r}): \boldsymbol{v}\times \boldsymbol{n}|_{\partial T}=\boldsymbol{0}\},\\
 \mathbb B^{\div}_{k}(T; \boldsymbol{r}):={}&\{\boldsymbol{v}\in \mathbb B^{3}_{k}(T; \boldsymbol{r}): \boldsymbol{v}\cdot\boldsymbol{n}|_{\partial T}=0\}.
\end{align*}
Hereafter, $T$ will be omitted in the notation, i.e., $ \mathbb B_{k}(\boldsymbol r) = \mathbb B_{k}(T;\boldsymbol r)$. 
When $r^f\geq 0$, functions in $ \mathbb B_{k}(\boldsymbol{r})$ vanish on $\partial T$, thus $\mathbb B^{\curl}_{k}(\boldsymbol{r}) = \mathbb B^{\div}_{k}(\boldsymbol{r}) =  \mathbb B^{3}_{k}(\boldsymbol{r})$. When $r^f = -1$, $\mathbb B^{3}_{k}(\boldsymbol{r}_+)\subset  \mathbb B^{\dd}_{k}(\boldsymbol{r}) \subset \mathbb B^{3}_{k}(\boldsymbol{r})$ for $\dd =\curl, \div$ as only tangential or normal direction vanishes respectively. Precise characterization of bubble spaces $ \mathbb B^{\curl}_{k}(\boldsymbol{r})$ and $ \mathbb B^{\div}_{k}(\boldsymbol{r})$ can also be found in~\cite{ChenHuang2024}.

For a tensor space $\mathbb X$ with $\mathbb X=\mathbb M, \mathbb S, \mathbb T, \mathbb S\cap\mathbb T$,  define bubble spaces
\begin{align*}
\mathbb B^{\div}_{k}(\boldsymbol{r}; \mathbb X):={}&\{\boldsymbol{\tau}\in \mathbb B_{k}(\boldsymbol{r}; \mathbb X): \boldsymbol{\tau}\boldsymbol{n}|_{\partial T}=0\},\\
\mathbb B^{\curl}_{k}(\boldsymbol{r};\mathbb X):={}&\{\boldsymbol{\tau}\in\mathbb B_{k}(\boldsymbol{r}; \mathbb X): \boldsymbol{\tau}\times\boldsymbol{n}|_{\partial T}=0\},\\
\mathbb B^{\div\div}_{k}(\boldsymbol{r}; \mathbb X):={}&\{\boldsymbol{\tau}\in \mathbb B_{k}(\boldsymbol{r}; \mathbb X): (\boldsymbol n^{\intercal}\boldsymbol \tau\boldsymbol n)|_{\partial T}=0, \tr_2^{\div\div}(\bs \tau)=0, \\
&\qquad\qquad\qquad\quad\;\; (\boldsymbol n_i^{\intercal}\boldsymbol \tau\boldsymbol n_j)|_{e}=0 \textrm{ for } e\in\Delta_{1}(T) \textrm{ and }i,j=1,2\},\\
\mathbb B^{\sym\curl}_{k}(\boldsymbol{r}; \mathbb X):={}&\{\boldsymbol{\tau}\in \mathbb B_{k}(\boldsymbol{r}; \mathbb X): (\boldsymbol{n}\times\sym(\boldsymbol{\tau}\times\boldsymbol{n})\times\boldsymbol{n})|_{\partial T}=0, \\
&\qquad\qquad\qquad\qquad\qquad\qquad\;\; (\boldsymbol{n}\cdot\boldsymbol{\tau}\times\boldsymbol{n})|_{\partial T}=0\},\\
\mathbb B^{\inc}_{k}(\boldsymbol{r};\mathbb X):={}&\{\boldsymbol{\tau}\in\mathbb B_{k}(\boldsymbol{r}; \mathbb X): (\boldsymbol{n}\times\boldsymbol{\tau}\times\boldsymbol{n})|_{\partial T}=0, \tr_2^{\inc}(\boldsymbol{\tau})=0\}, \\
\mathbb B^{\cott}_{k}(\boldsymbol{r};\mathbb S\cap\mathbb T):={}&\{\boldsymbol{\tau}\in\mathbb B_{k}(\boldsymbol{r}; \mathbb S\cap\mathbb T): \tr_1^{\cott}(\boldsymbol{\tau})=0, \tr_2^{\cott}(\boldsymbol{\tau})=0, \tr_3^{\cott}(\boldsymbol{\tau})=0\}.
\end{align*}
Precise characterization of bubble spaces $\mathbb B^{\div}_{k}(\boldsymbol{r}; \mathbb S)$ and $\mathbb B^{\div}_{k}(\boldsymbol{r}; \mathbb T)$ is shown in~\cite{ChenHuang2025}.



 Let $\boldsymbol r= (r^{\texttt{v}}, r^e, r^f)^{\intercal}$ be a smoothness vector, and nonnegative integer $k\geq 2r^{\texttt{v}}+1$. The shape function space $\mathbb P_{k}(T)$ is determined by the DoFs
\begin{subequations}\label{eq:Cr3D}
\begin{align}
\label{eq:C13d0}
\nabla^j u (\texttt{v}), & \quad j=0,1,\ldots,r^{\texttt{v}}, \texttt{v}\in \Delta_0(T), \\
\label{eq:C13d1}
\int_e \frac{\partial^{j} u}{\partial n_1^{i}\partial n_2^{j-i}} \, q \dd s, & \quad q \in \mathbb B_{k-j}(e; r^{\texttt{v}} - j), 0\leq i\leq j\leq r^{e}, e\in \Delta_1(T), \\
\label{eq:C13d2}
\int_f \frac{\partial^{j} u}{\partial n_f^{j}} \, q \dd S, & \quad q \in \mathbb B_{k-j}(f;\boldsymbol r-j), 0\leq j\leq r^{f},  f\in \Delta_2(T), \\
\label{eq:C13d3}
\int_T u \, q \dx, & \quad q \in \mathbb B_k(T;\boldsymbol r).
\end{align} 
\end{subequations}
By Lemma~3.11 in \cite{ChenHuang2024}, the number of DoFs in \eqref{eq:C13d0} at one vertex is ${r^{\texttt{v}}+3\choose3}=\frac{1}{6}(r^{\texttt{v}}+1)(r^{\texttt{v}}+2)(r^{\texttt{v}}+3)$.
The number of DoFs in \eqref{eq:C13d1} on one edge is $\frac{1}{6}(r^e+1)(r^e+2)(3k+2r^e-6r^{\texttt{v}}-3)$.
As $b_e\geq 0$, the test function space in~\eqref{eq:C13d1} can be changed to $q\in \mathbb P_{k - 2(r^{\texttt{v}}+1) + j}(e)$. 


When considering a mesh $ \mathcal T_h $, the DoFs defined by~\eqref{eq:Cr3D} define the global $ C^{r^f} $-continuous finite element space as follows:
\begin{align}\label{eq:scalarfem}
\mathbb V_k(\mathcal T_h; \boldsymbol{r}) = \{u\in C^{r^f}(\Omega): & \, u|_T\in\mathbb P_k(T)\textrm{ for all } T\in\mathcal T_h, \\
\notag
&\qquad\textrm{ and all the DoFs~\eqref{eq:Cr3D} are single-valued}\}.
\end{align}
When $r^f=-1$, the notation $C^{-1}(\Omega)$ here and below means that no interelement continuity condition is imposed.
The reference to the mesh $ \mathcal T_h $ will subsequently be omitted in the notation.

\subsection{Existing bubble complexes in three dimensions}
We recall several bubble complexes in three dimensions in \cite{ChenHuang2025}.


We say that the smoothness vector $\boldsymbol{r}$ satisfies the div-vector stability condition if it satisfies 
\begin{equation*}
\begin{cases}
 r^f\geq 0, & \,r^e\geq2r^f+1\geq 1, \quad r^{\texttt{v}}\geq 2r^e\geq 2,\\
 r^f = -1, & 
\begin{cases}
 r^e \geq 1, & r^{\texttt{v}}\geq 2r^e\geq 2,\\
 r^e \in \{0, -1\}, & r^{\texttt{v}}\geq 2r^e + 1.
\end{cases}
\end{cases}
\end{equation*}

We begin by revisiting the bubble de Rham complex \eqref{eq:femderhambubblecomplex} established in \cite[Lemma A.2]{ChenHuang2025}.

\begin{lemma}
Let 
$$
\boldsymbol r_0 \geq 0, \quad \bs r_1 = \bs r_0 - 1, \quad \boldsymbol r_2=\boldsymbol r_1\ominus1, \quad \boldsymbol r_3=\boldsymbol r_2\ominus1
$$ 
be smoothness vectors.
Assume $\boldsymbol{r}_2$ satisfies the div-vector stability condition, and $k\geq\max\{2r_{2}^{\texttt{v}}, 1\}$.
Then the bubble de Rham complex 
\begin{equation}\label{eq:femderhambubblecomplex}
0\xrightarrow{\subset}\mathbb B_{k+2}(\boldsymbol{r}_0)\xrightarrow{\grad}\mathbb B^{\curl}_{k+1}(\boldsymbol{r}_1)\xrightarrow{\curl}\mathbb B^{\div}_{k}(\boldsymbol{r}_2)\xrightarrow{\div}\mathbb B_{k-1}(\boldsymbol{r}_3)/\mathbb R\to0
\end{equation}
is exact.
\end{lemma}


We say that the smoothness vector $\boldsymbol{r}$ satisfies the div-tensor stability condition if it satisfies one of the following conditions:
\begin{enumerate}
\item $r^{\texttt{v}} \geq 2r^e + 1$ and $r^e \geq 2(r^f + 1)$, or
\item $r^{\texttt{v}} \geq 0$ and $r^e = r^f = -1$.
\end{enumerate}


We then present the bubble Hessian complex \eqref{eq:fembubblehessiancomplex} and the bubble elasticity complex~\eqref{eq:fembubbleelasticitycomplex}; see Lemmas A.4 and A.6 in \cite{ChenHuang2025}.

\begin{lemma}
Let smoothness vectors
$$
\boldsymbol r_0 \geq (4,2,1)^{\intercal},\quad \boldsymbol r_1 = \boldsymbol r_0 -2,\quad \boldsymbol r_2=\boldsymbol r_1\ominus1,\quad \boldsymbol r_3= \boldsymbol r_2\ominus1.$$ 
Assume $\boldsymbol{r}_2$ satisfies the div-tensor stability condition, and $k\geq\max\{2r_2^{\texttt{v}}+2,3\}$.
Then the bubble Hessian complex
\begin{equation}\label{eq:fembubblehessiancomplex}
\resizebox{.925\hsize}{!}{$
0\xrightarrow{\subset} \mathbb B_{k+2}(\boldsymbol{r}_0)\xrightarrow{\hess}\mathbb B^{\curl}_{k}(\boldsymbol{r}_1;\mathbb S)\xrightarrow{\curl} \mathbb B^{\div}_{k-1}(\boldsymbol{r}_2;\mathbb T) \xrightarrow{\div} \mathbb B_{k-2}(\boldsymbol{r}_3;\mathbb R^3)/{\rm RT}\xrightarrow{}0
$}
\end{equation}
is exact. 
\end{lemma}


\begin{lemma}
Let smoothness vectors
\begin{equation*}
\boldsymbol{r}_0 \geq (2, 1, 0)^{\intercal},
\; 
\boldsymbol r_1 =\boldsymbol{r}_0-1, 
\;
 \boldsymbol{r}_2=\max\{\boldsymbol{r}_1\ominus2, (0, -1, -1)^{\intercal}\},
\;
 \boldsymbol{r}_3=\boldsymbol{r}_2\ominus1.
\end{equation*}
Assume $\boldsymbol{r}_2$ satisfies the div-tensor stability condition, and $k\geq\max\{2r_2^{\texttt{v}},1\}$. 
Then the bubble elasticity complex
\begin{equation}\label{eq:fembubbleelasticitycomplex}
\resizebox{.925\hsize}{!}{$
0\xrightarrow{\subset} \mathbb B_{k+4}(\boldsymbol{r}_0;\mathbb R^3)\xrightarrow{\defm}\mathbb B^{\inc}_{k+3}(\boldsymbol{r}_1;\mathbb S)\xrightarrow{\inc} \mathbb B^{\div}_{k+1}(\boldsymbol{r}_2;\mathbb S) \xrightarrow{\div} \mathbb B_{k}(\boldsymbol{r}_3;\mathbb R^3)/{\rm RM}\xrightarrow{}0
$}
\end{equation}
is exact.
\end{lemma}


We say that the smoothness vector $\boldsymbol{r}$ satisfies the div\,div stability condition if $\boldsymbol r$ satisfies the div-tensor stability condition and $\boldsymbol r\ominus1$ satisfies the div-vector stability condition.

At the end of this subsection, we show the bubble div\,div complex \eqref{eq:fembubbledivdivcomplex}; see Lemma~A.8 in \cite{ChenHuang2025}.

\begin{lemma}
\label{lem:fembubbledivdiv+complex}
Let 
\begin{equation*}
\boldsymbol{r}_0\geq (2, 1, 0)^{\intercal}, \quad 
\boldsymbol{r}_1=\boldsymbol{r}_0-1,\quad 
\boldsymbol{r}_2= \max\{\boldsymbol{r}_1\ominus1, (0, -1, -1)^{\intercal}\},\quad 
\boldsymbol{r}_3=\boldsymbol{r}_2\ominus2.
\end{equation*}
Assume $\boldsymbol r_2$ satisfies the div\,div stability condition, and $k\geq \max\{2r_2^{\texttt{v}}+1,3\}$.
We have the following exact bubble div\,div complex   
\begin{align}
\label{eq:fembubbledivdivcomplex}
0\xrightarrow{\subset}\mathbb B_{k+2} (\boldsymbol{r}_0;\mathbb R^3) &\xrightarrow{\dev\grad} \mathbb B_{k+1}^{\sym\curl}(\boldsymbol{r}_1;\mathbb{T}) \\
&\xrightarrow{\sym\curl}  \mathbb B_{k}^{\div\div}(\boldsymbol{r}_2;\mathbb{S}) \xrightarrow{\div\div} \mathbb B_{k-2}(\boldsymbol{r}_3)/\mathbb P_1(T)\rightarrow 0. \notag
\end{align}
\end{lemma}


\subsection{Bubble conformal Hessian complex in three dimensions}
We will now proceed to derive the bubble conformal Hessian complex \eqref{polybubbleCHcomplex3d} using the BGG framework.

First, we present the geometric decomposition of the bubble spaces $\mathbb B_{k+1}^{\sym\curl}(\boldsymbol{r}_1;\mathbb T)$ and $\mathbb B_{k+1}^{\sym\curl}(\boldsymbol{r}_1;\mathbb S\cap\mathbb T)$.
\begin{lemma}
For tetrahedron $T$, $\boldsymbol{r}_1\geq (1,0,-1)^{\intercal}$ and $k\geq2r_1^{\texttt{v}}$, we have the geometric decomposition 
\begin{align}
\label{eq:symcurlTbubbledecomp}
&\mathbb B_{k+1}^{\sym\curl}(\boldsymbol{r}_1;\mathbb T)=\mathbb B_{k+1}((\boldsymbol{r}_1)_+;\mathbb T) \\
\notag
&\qquad\qquad\qquad\quad \oplus\Oplus_{f\in\Delta_2(T)}[r_1^f=-1]\big(\mathbb B_{k+1}((\boldsymbol{r}_1)_+;f)\otimes{\rm span}\{\mathbb{T}_1, \mathbb{T}_2, \mathbb{T}_3\}\big),
\\
\label{eq:symcurlSTbubbledecomp}
&\mathbb B_{k+1}^{\sym\curl}(\boldsymbol{r}_1;\mathbb S\cap\mathbb T)=\mathbb B_{k+1}((\boldsymbol{r}_1)_+;\mathbb S\cap\mathbb T) \\
\notag
&\qquad\qquad\qquad\quad \oplus\Oplus_{f\in\Delta_2(T)}[r_1^f=-1]\big(\mathbb B_{k+1}((\boldsymbol{r}_1)_+;f)\otimes{\rm span}\{\mathbb{T}_1\}\big),
\end{align}
where $\mathbb{T}_1=\boldsymbol{t}_1\boldsymbol{t}_1^{\intercal}+\boldsymbol{t}_2\boldsymbol{t}_2^{\intercal}-2\boldsymbol{n}_f\boldsymbol{n}_f^{\intercal}$, $\mathbb{T}_2=\boldsymbol{t}_1\boldsymbol{n}_f^{\intercal}$, and  $\mathbb{T}_3=\boldsymbol{t}_2\boldsymbol{n}_f^{\intercal}$.
\end{lemma}
\begin{proof}
For each face $f\in\Delta_2(T)$, $\{\mathbb{T}_1, \mathbb{T}_2, \mathbb{T}_3\}$ and $\{\boldsymbol{n}_f\boldsymbol{t}_1^{\intercal}, \boldsymbol{n}_f\boldsymbol{t}_2^{\intercal}, \boldsymbol{t}_1\boldsymbol{t}_2^{\intercal}, \boldsymbol{t}_2\boldsymbol{t}_1^{\intercal}, \boldsymbol{t}_1\boldsymbol{t}_1^{\intercal}-\boldsymbol{t}_2\boldsymbol{t}_2^{\intercal}\}$ are orthogonal to each other, and their combination forms a basis of $\mathbb T$.
Then the geometric decomposition \eqref{eq:symcurlTbubbledecomp} holds.

Similarly, $\{\mathbb{T}_1\}$ and $\{\sym(\boldsymbol{n}_f\boldsymbol{t}_1^{\intercal}), \sym(\boldsymbol{n}_f\boldsymbol{t}_2^{\intercal}), \sym(\boldsymbol{t}_1\boldsymbol{t}_2^{\intercal}), \boldsymbol{t}_1\boldsymbol{t}_1^{\intercal}-\boldsymbol{t}_2\boldsymbol{t}_2^{\intercal}\}$ are orthogonal to each other, and their combination forms a basis of $\mathbb S\cap\mathbb T$.
Hence, the geometric decomposition \eqref{eq:symcurlSTbubbledecomp} holds.
\end{proof}

\begin{lemma}
Assume $\boldsymbol{r}_1\geq (1,0,-1)^{\intercal}$ satisfies the div-vector stability condition, and $k\geq2r_{1}^{\texttt{v}}$.
For tetrahedron $T$, it holds 
\begin{equation}\label{eq:trsymcurlTbubble}
\tr\big(\sym\curl\mathbb B_{k+1}^{\sym\curl}(\boldsymbol{r}_1;\mathbb T)\big)=\mathbb B_{k}(\boldsymbol r_1\ominus1)/\mathbb R.
\end{equation}
\end{lemma}
\begin{proof}
For $r_1^f\geq0$, $\mathbb B_{k+1}^{\sym\curl}(\boldsymbol r_1;\mathbb T)=\mathbb B_{k+1}(\boldsymbol r_1)\otimes\mathbb T$ and $\mathbb B_{k+1}^{\div}(\boldsymbol r_1)=\mathbb B_{k+1}(\boldsymbol r_1)\otimes\mathbb R^3$. For $r_1^f=-1$, $\vskw\mathbb T_1=0$, $2\vskw\mathbb T_2=\boldsymbol t_2$, and $2\vskw\mathbb T_3=-\boldsymbol t_1$, so the geometric decomposition~\eqref{eq:symcurlTbubbledecomp} and~\cite[Lemma 3.12]{ChenHuang2024} give the same image identity. Thus, in both cases,
\begin{equation}\label{eq:20240310}
\vskw\mathbb B_{k+1}^{\sym\curl}(\boldsymbol{r}_1;\mathbb T)=\mathbb B_{k+1}^{\div}(\boldsymbol{r}_1).
\end{equation}
The identity $\tr(\sym\curl\boldsymbol\tau)=\tr\curl\boldsymbol\tau=2\div\vskw\boldsymbol\tau$ and the independent vector-divergence bubble surjectivity in~\cite[Lemma 3.2]{ChenHuang2025} now prove~\eqref{eq:trsymcurlTbubble}.
\end{proof}

\begin{lemma}\label{lem:divdivontobubble}
Let $\boldsymbol r_2=(r_2^{\texttt{v}},r_2^{e},r_2^{f})^{\intercal}$ satisfy the div\,div stability condition, set $\boldsymbol r_3=\boldsymbol r_2\ominus2$, and assume $k\geq2r_2^{\texttt{v}}+4$.
We have the $(\div\div; \mathbb S\cap \mathbb T)$ stability
\begin{equation}\label{eq:divdivontobubble}
\div\div\mathbb B_{k}^{\div\div}(\boldsymbol{r}_2;\mathbb S\cap\mathbb T)=\mathbb B_{k-2}(\boldsymbol{r}_3)/{\rm CH}.
\end{equation}
\end{lemma}
\begin{proof}
Integration by parts gives the inclusion into $\mathbb B_{k-2}(\boldsymbol r_3)/{\rm CH}$. For the reverse inclusion, choose the auxiliary vector
\[
\widehat{\boldsymbol r}_1=
\begin{cases}
(r_2^{\texttt{v}}+1,r_2^{e}+1,-1)^{\intercal},&r_2^{f}=-1,\\
(r_2^{\texttt{v}}+1,r_2^{e}+1,r_2^{f}+1)^{\intercal},&r_2^{f}\geq0.
\end{cases}
\]
Then $\widehat{\boldsymbol r}_1\ominus1=\boldsymbol r_2$. The div-tensor inequalities $r_2^{\texttt{v}}\geq2r_2^{e}+1$, $r_2^{e}\geq2(r_2^{f}+1)$, or the exceptional case $r_2^{e}=r_2^{f}=-1$, verify that $\widehat{\boldsymbol r}_1$ is a smoothness vector satisfying the div-vector stability condition.

The face trace identities, the zero normal--normal edge block of the image, and the derivative vanishing orders give
\[
\sym\curl\mathbb B_{k+1}^{\sym\curl}(\widehat{\boldsymbol r}_1;\mathbb T)
\subseteq\mathbb B_k^{\div\div}(\boldsymbol r_2;\mathbb S)\cap\ker(\div\div).
\]
For $q\in\mathbb B_{k-2}(\boldsymbol r_3)/{\rm CH}$, take $\widetilde{\boldsymbol\sigma}\in\mathbb B_k^{\div\div}(\boldsymbol r_2;\mathbb S)$ with $\div\div\widetilde{\boldsymbol\sigma}=q$ by~\cite[Lemma 4.18]{ChenHuang2025}. The symmetric Green identity gives
\(
2\int_T\tr\widetilde{\boldsymbol\sigma}\dx=(q,|\boldsymbol x|^2)_T=0.
\)
By \eqref{eq:trsymcurlTbubble} with $\widehat{\boldsymbol r}_1$, choose $\boldsymbol\tau$ such that $\tr(\sym\curl\boldsymbol\tau)=\tr\widetilde{\boldsymbol\sigma}$. The preceding image inclusion shows that $\boldsymbol\sigma=\widetilde{\boldsymbol\sigma}-\sym\curl\boldsymbol\tau$ belongs to $\mathbb B_k^{\div\div}(\boldsymbol r_2;\mathbb S\cap\mathbb T)$ and has $\div\div\boldsymbol\sigma=q$.

We also record the full trace surjectivity needed below:
\begin{equation}\label{eq:BGGdivdivST}
\tr\mathbb B_k^{\div\div}(\boldsymbol r_2;\mathbb S)=\mathbb B_k(\boldsymbol r_2).
\end{equation}
Its mean-zero part follows from \eqref{eq:trsymcurlTbubble}. For the remaining direction, choose $w=\prod_{i=0}^3\lambda_i^{\alpha_i}$ with nonnegative exponents differing by at most one and summing to $2r_2^{\texttt{v}}$. The parameter inequalities give $w\mathbb P_2(T)\subseteq\mathbb B_{k-2}(T;\boldsymbol r_2\ominus2)$. Let $\ell$ be the $w$-weighted projection of $|\boldsymbol x|^2$ onto $\mathbb P_1(T)$ and set $q_*=w(|\boldsymbol x|^2-\ell)$. Then $q_*\perp\mathbb P_1(T)$. By~\cite[Lemma 4.18]{ChenHuang2025}, a symmetric bubble lifting $\boldsymbol\sigma_*$ satisfies
\[
2(\tr\boldsymbol\sigma_*,1)_T=(q_*,|\boldsymbol x|^2)_T
=\int_Tw(|\boldsymbol x|^2-\ell)^2\dx>0.
\]
Together with the mean-zero traces this proves \eqref{eq:BGGdivdivST}; for $r_2^{\texttt{v}}=0$ the weight is simply $w=1$.
In particular, rank--nullity gives
\[
\dim\mathbb B_k^{\div\div}(\boldsymbol r_2;\mathbb S\cap\mathbb T)
=\dim\mathbb B_k^{\div\div}(\boldsymbol r_2;\mathbb S)-\dim\mathbb B_k(\boldsymbol r_2).
\]
\end{proof}

\begin{lemma}
Let smoothness vectors $(\boldsymbol{r}_0,\boldsymbol{r}_1, \boldsymbol{r}_2, \boldsymbol{r}_3)$ satisfy \eqref{eq:rCHbubble}, and $k\geq 2r_2^{\texttt{v}}+4$.
Assume $\boldsymbol r_1$ satisfies the div-vector stability condition, and $\boldsymbol r_2$ satisfies the div\,div stability condition.
We have
\begin{equation}\label{eq:symcurlontobubble}
\sym\curl\mathbb B_{k+1}^{\sym\curl}(\boldsymbol{r}_1;\mathbb S\cap\mathbb T)=\mathbb B_{k}^{\div\div}(\boldsymbol{r}_2;\mathbb S\cap\mathbb T)\cap\ker(\div\div).
\end{equation}
\end{lemma}
\begin{proof}
First, $\sym\curl\mathbb B_{k+1}^{\sym\curl}(\boldsymbol{r}_1;\mathbb S\cap\mathbb T)\subseteq\mathbb B_{k}^{\div\div}(\boldsymbol{r}_2;\mathbb S\cap\mathbb T)\cap\ker(\div\div)$ follows from complex \eqref{eq:fembubbledivdivcomplex}, and the identity $\tr\curl\boldsymbol{\tau}=2\div\vskw\boldsymbol{\tau}$. Then we prove the other side by
the anti-commutative diagram
\begin{equation*}
\adjustbox{scale=0.85,center}{%
\begin{tikzcd}[column sep=normal, row sep=normal]
&
\mathbb B_{k+2} (\boldsymbol{r}_1+1;\mathbb R^3)
  \arrow{r}{\curl}
 &
\mathbb B_{k+1}^{\div}(\boldsymbol{r}_1)
\arrow{r}{\div}
& \mathbb B_{k}(\boldsymbol{r}_2)/\mathbb R \\
\mathbb B_{k+2} (\boldsymbol{r}_1+1;\mathbb R^3)
 \arrow[ur,swap,"{\rm id}"'] \arrow{r}{\dev\grad}
 & 
\mathbb B_{k+1}^{\sym\curl}(\boldsymbol{r}_1;\mathbb T)
 \arrow[ur,swap,"- 2\vskw"'] \arrow{r}{\sym\curl}
 & 
\mathbb B_{k}^{\div\div}(\boldsymbol{r}_2;\mathbb S)\cap\ker(\div\div)
\arrow[r]
& 0.
\end{tikzcd}
}
\end{equation*}

For $\boldsymbol{\sigma}\in \mathbb B_{k}^{\div\div}(\boldsymbol{r}_2;\mathbb S\cap\mathbb T)\cap\ker(\div\div)$, by the exactness of the bottom complex in the last diagram, there exists a $\widetilde{\boldsymbol{\tau}}\in\mathbb B_{k+1}^{\sym\curl}(\boldsymbol{r}_1;\mathbb T)$ such that $\sym\curl\widetilde{\boldsymbol{\tau}}=\boldsymbol{\sigma}$. The fact $\tr\boldsymbol{\sigma}=0$ means $\div\vskw\widetilde{\boldsymbol{\tau}}=0$, which together with \eqref{eq:20240310} implies $\vskw\widetilde{\boldsymbol{\tau}}\in\mathbb B_{k+1}^{\div}(\boldsymbol{r}_1)\cap\ker(\div)$. By the top complex in the last diagram, there exists a $\boldsymbol{v}\in\mathbb B_{k+2}(\boldsymbol{r}_1+1;\mathbb R^3)$ such that $\vskw\widetilde{\boldsymbol{\tau}}=\curl\boldsymbol{v}$.
Take $\boldsymbol{\tau}=\widetilde{\boldsymbol{\tau}}-2\dev\grad\boldsymbol{v}\in \mathbb B_{k+1}^{\sym\curl}(\boldsymbol{r}_1;\mathbb T)$. Furthermore,
$\vskw\boldsymbol{\tau}=\vskw\widetilde{\boldsymbol{\tau}}-\curl\boldsymbol{v}=0$, i.e. $\boldsymbol{\tau}$ is symmetric. Finally, $\sym\curl\boldsymbol{\tau}=\sym\curl\widetilde{\boldsymbol{\tau}}=\boldsymbol{\sigma}$.
\end{proof}




\begin{theorem}\label{thm:bubbleCHcomplex3d}
Let smoothness vectors $(\boldsymbol{r}_0, \boldsymbol{r}_1, \boldsymbol{r}_2, \boldsymbol{r}_3)$ satisfy \eqref{eq:rCHbubble}.
Assume $\boldsymbol r_1$ satisfies the div-vector stability condition, $\boldsymbol r_2$ satisfies the div\,div stability condition, and $k\geq 2r_2^{\texttt{v}}+4$. 
Then the bubble conformal Hessian complex \eqref{polybubbleCHcomplex3d} is exact.
\end{theorem}
\begin{proof}
Thanks to \eqref{eq:divdivontobubble} and \eqref{eq:symcurlontobubble}, it suffices to prove
\begin{equation*}
\dev\hess\mathbb B_{k+3}(\boldsymbol{r}_0)=\mathbb B_{k+1}^{\sym\curl}(\boldsymbol{r}_1;\mathbb S\cap\mathbb T)\cap\ker(\sym\curl).
\end{equation*}
The inclusion $\dev\hess\mathbb B_{k+3}(\boldsymbol{r}_0)\subseteq\mathbb B_{k+1}^{\sym\curl}(\boldsymbol{r}_1;\mathbb S\cap\mathbb T)\cap\ker(\sym\curl)$ is trivial. For $\boldsymbol{\tau}\in\mathbb B_{k+1}^{\sym\curl}(\boldsymbol{r}_1;\mathbb S\cap\mathbb T)\cap\ker(\sym\curl)$, by complex \eqref{eq:fembubbledivdivcomplex}, $\boldsymbol{\tau}=\dev\grad\boldsymbol{v}$ with $\boldsymbol{v}\in\mathbb B_{k+2}(\boldsymbol{r}_1+1;\mathbb R^3)$. Since $\boldsymbol{\tau}$ is symmetric,
\begin{equation*}
\curl\boldsymbol{v}=2\vskw(\dev\grad\boldsymbol{v})=2\vskw\boldsymbol{\tau}=0.
\end{equation*}
Apply bubble complex \eqref{eq:femderhambubblecomplex} to get $\boldsymbol{v}=\grad q$ with $q\in\mathbb B_{k+3}(\boldsymbol{r}_0)$. That is $\boldsymbol{\tau}=\dev\hess q\in\dev\hess\mathbb B_{k+3}(\boldsymbol{r}_0)$.
\end{proof}

\subsection{Bubble conformal elasticity complex in three dimensions}\label{subsec:bubbleCEcomplex}
Next, we derive the bubble conformal elasticity complex \eqref{polybubbleCEcomplex3d}
using the tilde operation in \cite{ChenHuang2025} and the BGG framework.

\subsubsection{Characterization of $H(\cott)$ bubble space}
Any tensor in $\mathbb B^{\cott}_{k}(\boldsymbol{r};\mathbb S\cap\mathbb T)$ vanishes on edges
\cite[Lemma 5.3]{HuLinShi2023}.
When $\boldsymbol{r}\geq(8,4,2)^{\intercal}$, $\mathbb B_{k+3}^{\cott}(\boldsymbol{r};\mathbb{S}\cap\mathbb{T})=\mathbb B_{k+3}(\boldsymbol{r};\mathbb{S}\cap\mathbb{T})$. Next we give a characterization of $\mathbb B_{k+3}^{\cott}(\boldsymbol{r};\mathbb{S}\cap\mathbb{T})$ for $r^f=1$.

\begin{lemma}
Let $\boldsymbol{r}\geq(6,3,1)^{\intercal}$ with $r^f=1$. We have
\begin{align}
\label{eq:cottbubblerf1}
\mathbb B_{k+3}^{\cott}(\boldsymbol{r};\mathbb{S}\cap\mathbb{T})=\{\boldsymbol{\tau}\in\mathbb B_{k+3}(\boldsymbol{r};\mathbb{S}\cap\mathbb{T})&: \partial_{nn}(\boldsymbol{t}_1^{\intercal}\boldsymbol{\tau}\boldsymbol{t}_1-\boldsymbol{t}_2^{\intercal}\boldsymbol{\tau}\boldsymbol{t}_2)|_f=0, \\
\notag
&\;\;\; \partial_{nn}(\boldsymbol{t}_1^{\intercal}\boldsymbol{\tau}\boldsymbol{t}_2)|_f=0, f\in\Delta_2(T)\},
\end{align}
where $\boldsymbol{n}$ is a unit normal vector of $f$, and $\boldsymbol{t}_1$ and $\boldsymbol{t}_2$ are two orthonormal tangential vectors of $f$.
Then 
\begin{align}
\label{eq:cottbubblerf2}
\mathbb B_{k+3}^{\cott}(\boldsymbol{r};\mathbb{S}\cap\mathbb{T})&=\{\boldsymbol{\tau}=b_T^2\boldsymbol{q}: \boldsymbol{q}\in \mathbb B_{k-5}(\bar{\boldsymbol{r}};\mathbb{S}\cap\mathbb{T}),  \\
\notag
&\qquad (\boldsymbol{t}_1^{\intercal}\boldsymbol{q}\boldsymbol{t}_1-\boldsymbol{t}_2^{\intercal}\boldsymbol{q}\boldsymbol{t}_2)|_f=0, (\boldsymbol{t}_1^{\intercal}\boldsymbol{q}\boldsymbol{t}_2)|_f=0, f\in\Delta_2(T)\},
\end{align}
where $\bar{\boldsymbol{r}}:=(r^{\texttt{v}}-6, r^e-4, -1)^{\intercal}$.
\end{lemma}
\begin{proof}
Let $\boldsymbol{\tau}\in\mathbb B_{k+3}^{\cott}(\boldsymbol{r};\mathbb{S}\cap\mathbb{T})$.
Set $\boldsymbol\sigma=\sym\curl\boldsymbol\tau$. Since $\nabla\boldsymbol\tau|_f=0$, we have $\boldsymbol\sigma|_f=0$ and therefore $\defm_f(\boldsymbol n\cdot\boldsymbol\sigma\Pi_f)=0$. By the definition of $\tr_2^{\inc}$,
\[
\tr_2^{\inc}(\boldsymbol\sigma)
=2\defm_f(\boldsymbol n\cdot\boldsymbol\sigma\Pi_f)
-\Pi_f\partial_n\boldsymbol\sigma\Pi_f
=-\Pi_f\partial_n\boldsymbol\sigma\Pi_f.
\]
On the other hand, $\tr_1^{\cott}(\boldsymbol\tau)=\tr_3^{\cott}(\boldsymbol\tau)=0$, so \eqref{eq:trtrcott3} gives $\tr_2^{\inc}(\boldsymbol\sigma)=0$. Hence
\[
\Pi_f\partial_n(\sym\curl\boldsymbol\tau)\Pi_f=0.
\]
Using $\nabla=\boldsymbol t_1\partial_{t_1}+\boldsymbol t_2\partial_{t_2}+\boldsymbol n\partial_n$ and $\nabla\boldsymbol\tau|_f=0$, a direct tangential-block calculation gives
\[
2\Pi_f\partial_n(\sym\curl\boldsymbol\tau)\Pi_f
=\partial_{nn}\big(\boldsymbol n\times\boldsymbol\tau\Pi_f
-\Pi_f\boldsymbol\tau\times\boldsymbol n\big)\big|_f.
\]
In the basis $(\boldsymbol t_1,\boldsymbol t_2)$, the matrix on the right before differentiation is
\[
\begin{pmatrix}
-2\boldsymbol t_1^{\intercal}\boldsymbol\tau\boldsymbol t_2&
\boldsymbol t_1^{\intercal}\boldsymbol\tau\boldsymbol t_1-\boldsymbol t_2^{\intercal}\boldsymbol\tau\boldsymbol t_2\\
\boldsymbol t_1^{\intercal}\boldsymbol\tau\boldsymbol t_1-\boldsymbol t_2^{\intercal}\boldsymbol\tau\boldsymbol t_2&
2\boldsymbol t_1^{\intercal}\boldsymbol\tau\boldsymbol t_2
\end{pmatrix}.
\]
This proves the two conditions in~\eqref{eq:cottbubblerf1}. Conversely, if $\boldsymbol\tau\in\mathbb B_{k+3}(\boldsymbol r;\mathbb S\cap\mathbb T)$ satisfies those two conditions, the same identity gives $\Pi_f\partial_n(\sym\curl\boldsymbol\tau)\Pi_f=0$. The first-order vanishing on $f$ gives $\tr_1^{\cott}(\boldsymbol\tau)=\tr_2^{\cott}(\boldsymbol\tau)=0$, and \eqref{eq:trtrcott3} then yields $\tr_3^{\cott}(\boldsymbol\tau)=0$. Thus \eqref{eq:cottbubblerf1} holds in both directions, and \eqref{eq:cottbubblerf2} follows.
\end{proof}


\begin{remark}\rm
The face conditions in~\eqref{eq:cottbubblerf2} force $\boldsymbol q$ to vanish at every vertex~\cite[Proposition 5.2]{HuLinShi2023}. Consequently, all derivatives of $b_T^2\boldsymbol q$ up to order six vanish at the vertices.
\end{remark}

\begin{lemma}
Let $\boldsymbol{r}\geq(6,3,1)^{\intercal}$. We have
\begin{equation}\label{eq:cottbubblerf3}
\mathbb B_{k+3}^{\cott}(\boldsymbol{r};\mathbb{S}\cap\mathbb{T})=\{\boldsymbol{\tau}\in\mathbb B_{k+3} (\boldsymbol{r}; \mathbb S\cap \mathbb T): S^{-1}\inc\boldsymbol{\tau}\in \mathbb B_{k+1}^{\curl}(\boldsymbol{r}-2;\mathbb S)\}.
\end{equation}
\end{lemma}
\begin{proof}
When $r^f\geq2$, both sides in \eqref{eq:cottbubblerf3} are $\mathbb B_{k+3}(\boldsymbol{r};\mathbb{S}\cap\mathbb{T})$, so $\eqref{eq:cottbubblerf3}$ is true. Next we assume $r^f=1$.
For $\boldsymbol{\tau}\in\mathbb B_{k+3}(\boldsymbol{r};\mathbb{S}\cap\mathbb{T})$, by $\nabla\boldsymbol{\tau}|_f=0$ for $f\in\Delta_2(T)$, $S^{-1}\inc\boldsymbol{\tau}\in \mathbb B_{k+1}^{\curl}(\boldsymbol{r}-2;\mathbb S)$ is equivalent to 
\begin{equation*}
\Pi_f\big(\inc\boldsymbol{\tau}-\frac{1}{2}\tr(\inc\boldsymbol{\tau})I\big)\Pi_f=0,\quad f\in\Delta_2(T).
\end{equation*}
Thanks to $\nabla\boldsymbol{\tau}|_f=0$ again, 
we have $(\nabla\times\boldsymbol{\tau}\times\nabla)|_f=\partial_{nn}(\boldsymbol{n}\times\boldsymbol{\tau}\times\boldsymbol{n})$. Then
the last equation becomes
\begin{equation}\label{eq:20240202}
\Pi_f\big(\partial_{nn}(\boldsymbol{n}\times\boldsymbol{\tau}\times\boldsymbol{n}) -\frac{1}{2}\tr(\partial_{nn}(\boldsymbol{n}\times\boldsymbol{\tau}\times\boldsymbol{n}))I\big)\Pi_f=0,\quad f\in\Delta_2(T).
\end{equation}
Notice that on face $f$,
\begin{equation*}
\tr(\boldsymbol{n}\times\boldsymbol{\tau}\times\boldsymbol{n})=\sum_{i=1}^2\boldsymbol{t}_i^{\intercal}(\boldsymbol{n}\times\boldsymbol{\tau}\times\boldsymbol{n})\boldsymbol{t}_i=-\sum_{i=1}^2\boldsymbol{t}_i^{\intercal}\boldsymbol{\tau}\boldsymbol{t}_i=\boldsymbol{n}^{\intercal}\boldsymbol{\tau}\boldsymbol{n}.
\end{equation*}
So equation \eqref{eq:20240202} is equivalent to
\begin{equation*}
\partial_{nn}(\boldsymbol{t}_1^{\intercal}\boldsymbol{\tau}\boldsymbol{t}_1)=\partial_{nn}(\boldsymbol{t}_2^{\intercal}\boldsymbol{\tau}\boldsymbol{t}_2)=-\frac{1}{2}\partial_{nn}(\boldsymbol{n}^{\intercal}\boldsymbol{\tau}\boldsymbol{n}), \quad \partial_{nn}(\boldsymbol{t}_1^{\intercal}\boldsymbol{\tau}\boldsymbol{t}_2)=0
\end{equation*}
on face $f$. Thus, \eqref{eq:cottbubblerf3} holds from \eqref{eq:cottbubblerf1}.
\end{proof}

\subsubsection{Bubble conformal elasticity complex}

\begin{lemma}\label{lem:terminal-bubble-divST}
Assume that $\boldsymbol{r} \geq (4, 1, -1)^{\intercal}$ satisfies the div-tensor stability condition, and that $\boldsymbol{r} + 1$ also satisfies the div-tensor stability condition.
Let $k\geq 2r^{\texttt{v}}+2$. Then we have the $(\div; \mathbb S\cap \mathbb T)$ stability
\begin{equation}\label{eq:divSTbubblestability1}
\div \mathbb B^{\div}_{k}(\boldsymbol{r};\mathbb S\cap \mathbb T) = 
\mathbb B_{k-1}(\boldsymbol{r}\ominus 1;\mathbb R^3)/{\rm CK}.
\end{equation}
\end{lemma}
\begin{proof}
The independent terminal bubble maps in~\cite[Lemma 4.8]{ChenHuang2025} give the following diagram:
\begin{equation*}
\begin{tikzcd}[column sep=small, row sep=normal]
&
\mathbb B^{\div}_{k+1}(\boldsymbol{r}+1; \mathbb S)
  \arrow{r}{\div}
 &
\mathbb B_{k}(\boldsymbol{r}; \mathbb R^3)
\arrow{r}{}
& {\rm RM} \\
{\mathbb B}^{\curl}_{k+1}(\boldsymbol{r}+1;\mathbb S)
 \arrow[ur,swap,"S"'] \arrow{r}{\curl}
 & 
\mathbb B^{\div}_{k}(\boldsymbol{r};\mathbb T)
 \arrow[ur,swap,"-2\vskw "'] \arrow{r}{\div}
 & 
\mathbb B_{k-1}(\boldsymbol{r}\ominus 1;\mathbb R^3)/{\rm RT}
\arrow[r] 
&0.
\end{tikzcd}
\end{equation*}
We use this diagram to prove \eqref{eq:divSTbubblestability1}.
As $\boldsymbol{r}+1\geq 0$, both ${\mathbb B}^{\div}_{k+1}(\boldsymbol{r}+1;\mathbb S) = {\mathbb B}^{\curl}_{k+1}(\boldsymbol{r}+1;\mathbb S) = {\mathbb B}_{k+1}(\boldsymbol{r} +1)\otimes \mathbb S$. Therefore, $S$ is one-to-one.

Given $\boldsymbol{v} \in \mathbb B_{k-1}(\boldsymbol{r}\ominus 1;\mathbb R^3)/{\rm CK}$, apply~\cite[Lemma 4.8]{ChenHuang2025} with $(\mathbb X,\boldsymbol r,m)=(\mathbb T,\boldsymbol r,k)$ to obtain $\boldsymbol\tau\in\mathbb B_k^{\div}(\boldsymbol r;\mathbb T)$ with $\div\boldsymbol\tau=\boldsymbol v$.
Thanks to $\dev\grad{\rm CK}=\mskw{\rm RM}$ by \eqref{eq:CKprop1}, $\vskw\boldsymbol{\tau}\in\mathbb B_{k}(\boldsymbol{r}; \mathbb R^3)/{\rm RM}$. Apply the same independent terminal result with $(\mathbb X,\boldsymbol r,m)=(\mathbb S,\boldsymbol r+1,k+1)$ to obtain $2\vskw\boldsymbol\tau=\div\boldsymbol\sigma$, where $\boldsymbol\sigma\in\mathbb B_{k+1}^{\div}(\boldsymbol r+1;\mathbb S)$. We can see that $\boldsymbol{\tau}-\curl S^{-1}\boldsymbol{\sigma}\in\mathbb B^{\div}_{k}(\boldsymbol{r};\mathbb T)$, and
\begin{equation*}
\div(\boldsymbol{\tau}-\curl S^{-1}\boldsymbol{\sigma})=\boldsymbol{v}, \quad 2\vskw(\boldsymbol{\tau}-\curl S^{-1}\boldsymbol{\sigma})=0.
\end{equation*}
This completes the proof.
\end{proof}


Let $\boldsymbol{r}_1\geq(6,3,1)^{\intercal}$, $\boldsymbol{r}_2=\boldsymbol{r}_1\ominus3$, $\boldsymbol{r}_3=\boldsymbol{r}_2\ominus1$, and $k\geq 2r_2^{\texttt{v}}+2$.
Assume both $\boldsymbol r_1 - 2$ and $\boldsymbol r_2$ satisfy the div-tensor stability condition.
Applying a $\widetilde{\quad}$ operation to the bubble Hessian complex~\eqref{eq:fembubblehessiancomplex} with the transition from $\mathbb T$ to $\mathbb S\cap\mathbb T$, we obtain the exact sequence
\begin{equation}\label{eq:elasticitybottomcomplex1}
0\xrightarrow{\subseteq}\mathbb B_{k+3}(\boldsymbol{r}_1)\xrightarrow{\hess}\widetilde{\mathbb B}_{k+1}^{\curl}(\boldsymbol{r}_1-2;\mathbb S) \xrightarrow{\curl} \mathbb B^{\div}_{k}(\boldsymbol{r}_2; \mathbb S\cap \mathbb T)\cap\ker(\div) \xrightarrow{\div} 0,
\end{equation}
where 
\begin{equation*}
\widetilde{\mathbb B}_{k+1}^{\curl}(\boldsymbol{r}_1-2;\mathbb S) :=\{\boldsymbol{\tau}\in\mathbb B_{k+1}^{\curl}(\boldsymbol{r}_1-2;\mathbb S):\curl\boldsymbol{\tau}\in\mathbb B^{\div}_{k}(\boldsymbol{r}_2; \mathbb S\cap \mathbb T)\}.
\end{equation*}
Then define space
\begin{equation*}
\widetilde{\mathbb B}^{\div}_{k+1}(\boldsymbol{r}_1-2; \mathbb S):=S(\widetilde{\mathbb B}_{k+1}^{\curl}(\boldsymbol{r}_1-2;\mathbb S)).
\end{equation*}
\begin{lemma}\label{lem:div0Mfem1}
Let $\boldsymbol{r}_1\geq(6,3,1)^{\intercal}$ and $k\geq 2r_1^{\texttt{v}}-4$.
It holds
\begin{equation*}
\widetilde{\mathbb B}^{\div}_{k+1}(\boldsymbol{r}_1-2; \mathbb S)\subseteq \mathbb B^{\div}_{k+1}(\boldsymbol{r}_1-2; \mathbb S)\cap\ker(\div).
\end{equation*}
\end{lemma}
\begin{proof}
By $(S\boldsymbol\tau)\boldsymbol{n} = -2\vskw(\boldsymbol\tau\times\boldsymbol{n})$, $\widetilde{\mathbb B}^{\div}_{k+1}(\boldsymbol{r}_1-2; \mathbb S)\subseteq \mathbb B^{\div}_{k+1}(\boldsymbol{r}_1-2; \mathbb S)$.
And $\div\widetilde{\mathbb B}^{\div}_{k+1}(\boldsymbol{r}_1-2; \mathbb S)=0$ follows from $\div(S\boldsymbol\tau) = 2\vskw(\curl\boldsymbol\tau).$
\end{proof}

\newtheorem*{cottonauxlemma}{Auxiliary endpoint lemma}
\begin{cottonauxlemma}\label{lem:cott-endpoint}
Let $k\geq8$, $\boldsymbol s_1=(6,3,1)^{\intercal}$, $\boldsymbol s_2=(3,0,-1)^{\intercal}$, and
\[
\mathcal Q_m=\{\boldsymbol q\in\mathbb P_m(T;\mathbb S\cap\mathbb T):
\dev_f(\Pi_f\boldsymbol q\Pi_f)=0\text{ on every face }f\}.
\]
Here $\dev_f$ is the tangential deviatoric part. Set $U_k=b_T^3\mathbb P_{k-8}(T;\mathbb R^3)$, $C_k=\mathbb B_{k+3}^{\cott}(\boldsymbol s_1;\mathbb S\cap\mathbb T)$, $D_k=\mathbb B_k^{\div}(\boldsymbol s_2;\mathbb S\cap\mathbb T)$, and $W_k=\mathbb B_{k-1}((2,-1,-1)^{\intercal};\mathbb R^3)$. Then $C_k=b_T^2\mathcal Q_{k-5}$ and
\[
0\longrightarrow U_k\xrightarrow{\dev\defm}C_k
\xrightarrow{\cott}D_k\xrightarrow{\div}W_k/{\rm CK}\longrightarrow0
\]
is exact, with $U_k$ defined by its polynomial factor rather than the inadmissible smoothness vector $(7,4,2)^{\intercal}$.
\end{cottonauxlemma}
\begin{proof}
The geometric decomposition in~\cite[Proposition 5.2]{HuLinShi2023} makes $\boldsymbol q\in\mathcal Q_m$ vanish at vertices. Factoring $b_T^2$ in \eqref{eq:cottbubblerf1}--\eqref{eq:cottbubblerf2} therefore gives $C_k=b_T^2\mathcal Q_{k-5}$. Theorem~5.1 of~\cite{HuLinShi2023}, with degree parameter $k+3$, proves exactness through the Cotton image. This image is $D_k\cap\ker\div$: zero normal face traces force a symmetric traceless tensor to vanish on edges, and the vertex vanishing of $\boldsymbol q$ makes $\cott(b_T^2\boldsymbol q)$ vanish there through order three. Finally, $D_k$ is exactly the space $\mathbb B_k^{\div,(3)}$ of~\cite{HuLinShi2023}, and $W_k$ consists of vector polynomials whose derivatives up to order two vanish at every vertex. Thus~\cite[Theorem 2.1, $s=3$]{HuLinShi2023} gives the terminal surjectivity for $k\geq7$, covering the stated range.
\end{proof}

For the regular branch below, assume explicitly that $\boldsymbol r_0=\boldsymbol r_1+1$ is also a smoothness vector. Applying the $\widetilde{\quad}$ operation to the bubble elasticity complex \eqref{eq:fembubbleelasticitycomplex} then gives the exact sequence
\begin{equation}\label{eq:elasticitybubblecomplex}
0\xrightarrow{\subset}\mathbb B_{k+4}(\boldsymbol{r}_0;\mathbb R^3)\xrightarrow{\defm}\widetilde{\mathbb B}_{k+3}^{\inc} (\boldsymbol{r}_1; \mathbb S) \xrightarrow{\inc} \widetilde{\mathbb B}^{\div}_{k+1}(\boldsymbol{r}_1-2; \mathbb S) \xrightarrow{\div} 0,
\end{equation}
where 
\begin{align*}
\widetilde{\mathbb B}_{k+3}^{\inc} (\boldsymbol{r}_1; \mathbb S)&:=\{\boldsymbol{\tau}\in\mathbb B_{k+3} (\boldsymbol{r}_1; \mathbb S): \inc\boldsymbol{\tau}\in\widetilde{\mathbb B}^{\div}_{k+1}(\boldsymbol{r}_1-2; \mathbb S)\}.
\end{align*}
By the definition of $\widetilde{\mathbb B}^{\div}_{k+1}(\boldsymbol{r}_1-2; \mathbb S)$ and $\widetilde{\mathbb B}_{k+1}^{\curl}(\boldsymbol{r}_1-2;\mathbb S)$, we have
\begin{align*}
\widetilde{\mathbb B}_{k+3}^{\inc} (\boldsymbol{r}_1; \mathbb S)&=\{\boldsymbol{\tau}\in\mathbb B_{k+3} (\boldsymbol{r}_1; \mathbb S): S^{-1}\inc\boldsymbol{\tau}\in \widetilde{\mathbb B}_{k+1}^{\curl}(\boldsymbol{r}_1-2;\mathbb S)\} \\
&=\{\boldsymbol{\tau}\in\mathbb B_{k+3} (\boldsymbol{r}_1; \mathbb S): S^{-1}\inc\boldsymbol{\tau}\in \mathbb B_{k+1}^{\curl}(\boldsymbol{r}_1-2;\mathbb S)\}.
\end{align*}

We further give a characterization of the space $\widetilde{\mathbb B}_{k+3}^{\inc} (\boldsymbol{r}_1; \mathbb S)$.
\begin{lemma}
Let $\boldsymbol{r}_1\geq(6,3,1)^{\intercal}$ and $k\geq 2r_1^{\texttt{v}}-4$. We have
\begin{equation}\label{eq:VcurlMtildeDecomp}
\widetilde{\mathbb B}_{k+3}^{\inc} (\boldsymbol{r}_1; \mathbb S)=\mathbb B_{k+3}^{\cott}(\boldsymbol{r}_1;\mathbb{S}\cap\mathbb{T}) \oplus\iota\mathbb B_{k+3}(\boldsymbol{r}_1).
\end{equation}
\end{lemma}
\begin{proof}
By \eqref{eq:cottbubblerf3}, $\mathbb B_{k+3}^{\cott}(\boldsymbol{r}_1;\mathbb{S}\cap\mathbb{T})\subseteq\widetilde{\mathbb B}_{k+3}^{\inc} (\boldsymbol{r}_1; \mathbb S)$.
For $\boldsymbol{\tau}=\iota v\in\iota\mathbb B_{k+3}(\boldsymbol{r}_1)$, it follows
\begin{equation*}
\inc\boldsymbol{\tau}=\inc(\iota v)=-S(\hess v)\in\widetilde{\mathbb B}^{\div}_{k+1}(\boldsymbol{r}_1-2; \mathbb S),
\end{equation*}
which gives $\iota\mathbb B_{k+3}(\boldsymbol{r}_1)\subseteq\widetilde{\mathbb B}_{k+3}^{\inc} (\boldsymbol{r}_1; \mathbb S)$. 
Then decomposition~\eqref{eq:VcurlMtildeDecomp} holds from $\boldsymbol{\tau}=\dev\boldsymbol{\tau}+\frac{1}{3}\iota(\tr\boldsymbol{\tau})$ and \eqref{eq:cottbubblerf3}.
\end{proof}

Now we are in the position to establish the exactness of the bubble conformal elasticity complex \eqref{polybubbleCEcomplex3d}.

\begin{lemma}\label{lem:bubbleCEcomplex3d}
Let $\boldsymbol{r}_1\geq(6,3,1)^{\intercal}$, $\boldsymbol{r}_0=\boldsymbol{r}_1+1$, $\boldsymbol{r}_2=\boldsymbol{r}_1\ominus3$, $\boldsymbol{r}_3=\boldsymbol{r}_2\ominus1$ and $k\geq 2r_2^{\texttt{v}}+2$. Assume that these vectors are smoothness vectors, except at the auxiliary endpoint $\boldsymbol r_1=(6,3,1)^{\intercal}$, where the starting space is replaced by $U_k$ as in the auxiliary endpoint lemma.
Assume both $\boldsymbol r_1 - 2$ and $\boldsymbol r_2$ satisfy the div-tensor stability condition.
Then the bubble conformal elasticity complex \eqref{polybubbleCEcomplex3d} is exact.
\end{lemma}
\begin{proof}
The auxiliary endpoint is proved above. For the remaining parameters, $\boldsymbol r_0$ is a valid smoothness vector, so both input bubble complexes and the reduction step apply with their stated assumptions.
Combining complex~\eqref{eq:elasticitybubblecomplex} and complex~\eqref{eq:elasticitybottomcomplex1} induces the diagram
\begin{equation*}
\begin{tikzcd}
\mathbb B_{k+4}(\boldsymbol{r}_0;\mathbb R^3)
\arrow{r}{\defm}
&
\widetilde{\mathbb B}_{k+3}^{\inc} (\boldsymbol{r}_1; \mathbb S)
 \arrow{r}{\inc}
 &
\widetilde{\mathbb B}^{\div}_{k+1}(\boldsymbol{r}_1-2; \mathbb S)
 \arrow{r}{\div}
 & 0
 \\
\mathbb B_{k+3}(\boldsymbol{r}_1)
 \arrow[ur,swap,"\iota"'] \arrow{r}{\hess}
 & 
\widetilde{\mathbb B}_{k+1}^{\curl}(\boldsymbol{r}_1-2;\mathbb S) 
 \arrow[ur,swap,"S"'] \arrow{r}{\curl}
 & 
\mathbb B^{\div}_{k}(\boldsymbol{r}_2; \mathbb S\cap \mathbb T)\cap\ker(\div). 
 \arrow[ur,swap,"-2\vskw"'] 
\end{tikzcd}
\end{equation*}
With the decomposition~\eqref{eq:VcurlMtildeDecomp} at our disposal, we apply the BGG framework to arrive at
\begin{equation}\label{eq:202402023}
0\xrightarrow{\subseteq}\mathbb B_{k+4}(\boldsymbol{r}_0;\mathbb R^3) \xrightarrow{\dev\defm} \mathbb B_{k+3}^{\cott}(\boldsymbol{r}_1;\mathbb{S}\cap\mathbb{T}) \xrightarrow{\cott}  \mathbb B^{\div}_{k}(\boldsymbol{r}_2; \mathbb S\cap \mathbb T)\cap\ker(\div).
\end{equation}

Then we prove $\div\mathbb B^{\div}_{k}(\boldsymbol{r}_2; \mathbb S\cap \mathbb T)=\mathbb B_{k-1}(\boldsymbol{r}_3;\mathbb R^3)/{\rm CK}$ to acquire the exactness of bubble complex \eqref{polybubbleCEcomplex3d}.
When $\boldsymbol r_1\geq(8,4,1)^{\intercal}$, it follows from the $(\div; \mathbb S\cap \mathbb T)$ stability~\eqref{eq:divSTbubblestability1}.
For the remaining regular case $\boldsymbol r_1=(r_2^{\texttt{v}}+3,3,1)^{\intercal}$, $\boldsymbol r_2=(r_2^{\texttt{v}},0,-1)^{\intercal}$, let $n_0,n_1$ be the dimensions of the first two spaces in \eqref{eq:202402023} and $n_2=\dim\mathbb B_k^{\div}(\boldsymbol r_2;\mathbb S\cap\mathbb T)$. Factoring $b_T^3$, using \eqref{eq:cottbubblerf2}, and invoking~\cite[Proposition 5.3]{HuLinShi2023} give
\[
\begin{aligned}
n_0&=3\binom{k-5}{3}-12\binom{r_2^{\texttt{v}}-2}{3},\\
n_1&=5\binom{k-2}{3}-8\binom{k-3}{2}-20\binom{r_2^{\texttt{v}}}{3}+24\binom{r_2^{\texttt{v}}-1}{2},\\
n_2&=8\binom{k-1}{2}+5\binom{k-1}{3}-24\binom{r_2^{\texttt{v}}}{2}-20\binom{r_2^{\texttt{v}}}{3}.
\end{aligned}
\]
Exactness of \eqref{eq:202402023} now yields
\[
\dim\div\mathbb B_k^{\div}(\boldsymbol r_2;\mathbb S\cap\mathbb T)
=n_2-n_1+n_0
=3\binom{k+2}{3}-10-12\binom{r_2^{\texttt{v}}+2}{3},
\]
which is $\dim(\mathbb B_{k-1}(\boldsymbol r_3;\mathbb R^3)/{\rm CK})$. The inclusion from integration by parts proves the required surjectivity.
\end{proof}

\begin{theorem}\label{thm:bubbleCEcomplex3d}
Let smoothness vectors $(\boldsymbol{r}_0, \boldsymbol{r}_1, \boldsymbol{r}_2, \boldsymbol{r}_3)$ satisfy \eqref{eq:rCE}.
Assume both $\boldsymbol r_1 \ominus 2$ and $\boldsymbol r_2$ satisfy the div-tensor stability condition, and $k\geq 2r_2^{\texttt{v}}+2$.
Then the bubble conformal elasticity complex \eqref{polybubbleCEcomplex3d} is exact.
\end{theorem}
\begin{proof}
The case $\boldsymbol{r}_1\geq(6,3,1)^{\intercal}$ has been proved in Lemma~\ref{lem:bubbleCEcomplex3d}. We assume $r_1^f=-1,0$.

Since both $\boldsymbol r_1$ and $\boldsymbol r_0=\boldsymbol r_1+1$ are smoothness vectors, $r_1^{\texttt{v}}\geq2r_1^{e}+1$ and $r_1^{e}\geq2r_1^f+1$. Thus, for $r_1^{\texttt{v}}\geq7$, the auxiliary vector $(r_1^{\texttt{v}},\max\{r_1^e,3\},1)^{\intercal}$ and its increment by one are smoothness vectors. The remaining case $r_1^{\texttt{v}}=6$ uses the auxiliary endpoint lemma with starting space $U_k$.

By~\cite[Theorem 7.1]{HuLinShi2023}, with degree parameter $k+3$, any Cotton-free bubble $\boldsymbol\tau$ is $\dev\defm\boldsymbol v$ with $\boldsymbol v\in b_T\mathbb P_k(T;\mathbb R^3)$. For a face normal $\boldsymbol n$ and a unit tangent $\boldsymbol t$,
\[
\begin{aligned}
\partial_n(\boldsymbol v\cdot\boldsymbol t)&=2\boldsymbol n^{\intercal}\boldsymbol\tau\boldsymbol t-\partial_t(\boldsymbol v\cdot\boldsymbol n),\\
\partial_n(\boldsymbol v\cdot\boldsymbol n)&=\boldsymbol n^{\intercal}\boldsymbol\tau\boldsymbol n-\boldsymbol t^{\intercal}\boldsymbol\tau\boldsymbol t+\partial_t(\boldsymbol v\cdot\boldsymbol t).
\end{aligned}
\]
Using both tangent directions and $\boldsymbol v|_{\partial T}=0$, induction on the number of normal derivatives gives the prescribed vertex, edge, and face vanishing through orders $\boldsymbol r_0=\boldsymbol r_1+1$. Thus $\boldsymbol v\in\mathbb B_{k+4}(\boldsymbol r_0;\mathbb R^3)$. The reverse inclusion follows from the trace identities, completing the kernel identification.
Set $\widehat{\boldsymbol r}_1=(r_1^{\texttt{v}},\max\{r_1^e,3\},1)^{\intercal}$ and $\widehat{\boldsymbol r}_2=\widehat{\boldsymbol r}_1\ominus3$. Then
\[
\begin{aligned}
\mathbb B_k^{\div}(\widehat{\boldsymbol r}_2;\mathbb S\cap\mathbb T)
&=\mathbb B_k^{\div}(\boldsymbol r_2;\mathbb S\cap\mathbb T),
\qquad \widehat{\boldsymbol r}_2\ominus1=\boldsymbol r_3,\\
\mathbb B_{k+3}^{\cott}(\widehat{\boldsymbol r}_1;\mathbb S\cap\mathbb T)
&\subseteq\mathbb B_{k+3}^{\cott}(\boldsymbol r_1;\mathbb S\cap\mathbb T).
\end{aligned}
\]
The first equality uses the fact that zero face traction forces a symmetric traceless tensor to vanish on edges, so the edge parameters $-1$ and $0$ give the same divergence bubble space. Exactness for the auxiliary parameters follows from Lemma~\ref{lem:bubbleCEcomplex3d}, including its auxiliary endpoint when $r_1^{\texttt{v}}=6$. Together with the complex property, these relations give
\begin{equation}\label{eq:20240314}
\cott\mathbb B_{k+3}^{\cott}(\boldsymbol{r}_1;\mathbb{S}\cap\mathbb{T})=\mathbb B^{\div}_{k}(\boldsymbol{r}_2;\mathbb S\cap\mathbb T)\cap\ker(\div)
\end{equation}
and $\div\mathbb B_k^{\div}(\boldsymbol r_2;\mathbb S\cap\mathbb T)=\mathbb B_{k-1}(\boldsymbol r_3;\mathbb R^3)/{\rm CK}$.
\end{proof}

\section{Finite element conformal Hessian complex in three dimensions}\label{sec:FEconformalHesscomplex}

In this section, we will devise the finite element conformal Hessian complex
\begin{align}\label{femCHcomplex3d}
{\rm CH}\xrightarrow{\subset} \mathbb V_{k+3}(\boldsymbol{r}_0)&\xrightarrow{\dev\hess} \mathbb V_{k+1}^{\sym\curl}(\boldsymbol{r}_1;\mathbb S\cap\mathbb T)\xrightarrow{\sym\curl} \mathbb V_{k}^{\div\div}(\boldsymbol{r}_2;\mathbb S\cap\mathbb T) \\
\notag
&\xrightarrow{\div{\div}} \mathbb V_{k-2}(\boldsymbol{r}_3)\xrightarrow{}0
\end{align}
with $k\geq 2r_2^{\texttt{v}}+4$ and smoothness vectors $(\boldsymbol{r}_0, \boldsymbol{r}_1, \boldsymbol{r}_2, \boldsymbol{r}_3)$ satisfying
\begin{equation}\label{eq:rCH}
\boldsymbol{r}_0\geq (4,2,1)^{\intercal}, \quad
\boldsymbol r_1=\boldsymbol r_0-2  \geq (2,0,-1)^{\intercal},
\quad
 \boldsymbol{r}_2=\boldsymbol{r}_1\ominus1,
\quad
 \boldsymbol{r}_3=\boldsymbol{r}_2\ominus2.
\end{equation}
Here spaces $\mathbb V_{k+3}(\boldsymbol{r}_0)$ and $\mathbb V_{k-2}(\boldsymbol{r}_3)$ are the smooth scalar finite element spaces defined as~\eqref{eq:scalarfem}.
We will define space $\mathbb V_{k}^{\div\div}(\boldsymbol{r}_2;\mathbb S\cap\mathbb T)$ in Section~\ref{subsec:divdivSTfem} and space $\mathbb V_{k+1}^{\sym\curl}(\boldsymbol{r}_1;\mathbb S\cap\mathbb T)$ in Section~\ref{subsec:symcurlSTfem}.

\subsection{$H(\div\div)$-conforming finite element}\label{subsec:divdivSTfem}
When $\boldsymbol{r}_2\geq(4,2,1)^{\intercal}$, define $\mathbb V_{k}^{\div\div}(\boldsymbol{r}_2;\mathbb S\cap \mathbb T):=\mathbb V_{k}(\boldsymbol{r}_2)\otimes (\mathbb S\cap \mathbb T)$.
Then we consider the construction of $H(\div\div)$-conforming $\mathbb V_{k}^{\div\div}(\boldsymbol{r}_2;\mathbb S\cap \mathbb T)$ for $\boldsymbol{r}_2\geq(0,-1,-1)^{\intercal}$ with $r_2^f=-1,0$.

Take $\mathbb P_k(T;\mathbb S\cap\mathbb T)$ as the space of shape functions. 
For $r_2^f=-1,0$, the DoFs are given by
\begin{subequations}\label{eq:divdivSTdofs}
\begin{align}
\label{eq:divdivSTdV}
\nabla^j\boldsymbol{\tau} (\texttt{v}), & \quad j=0,1,\ldots,r_2^{\texttt{v}}, \texttt{v}\in \Delta_0(T),\\
\label{eq:divdivSTdE1}
\int_e \frac{\partial^{j}\boldsymbol{\tau}}{\partial n_1^{i}\partial n_2^{j-i}}:\boldsymbol{q} \dd s, & \quad \boldsymbol{q} \in \mathbb P_{k - 2(r_2^{\texttt{v}}+1) + j}(e;\mathbb S\cap\mathbb T), 0\leq i\leq j\leq r_2^{e}, e\in \Delta_1(T), \\
\int_e (\boldsymbol{n}_i^{\intercal}\boldsymbol{\tau}\boldsymbol{n}_j)\,q \dd s, &\quad q\in \mathbb B_{k}(e; r_2^{\texttt{v}}), 1\leq i\leq j\leq 2, e\in \Delta_1(T), \textrm{ if } r_2^{e}=-1, \label{eq:divdivSTdE2}\\
\label{eq:divdivSTdF1}
\int_f \boldsymbol{\tau} :\boldsymbol{q} \dd S, & \quad  \boldsymbol{q}\in \mathbb B_{k}(f;\boldsymbol r_2)\otimes (\mathbb S\cap\mathbb T), f\in \Delta_2(T), \textrm{ if } r_2^f=0,\\
\label{eq:divdivSTdF2}
\int_f (\boldsymbol n^{\intercal}\boldsymbol{\tau}\boldsymbol n) \ q \dd S, & \quad  q\in \mathbb B_{k}(f;(\boldsymbol{r}_2)_+), f\in \Delta_2(T), \textrm{ if } r_2^f=-1,\\
\label{eq:divdivSTdF3}
\int_f \tr_2^{\div\div}(\boldsymbol \tau) \, q \dd S, & \quad  q\in \mathbb P_2(f)\oplus(\mathbb B_{k-1}(f;\boldsymbol r_2\ominus1)\cap\mathbb P_{2}^{\perp}(f)), f\in \Delta_2(T),\\
\label{eq:divdivSTdT1}
\int_T (\div\div\boldsymbol{\tau})\,q\dx, & \quad q\in \mathbb B_{k-2}(\boldsymbol{r}_2\ominus2)/{\rm CH}, \\
\label{eq:divdivSTdT2}
\int_T \boldsymbol{\tau}:\boldsymbol{q} \dx, & \quad \boldsymbol{q} \in \mathbb B_{k}^{\div\div}(\boldsymbol{r}_2;\mathbb S\cap\mathbb T)\cap\ker(\div\div).
\end{align}
\end{subequations}

\begin{lemma}\label{lem:20240316}
Let $\boldsymbol r_2$ satisfy the div\,div stability condition, with $r_2^f=-1,0$ and $k\geq2r_2^{\texttt{v}}+4$.
The number of DoFs \eqref{eq:divdivSTdofs} equals $\dim\mathbb P_k(T;\mathbb S\cap\mathbb T)$.
\end{lemma}
\begin{proof}
Write 
$B_f=\mathbb B_{k-1}(f;\boldsymbol r_2\ominus1)$. A balanced triangular Bernstein monomial $w_1$ of degree $2r_2^{\texttt{v}}+1$ satisfies $w_1\mathbb P_2(f)\subseteq B_f$: the required vertex and edge vanishing orders are at most $r_2^{\texttt{v}}$ and $r_2^{e}$, and $2r_2^{\texttt{v}}+3\leq k-1$. Positivity of $\int_fw_1p^2\dd S$ for $0\ne p\in\mathbb P_2(f)$ makes the pairing $B_f\times\mathbb P_2(f)$ full rank, so
\[
\dim\big(\mathbb P_2(f)\oplus(B_f\cap\mathbb P_2(f)^\perp)\big)=\dim B_f.
\]
The moments determine a residual in $B_f$ by first making it orthogonal to $\mathbb P_2(f)$ and then testing its squared norm. Thus \eqref{eq:divdivSTdF3} is an equivalent replacement of moments.

Compare with the symmetric element~\cite[Section 4.4, (40a)--(40f)]{ChenHuang2025}, valid under the present terminal assumptions. Passing to $\mathbb S\cap\mathbb T$ removes the trace component from the full tensor data in~\eqref{eq:divdivSTdV}--\eqref{eq:divdivSTdE1} and, for $r_2^f=0$, in~\eqref{eq:divdivSTdF1}, hence removes
\[
\dim\mathbb P_k(T)-\dim\mathbb B_k(\boldsymbol r_2)
\]
boundary data. The normal--normal edge blocks and scalar face traces retain their dimensions. Trace surjectivity in Lemma~\ref{lem:divdivontobubble} removes $\dim\mathbb B_k(\boldsymbol r_2)$ interior data, while its terminal surjectivity identifies \eqref{eq:divdivSTdT1}--\eqref{eq:divdivSTdT2} with the traceless bubble dimension. The total is therefore
\[
6\dim\mathbb P_k(T)
-\big(\dim\mathbb P_k(T)-\dim\mathbb B_k(\boldsymbol r_2)\big)
-\dim\mathbb B_k(\boldsymbol r_2)
=5\dim\mathbb P_k(T).
\]
\end{proof}

\begin{lemma}
Let $\boldsymbol r_2$ satisfy the div\,div stability condition, with $r_2^f=-1,0$ and $k\geq2r_2^{\texttt{v}}+4$.
DoFs \eqref{eq:divdivSTdofs} are unisolvent for $\mathbb P_k(T;\mathbb S\cap\mathbb T)$.
\end{lemma}
\begin{proof}
Thanks to Lemma~\ref{lem:20240316}, the dimensions match.
Take $\boldsymbol{\tau}\in\mathbb P_k(T;\mathbb S\cap\mathbb T)$, and assume all the DoFs \eqref{eq:divdivSTdofs} vanish. 
The vanishing DoFs \eqref{eq:divdivSTdV}-\eqref{eq:divdivSTdF3} indicate $\boldsymbol{\tau}\in\mathbb B_{k}^{\div\div}(\boldsymbol r_2;\mathbb S\cap\mathbb T)$. Then
integration by parts gives
\begin{equation*}
(\div\div\boldsymbol \tau, q)_T=0, \quad \forall~q\in  {\rm CH}.
\end{equation*}
This together with \eqref{eq:divdivontobubble} and the vanishing DoF \eqref{eq:divdivSTdT1} yields $\div\div\boldsymbol \tau=0$. Finally,
we conclude $\boldsymbol{\tau}=0$ from the vanishing DoF \eqref{eq:divdivSTdT2}.
\end{proof}

For $r_2^f=-1,0$, the finite element space $\mathbb V_{k}^{\div\div}(\boldsymbol{r}_2;\mathbb S\cap \mathbb T)$ is defined as follows
\begin{align*}
\mathbb V_{k}^{\div\div}(\boldsymbol{r}_2;\mathbb S\cap \mathbb T):=\{\boldsymbol \tau\in L^2(\Omega;\mathbb S\cap \mathbb T): &\, \boldsymbol \tau|_T\in\mathbb P_k(T;\mathbb S\cap\mathbb T) \textrm{ for each } T\in\mathcal T_h, \\
&\qquad\quad\;\;\; \textrm{ DoFs \eqref{eq:divdivSTdofs} are single-valued} \}.    
\end{align*}
The single-valued DoFs in \eqref{eq:divdivSTdV}-\eqref{eq:divdivSTdF3} imply $\mathbb V_{k}^{\div\div}(\boldsymbol{r}_2;\mathbb S\cap \mathbb T)\subset H(\div\div, \Omega; \mathbb S\cap \mathbb T)$ in view of Lemma \ref{lm:divdivconforming}.

\begin{lemma}\label{lem:global-terminal-divdiv}
Let $\boldsymbol r_2$ satisfy the div\,div stability condition and let $k\geq2r_2^{\texttt{v}}+4$.
We have the $(\div\div; \mathbb S\cap\mathbb T)$ stability
\begin{equation}\label{eq:divdivSTstability}
\div\div\mathbb V_{k}^{\div\div}(\boldsymbol{r}_2;\mathbb S\cap \mathbb T)=\mathbb V_{k-2}(\boldsymbol{r}_2\ominus2).
\end{equation}
\end{lemma}
\begin{proof}
Put $D=\div\div$. The forward inclusion follows from the defining smoothness and traces. Fix $v_h\in\mathbb V_{k-2}(\boldsymbol r_2\ominus2)$. For $r_2^f\geq1$, the global stability results in~\cite[Lemmas 3.4 and 4.5, and Corollary 4.15]{ChenHuang2025} provide $\boldsymbol\sigma_h\in\mathbb V_k(\boldsymbol r_2)\otimes\mathbb S$ and $\boldsymbol w_h\in\mathbb V_{k+1}(\boldsymbol r_2+1;\mathbb R^3)$ such that $D\boldsymbol\sigma_h=v_h$ and $2\div\boldsymbol w_h=\tr\boldsymbol\sigma_h$. Then
\[
\boldsymbol\tau_h=\boldsymbol\sigma_h-\sym\curl(\mskw\boldsymbol w_h)
\]
belongs to $\mathbb V_k(\boldsymbol r_2)\otimes(\mathbb S\cap\mathbb T)$ and satisfies $D\boldsymbol\tau_h=v_h$, since $\tr(\sym\curl(\mskw\boldsymbol w_h))=2\div\boldsymbol w_h$ and $D\sym\curl=0$. Henceforth assume $r_2^f=-1,0$.

Set $E_e=\mathbb B_k(e;r_2^{\texttt{v}})$ and $F_f=\mathbb B_k(f;(\boldsymbol r_2)_+)$. After fixing the lower-dimensional derivative data, use the equivalent test spaces
\[
Q_e=\mathbb P_2(e)\oplus(E_e\cap\mathbb P_2(e)^\perp),\qquad
Q_f=\mathbb P_0(f)\oplus(F_f\cap\mathbb P_0(f)^\perp).
\]
Indeed, $b_e^{r_2^{\texttt{v}}+1}\mathbb P_2(e)\subseteq E_e$, and $F_f$ contains a positive Bernstein monomial of degree $2r_2^{\texttt{v}}+3$. The argument of Lemma~\ref{lem:20240316} proves equivalence with~\eqref{eq:divdivSTdE2} and~\eqref{eq:divdivSTdF2}. The same replacement applies to the normal--normal components of the zero-order full moments.

Take $\boldsymbol\tau\in H^2(\Omega;\mathbb S\cap\mathbb T)$ with $D\boldsymbol\tau=v_h$~\cite[(52)]{ArnoldHu2021}. We construct $\widehat{\boldsymbol\tau}_h$ by vertex, edge, and face data, matching the vertex derivatives of $D\widehat{\boldsymbol\tau}_h$ and $v_h$ through order $r_2^{\texttt{v}}-2$ and their edge normal derivatives through order $r_2^{e}-2$ (negative orders impose no conditions), and preserving
\[
\begin{aligned}
(\boldsymbol n_i^{\intercal}(\widehat{\boldsymbol\tau}_h-\boldsymbol\tau)\boldsymbol n_j,p)_e&=0,
&&p\in\mathbb P_2(e),\quad1\leq i\leq j\leq2,\\
(\boldsymbol n^{\intercal}(\widehat{\boldsymbol\tau}_h-\boldsymbol\tau)\boldsymbol n,1)_f&=0,\\
(\tr_2^{\div\div}(\widehat{\boldsymbol\tau}_h-\boldsymbol\tau),p)_f&=0,
&&p\in\mathbb P_2(f).
\end{aligned}
\]
For orthogonal unit vectors $\boldsymbol u,\boldsymbol v$, take coordinates $s,t,z$ along $\boldsymbol u,\boldsymbol v,\boldsymbol u\times\boldsymbol v$ and define by linearity
\[
\mathcal R_{uv}(s^it^jz^\ell)
=\frac{s^{i+1}t^{j+1}z^\ell}{(i+1)(j+1)}\sym(\boldsymbol u\boldsymbol v^{\intercal}).
\]
Then $D\mathcal R_{uv}p=p$, the tensor is traceless, and $\deg\mathcal R_{uv}p\leq\deg p+2$. At each vertex, center the coordinates there and use this formula to assign derivatives through order $r_2^{\texttt{v}}$, matching $D\widehat{\boldsymbol\tau}_h$ and $v_h$ through order $r_2^{\texttt{v}}-2$; take zero vertex data if $r_2^{\texttt{v}}<2$. Fix the low-order edge moments using $Q_e$, with the tangential diagonal component set to minus the normal-block trace. For $r_2^{e}\geq2$, take $\boldsymbol u=\boldsymbol n_1$, $\boldsymbol v=\boldsymbol n_2$, with $s=t=0$ on the edge. The formula assigns corrections to the edge normal derivatives of orders $2,\ldots,r_2^{e}$, matching the residual derivatives through order $r_2^{e}-2$ and preserving the edge moments. For a term $g_\beta(z)(s,t)^\beta$, vertex matching gives endpoint vanishing through order $r_2^{\texttt{v}}-2-|\beta|=r_2^{\texttt{v}}-(|\beta|+2)$, as required.

On each face, prescribe the constant normal--normal moment using $Q_f$ and the $\mathbb P_2(f)$ second-trace moments using~\eqref{eq:divdivSTdF3}.

Matching these boundary derivatives gives the residual
\[
r_T=(v_h-D\widehat{\boldsymbol\tau}_h)|_T\in\mathbb B_{k-2}(T;\boldsymbol r_2\ominus2).
\]
 For $q\in{\rm CH}$, $\dev\hess q=0$, $q|_e\in\mathbb P_2(e)$, $q|_f\in\mathbb P_2(f)$, and $\partial_nq|_f\in\mathbb P_0(f)$. The prescribed moments and Green's identity~\eqref{eq:greenidentitydivdiv} therefore give $r_T\perp{\rm CH}$. Lemma~\ref{lem:divdivontobubble} supplies $\boldsymbol b_T\in\mathbb B_k^{\div\div}(T;\boldsymbol r_2;\mathbb S\cap\mathbb T)$ with $D\boldsymbol b_T=r_T$. Extending the bubbles by zero, $\boldsymbol\tau_h=\widehat{\boldsymbol\tau}_h+\sum_T\boldsymbol b_T$ satisfies $D\boldsymbol\tau_h=v_h$.
\end{proof}

\begin{example}\label{ex:lower-vertex-divdiv}\rm
With $\boldsymbol r_3=-1$, the choice $\boldsymbol r_2=(1,-1,-1)^{\intercal}$ and $k\geq6$ recovers the $(\div\div;\mathbb S\cap\mathbb T)$ stability in~\cite{GuoHuLin2025,GuoHuLin2025a}.
Taking $\boldsymbol{r}_2=(0,-1,-1)^{\intercal}$, $\boldsymbol{r}_3=-1$, and $k\geq4$, we obtain the following $(\div\div; \mathbb S\cap\mathbb T)$ stability
\begin{equation*}
\begin{pmatrix}
	0\\
	-1\\
	-1
\end{pmatrix}
\xrightarrow{\div\div} 
\begin{pmatrix}
	 -1\\
	 -1\\
	 -1
\end{pmatrix} \to 0.
\end{equation*}
\end{example}

\subsection{$H(\sym\curl)$-conforming finite element}\label{subsec:symcurlSTfem}
Now we design $H(\sym\curl)$-conforming finite element space $\mathbb V_{k+1}^{\sym\curl}(\boldsymbol{r}_1;\mathbb S\cap\mathbb T)$.
When $\boldsymbol{r}_1\geq(2,1,0)^{\intercal}$, define $\mathbb V_{k+1}^{\sym\curl}(\boldsymbol{r}_1;\mathbb S\cap\mathbb T):=\mathbb V_{k+1}(\boldsymbol{r}_1)\otimes (\mathbb S\cap \mathbb T)$.
Then we consider the construction of $H(\sym\curl)$-conforming $\mathbb V_{k+1}^{\sym\curl}(\boldsymbol{r}_1;\mathbb S\cap\mathbb T)$ for $\boldsymbol{r}_1\geq(2,0,-1)^{\intercal}$ with $r_1^f=-1$.

Take $\mathbb P_{k+1}(T;\mathbb S\cap\mathbb T)$ as the space of shape functions. 
For $r_1^f=-1$, the DoFs are given by
\begin{subequations}\label{eq:symcurlSTdofs}
\begin{align}
\nabla^i\boldsymbol{\tau}(\texttt{v}), & \quad i=0,\ldots, r_1^{\texttt{v}}, \label{eq:3dCrsymcurlSTfemdofV1}\\
\int_e \frac{\partial^{j}\boldsymbol{\tau}}{\partial n_1^{i}\partial n_2^{j-i}}:\boldsymbol{q} \dd s, &\quad \boldsymbol{q}\in \mathbb B_{k+1-j}(e; r_1^{\texttt{v}} - j)\otimes (\mathbb S\cap\mathbb T), \label{eq:3dCrsymcurlSTfemdofE1}\\
\notag
&\quad 0\leq i\leq j\leq r_1^e, \\
\int_e \tr_{e}^{\sym\curl}(\boldsymbol{\tau})\,q \dd s, &\quad q\in \mathbb B_{k}(e;  r_1^{\texttt{v}}-1), \textrm{ if } r_1^{e}=0, \label{eq:3dCrsymcurlSTfemdofE2}\\
\int_e (\boldsymbol{n}_i^{\intercal}(\sym\curl\boldsymbol{\tau})\boldsymbol{n}_j)\,q \dd s, &\quad q\in \mathbb B_{k}(e;  r_2^{\texttt{v}}), 1\leq i\leq j\leq 2, \textrm{ if } r_2^{e}=-1, \label{eq:3dCrsymcurlSTfemdofE3}\\
\int_f (\boldsymbol{n}\cdot\boldsymbol{\tau}\times\boldsymbol{n})\cdot\boldsymbol{q} \dd S, &\quad \boldsymbol{q}\in \curl_f\mathbb B_{k+2}(f;\boldsymbol r_1+1), \label{eq:3dCrsymcurlSTfemdofF1}\\
\int_f \div_f(\boldsymbol{n}\cdot\boldsymbol{\tau}\times\boldsymbol{n})\,q \dd S, &\quad q\in \mathbb B_{k}(f;(\boldsymbol r_2)_+)/\mathbb R, \label{eq:3dCrsymcurlSTfemdofF2}\\
\int_f \Pi_f\sym(\boldsymbol{\tau}\times\boldsymbol{n})\Pi_f:\boldsymbol{q} \dd S, &\quad \boldsymbol{q}\in \sym\curl_f\grad_f\mathbb B_{k+3}(f;\boldsymbol r_0), \label{eq:3dCrsymcurlSTfemdofF3}\\
\int_f \div_f\div_f(\sym(\boldsymbol{\tau}\times\boldsymbol{n}))\,q \dd S, &\quad q\in \mathbb B_{k-1}(f;\boldsymbol r_2\ominus 1)/{\rm CH}(f), \label{eq:3dCrsymcurlSTfemdofF4}\\
\int_T \boldsymbol{\tau}:\boldsymbol{q} \dx, &\quad \boldsymbol{q}\in \mathbb B_{k+1}^{\sym\curl}(\boldsymbol{r}_1;\mathbb S\cap\mathbb T), \label{eq:3dCrsymcurlSTfemdofT1}
\end{align}
\end{subequations}
for each $\texttt{v}\in \Delta_{0}(T)$, $e\in \Delta_{1}(T)$ and $f\in \Delta_{2}(T)$.

\begin{lemma}\label{lem:symcurlSTfemunisol}
Let smoothness vectors $(\boldsymbol{r}_0, \boldsymbol{r}_1, \boldsymbol{r}_2)$ satisfy \eqref{eq:rCH} with $r_1^f=-1$, and $k\geq 2r_2^{\texttt{v}}+4$. DoFs \eqref{eq:symcurlSTdofs} are unisolvent for $\mathbb P_{k+1}(T;\mathbb S\cap\mathbb T)$.
\end{lemma}
\begin{proof}
By complex \eqref{eq:femderhambubblecomplex2d} and complex \eqref{eq:femCdivdivbubblecomplex2d},
the number of DoFs \eqref{eq:3dCrsymcurlSTfemdofF1}-\eqref{eq:3dCrsymcurlSTfemdofF4} is 
\begin{equation*}
4\dim\mathbb B^{\div_f}_{k+1}(\boldsymbol{r}_1;\mathbb R^2) + 4\dim\mathbb B^{\div_f\div_f}_{k+1}(\boldsymbol{r}_1;\mathbb S\cap \mathbb T) - 12[r_2^e=-1]\dim\mathbb B_{k}(e;  r_2^{\texttt{v}}).
\end{equation*}
Hence, 
the number of DoFs \eqref{eq:3dCrsymcurlSTfemdofE2}-\eqref{eq:3dCrsymcurlSTfemdofF4} is 
\begin{equation*}
8\dim\mathbb B_{k+1}(f;\boldsymbol{r}_1) + 4\dim\mathbb B^{\div_f\div_f}_{k+1}(\boldsymbol{r}_1;\mathbb S\cap \mathbb T) + 12[r_2^e=-1]\dim\mathbb B_{k}(e;  r_2^{\texttt{v}}),
\end{equation*}
by \eqref{eq:20240317}, which equals $16\dim\mathbb B_{k+1}(f;\boldsymbol{r}_1)$.
Then we get from \eqref{eq:symcurlSTbubbledecomp} that the number of DoFs~\eqref{eq:symcurlSTdofs} equals $\dim\mathbb P_{k+1}(T;\mathbb S\cap\mathbb T)$.

Take $\boldsymbol{\tau}\in\mathbb P_{k+1}(T;\mathbb S\cap\mathbb T)$, and assume all the DoFs \eqref{eq:symcurlSTdofs} vanish. 
Following the proof of Lemma 6.2 in \cite{ChenHuang2022}, by \eqref{eq:edgedofprop1}, we get from the vanishing DoFs \eqref{eq:3dCrsymcurlSTfemdofV1}-\eqref{eq:3dCrsymcurlSTfemdofE3} and \eqref{eq:3dCrsymcurlSTfemdofF4} that $\Pi_f\sym(\boldsymbol{\tau}\times\boldsymbol{n})\Pi_f\in\mathbb B^{\div_f\div_f}_{k+1}(\boldsymbol{r}_1;\mathbb S\cap \mathbb T)\cap\ker(\div_f\div_f)$, and from the vanishing DoFs \eqref{eq:3dCrsymcurlSTfemdofV1}-\eqref{eq:3dCrsymcurlSTfemdofE1}, \eqref{eq:3dCrsymcurlSTfemdofE3} and \eqref{eq:3dCrsymcurlSTfemdofF2} that $(\boldsymbol{n}\cdot\boldsymbol{\tau}\times\boldsymbol{n})|_f\in\mathbb B^{\div_f}_{k+1}(\boldsymbol{r}_1;\mathbb R^2)\cap\ker(\div_f)$.
Using bubble complex \eqref{eq:femderhambubblecomplex2d} and bubble complex \eqref{eq:femCdivdivbubblecomplex2d}, the vanishing DoFs \eqref{eq:3dCrsymcurlSTfemdofF1} and \eqref{eq:3dCrsymcurlSTfemdofF3} imply that $(\boldsymbol{n}\cdot\boldsymbol{\tau}\times\boldsymbol{n})|_{\partial T}=0$ and $\Pi_f\sym(\boldsymbol{\tau}\times\boldsymbol{n})\Pi_f=0$.
That is $\boldsymbol{\tau}\in\mathbb B_{k+1}^{\sym\curl}(\boldsymbol{r}_1;\mathbb S\cap\mathbb T)$, which together with the vanishing DoF \eqref{eq:3dCrsymcurlSTfemdofT1} gives $\boldsymbol{\tau}=0$.
\end{proof}

For $\boldsymbol{r}_1\geq(2,0,-1)^{\intercal}$ with $r_1^f=-1$, the finite element space $\mathbb V_{k+1}^{\sym\curl}(\boldsymbol{r}_1;\mathbb S\cap\mathbb T)$ is defined as follows
\begin{align*}
\mathbb V_{k+1}^{\sym\curl}(\boldsymbol{r}_1;\mathbb S\cap\mathbb T):=\{\boldsymbol \tau\in L^2(\Omega;\mathbb S\cap \mathbb T): &\, \boldsymbol \tau|_T\in\mathbb P_{k+1}(T;\mathbb S\cap\mathbb T) \textrm{ for each } T\in\mathcal T_h, \\
&\qquad\qquad\;\;\; \textrm{ DoFs \eqref{eq:symcurlSTdofs} are single-valued} \}.    
\end{align*}
According to the proof of Lemma~\ref{lem:symcurlSTfemunisol}, $(\boldsymbol{n}\cdot\boldsymbol{\tau}\times\boldsymbol{n})|_{f}$ and $\Pi_f\sym(\boldsymbol{\tau}\times\boldsymbol{n})\Pi_f$ are determined by DoFs \eqref{eq:3dCrsymcurlSTfemdofV1}-\eqref{eq:3dCrsymcurlSTfemdofF4} on face $f$, so $\mathbb V_{k+1}^{\sym\curl}(\boldsymbol{r}_1;\mathbb S\cap\mathbb T)\subset H(\sym\curl, \Omega; \mathbb S\cap \mathbb T)$.

\subsection{Finite element conformal Hessian complex}

\begin{theorem}\label{thm:femCHcomplex3d}
Let smoothness vectors $(\boldsymbol r_0, \boldsymbol r_1, \boldsymbol r_2, \boldsymbol r_3)$ be given by~\eqref{eq:rCH}, and $k\geq 2r_2^{\texttt{v}}+4$.
Then the finite element conformal Hessian complex \eqref{femCHcomplex3d} is exact.
\end{theorem}
\begin{proof}
The trace conditions make \eqref{femCHcomplex3d} a complex. A tensor in the $\sym\curl$ kernel has a piecewise polynomial $H^2$ potential $u$ by continuous exactness~\eqref{conformalHesscomplex3d} and polynomial exactness~\eqref{polyconformalHesscomplex3d}. For $\boldsymbol\tau=\dev\hess u$, a face normal $\boldsymbol n$ and a unit tangent $\boldsymbol t$ give
\[
\partial_n^2u=\partial_t^2u+\boldsymbol n^{\intercal}\boldsymbol\tau\boldsymbol n-\boldsymbol t^{\intercal}\boldsymbol\tau\boldsymbol t.
\]
Starting from the single-valued traces $u$ and $\partial_nu$, successive differentiation gives the prescribed vertex, edge, and face continuity through orders $\boldsymbol r_0=\boldsymbol r_1+2$. Thus $u\in\mathbb V_{k+3}(\boldsymbol r_0)$, proving the kernel identity. By~\eqref{eq:divdivSTstability}, it remains to prove
\begin{align}
\label{eq:CHfemcomplexdim}
\dim\mathbb V_{k+3}(\boldsymbol{r}_0)&-\dim\mathbb V_{k+1}^{\sym\curl}(\boldsymbol{r}_1;\mathbb S\cap\mathbb T)\\
\notag
& +\dim\mathbb V_{k}^{\div\div}(\boldsymbol{r}_2;\mathbb S\cap\mathbb T)-\dim\mathbb V_{k-2}(\boldsymbol{r}_3)=5.
\end{align}

Let $C_i$, $i=0,1,2,3$, denote the alternating entitywise DoF count in~\eqref{eq:CHfemcomplexdim} on one $i$-dimensional simplex. Then \eqref{eq:CHfemcomplexdim} is equivalent to
\begin{equation}\label{eq:CHfemcomplexdimC}
C_0|\Delta_0(\mathcal{T}_h)|+C_1|\Delta_1(\mathcal{T}_h)|+C_2|\Delta_2(\mathcal{T}_h)|+C_3|\Delta_3(\mathcal{T}_h)|=5.
\end{equation}
By the bubble complex \eqref{polybubbleCHcomplex3d}, $C_3=-5$.
We compute $C_0$ and $C_1$ as follows
\begin{align*}
C_0&={r_0^{\texttt{v}}+3\choose3}-5{r_1^{\texttt{v}}+3\choose3}+5{r_2^{\texttt{v}}+3\choose3}-{r_3^{\texttt{v}}+3\choose3}=5, \\
C_1&= \frac{1}{6}(r_0^e+1)(r_0^e+2)(3k+2r_0^e-6r_0^{\texttt{v}}+6)-\frac{5}{6}(r_1^e+1)(r_1^e+2)(3k+2r_1^e-6r_1^{\texttt{v}}) \\
&\quad-[r_1^e=0](k-2r_1^{\texttt{v}}+1)+\frac{5}{6}(r_2^e+1)(r_2^e+2)(3k+2r_2^e-6r_2^{\texttt{v}}-3) \\
&\quad -\frac{1}{6}(r_3^e+1)(r_3^e+2)(3k+2r_3^e-6r_3^{\texttt{v}}-9) =-5.
\end{align*}
Notice that
\begin{align*}
4C_0+6C_1+4C_2+C_3&=\dim\mathbb P_{k+3}(T)-\dim\mathbb P_{k+1}(T;\mathbb S\cap\mathbb T)\\
&\quad +\dim\mathbb P_{k}(T;\mathbb S\cap\mathbb T)-\dim\mathbb P_{k-2}(T)=5,
\end{align*}
in which we have used the polynomial conformal Hessian complex \eqref{polyconformalHesscomplex3d}. Hence, $C_2=5$.
Therefore, \eqref{eq:CHfemcomplexdimC} follows from Euler's formula \eqref{eulerformula}.
\end{proof}

\begin{example}\rm 
Taking $\boldsymbol{r}_0=(4,2,1)^{\intercal}$, $\boldsymbol{r}_1=\boldsymbol{r}_0-2$, $\boldsymbol{r}_2=\boldsymbol{r}_1\ominus1$, $\boldsymbol{r}_3=\boldsymbol{r}_2\ominus2$, and $k\geq6$, we obtain the finite element conformal Hessian complex
\begin{equation*}
{\rm CH} \hookrightarrow
\begin{pmatrix}
	 4\\
	 2\\
	 1
\end{pmatrix}
\xrightarrow{\dev\hess}
\begin{pmatrix}
 2\\
 0\\
 -1
\end{pmatrix}
\xrightarrow{\sym\curl}
\begin{pmatrix}
	1\\
	-1\\
	-1
\end{pmatrix}
\xrightarrow{\div\div} 
\begin{pmatrix}
	 -1\\
	 -1\\
	 -1
\end{pmatrix} \to 0,
\end{equation*}
which was recently constructed in~\cite{GuoHuLin2025,GuoHuLin2025a}.\end{example}

\begin{remark}\label{rem:femCHtilde}\rm
Let $\boldsymbol{r}_0, \boldsymbol{r}_1, \boldsymbol{r}_2$ and $\boldsymbol{r}_3$ be smoothness vectors satisfying
\begin{equation*}
\boldsymbol{r}_0\geq (4,2,1)^{\intercal}, \quad
\boldsymbol r_1=\boldsymbol r_0-2  \geq (2,0,-1)^{\intercal},
\quad
 \boldsymbol{r}_2\geq\boldsymbol{r}_1\ominus1,
\quad
 \boldsymbol{r}_3\geq\boldsymbol{r}_2\ominus2.
\end{equation*}
Assume additionally that $\boldsymbol r_2$ satisfies the div\,div stability condition. Lemma~\ref{lem:global-terminal-divdiv} supplies the terminal surjectivity, so the tilde operation in \cite{ChenHuang2025,ChenHuang2024}, applied to complex \eqref{femCHcomplex3d}, gives the following exact complex for sufficiently large $k$:
\begin{align*}
{\rm CH}\xrightarrow{\subset} \mathbb V_{k+3}(\boldsymbol{r}_0)&\xrightarrow{\dev\hess} \mathbb V_{k+1}^{\sym\curl}(\boldsymbol{r}_1,\boldsymbol{r}_2;\mathbb S\cap\mathbb T)\xrightarrow{\sym\curl} \mathbb V_{k}^{\div\div}(\boldsymbol{r}_2,\boldsymbol{r}_3;\mathbb S\cap\mathbb T) \\
\notag
&\xrightarrow{\div{\div}} \mathbb V_{k-2}(\boldsymbol{r}_3)\xrightarrow{}0,
\end{align*}
where
\begin{align*}
\mathbb V_{k+1}^{\sym\curl}(\boldsymbol{r}_1,\boldsymbol{r}_2;\mathbb S\cap\mathbb T)&:=\{\boldsymbol{\tau}\in\mathbb V_{k+1}^{\sym\curl}(\boldsymbol{r}_1;\mathbb S\cap\mathbb T): \sym\curl\boldsymbol{\tau}\in\mathbb V_{k}^{\div\div}(\boldsymbol{r}_2;\mathbb S\cap\mathbb T)\}, \\
\mathbb V_{k}^{\div\div}(\boldsymbol{r}_2,\boldsymbol{r}_3;\mathbb S\cap\mathbb T)&:=\{\boldsymbol{\tau}\in\mathbb V_{k}^{\div\div}(\boldsymbol{r}_2;\mathbb S\cap\mathbb T): \div\div\boldsymbol{\tau}\in\mathbb V_{k-2}(\boldsymbol{r}_3)\}.
\end{align*}
\end{remark}

\section{Finite element conformal elasticity complex in three dimensions}\label{sec:FEconformalelascomplex}

In this section, we will construct the finite element conformal elasticity complex
\begin{align}\label{femCEcomplex3d}
{\rm CK}\xrightarrow{\subset} \mathbb V_{k+4}(\boldsymbol{r}_0;\mathbb R^3)&\xrightarrow{\dev\defm} \mathbb V_{k+3}^{\cott}(\boldsymbol{r}_1;\mathbb S\cap\mathbb T)\xrightarrow{\cott} \mathbb V_{k}^{\div}(\boldsymbol{r}_2;\mathbb S\cap\mathbb T) \\
\notag
&\xrightarrow{\div} \mathbb V_{k-1}(\boldsymbol{r}_3;\mathbb R^3)\xrightarrow{}0
\end{align}
with $k\geq 2r_2^{\texttt{v}}+5$ and smoothness vectors $(\boldsymbol{r}_0, \boldsymbol{r}_1, \boldsymbol{r}_2, \boldsymbol{r}_3)$ satisfying \eqref{eq:rCE}.
Here spaces $\mathbb V_{k+4}(\boldsymbol{r}_0;\mathbb R^3)=\mathbb V_{k+4}(\boldsymbol{r}_0)\otimes\mathbb R^3$ and $\mathbb V_{k-1}(\boldsymbol{r}_3;\mathbb R^3)=\mathbb V_{k-1}(\boldsymbol{r}_3)\otimes\mathbb R^3$ are the smoothness vector finite element spaces defined as~\eqref{eq:scalarfem}.
We will define space $\mathbb V_{k}^{\div}(\boldsymbol{r}_2;\mathbb S\cap\mathbb T)$ in Section~\ref{subsec:divSTfem} and space $\mathbb V_{k+3}^{\cott}(\boldsymbol{r}_1;\mathbb S\cap\mathbb T)$ in Section~\ref{subsec:cottSTfem}.

\subsection{$H(\div)$-conforming finite element}\label{subsec:divSTfem}
Define $\mathbb V_{k}^{\div}(\boldsymbol{r}_2;\mathbb S\cap \mathbb T):=\mathbb V_{k}(\boldsymbol{r}_2)\otimes (\mathbb S\cap \mathbb T)$ for $\boldsymbol{r}_2\geq0$.
Then we focus on the construction of $H(\div)$-conforming $\mathbb V_{k}^{\div}(\boldsymbol{r}_2;\mathbb S\cap \mathbb T)$ for $\boldsymbol{r}_2\geq(3,-1,-1)^{\intercal}$ with $r_2^f=-1$.

Take $\mathbb P_{k}(T;\mathbb S\cap \mathbb T)$ as the space of shape functions. 
For $r_2^f=-1$,
the DoFs are 
\begin{subequations}\label{eq:divSTdof}
\begin{align}
\nabla^i\boldsymbol{\tau}(\texttt{v}), & \quad i=0,\ldots, r_2^{\texttt{v}}, \texttt{v}\in \Delta_{0}(T), \label{eq:3dCrdivSTfemdofV}\\
\int_e \frac{\partial^{j}\boldsymbol{\tau}}{\partial n_1^{i}\partial n_2^{j-i}}:\boldsymbol{q} \dd s, &\quad \boldsymbol{q}\in \mathbb P_{k - 2(r_2^{\texttt{v}}+1) + j}(e;\mathbb S\cap\mathbb T), 0\leq i\leq j\leq r_2^{e}, e\in \Delta_1(T), \label{eq:3dCrdivSTfemdofE1}\\
\int_e (\boldsymbol{n}_i^{\intercal}\boldsymbol{\tau}\boldsymbol{n}_j)\,q \dd s, &\quad q\in \mathbb B_{k}(e; r_2^{\texttt{v}}), 1\leq i\leq j\leq 2, e\in \Delta_{1}(T), \textrm{ if } r_2^{e}=-1, \label{eq:3dCrdivSTfemdofE2}\\
\int_f (\boldsymbol n^{\intercal}\boldsymbol{\tau}\boldsymbol{n}) \, q \dd S, &\quad q\in \mathbb B_{k}(f;(\boldsymbol{r}_2)_+), f\in \Delta_{2}(T), \label{eq:3dCrdivSTfemdofF1}\\
\int_f (\Pi_f \boldsymbol{\tau}\boldsymbol{n})\cdot\boldsymbol{q} \dd S, &\quad \boldsymbol{q}\in \mathbb B_{k}^{\div} (f;\boldsymbol r_2), f\in \Delta_{2}(T), \label{eq:3dCrdivSTfemdofF2}\\
\int_T \boldsymbol{\tau}:\boldsymbol{q} \dx, &\quad \boldsymbol{q}\in \mathbb B_{k}^{\div}(\boldsymbol{r}_2;\mathbb S\cap\mathbb T). \label{eq:3dCrdivSTfemdofT}
\end{align}
\end{subequations}

\begin{lemma}\label{lm:divST}
Let $\boldsymbol r_2$ be a smoothness vector with $r_2^f = -1, r_2^{\texttt{v}}\geq 0$, and let $k\geq 2r_2^{\texttt{v}}+1$. DoFs~\eqref{eq:divSTdof} are unisolvent for $\mathbb P_{k}(T;\mathbb S\cap\mathbb T)$. 
\end{lemma}
\begin{proof}
For $\boldsymbol{\tau}\in\mathbb B_{k}^{\div}(\boldsymbol{r}_2;\mathbb S\cap\mathbb T)$, $\boldsymbol{\tau}|_e=0$ for $e\in\Delta_1(T)$. This means
\begin{align*}
\mathbb B_{k}^{\div}(\boldsymbol{r}_2;\mathbb S\cap\mathbb T)&=\mathbb B_{k}(T;(\boldsymbol{r}_2)_+)\otimes(\mathbb S\cap\mathbb T)\\
&\quad\oplus\Oplus_{f\in\Delta_2(T)}\mathbb B_{k}(f;(\boldsymbol{r}_2)_+)\otimes{\rm span}\{\sym(\boldsymbol{t}_1\boldsymbol{t}_2^{\intercal}), \boldsymbol{t}_1\boldsymbol{t}_1^{\intercal}-\boldsymbol{t}_2\boldsymbol{t}_2^{\intercal}\}.
\end{align*}
On the other side,
\begin{align*}
\mathbb B_{k}^{\div} (f;\boldsymbol r_2)&=\mathbb B_{k}^2(f;(\boldsymbol{r}_2)_+) \oplus[r_2^e=-1]\Oplus_{e\in\Delta_1(f)}\big(\mathbb B_{k}(e;r_2^{\texttt{v}})\otimes{\rm span}\{\boldsymbol{t}_e\}\big).
\end{align*}
We conclude the result from the last two decompositions and the unisolvence of DoFs \eqref{eq:Cr3D} for $\mathbb P_k(T)$.
\end{proof}

In the terminal parameter range of Lemma~\ref{lem:global-terminal-div}, the face moments \eqref{eq:3dCrdivSTfemdofF1}--\eqref{eq:3dCrdivSTfemdofF2} can be equivalently replaced by
\begin{align*}
\int_f (\boldsymbol n^{\intercal}\boldsymbol{\tau}\boldsymbol{n}) \, q \dd S, &\quad q\in (\mathbb B_{k}(f;(\boldsymbol{r}_2)_+)/{\rm CH}(f))\oplus {\rm CH}(f), f\in \Delta_{2}(T), \\
\int_f (\Pi_f \boldsymbol{\tau}\boldsymbol{n})\cdot\boldsymbol{q} \dd S, &\quad \boldsymbol{q}\in (\mathbb B_{k}^{\div} (f;\boldsymbol r_2)/{\rm RM}^+(f)) \oplus {\rm RM}^+(f), f\in \Delta_{2}(T), 
\end{align*}
where ${\rm RM}^+(f):=\{\Pi_f\boldsymbol v:\boldsymbol v\in{\rm CK}\}\subseteq\Pi_f\mathbb P_2(f;\mathbb R^3)$.
To justify the replacement, take a balanced triangular Bernstein monomial $w$ of degree $2r_2^{\texttt{v}}$. The parameter inequalities and $k\geq2r_2^{\texttt{v}}+2$ give $w{\rm CH}(f)\subseteq\mathbb B_k(f;(\boldsymbol r_2)_+)$ and $w{\rm RM}^+(f)\subseteq\mathbb B_k^{\div}(f;\boldsymbol r_2)$. The weighted Gram matrices on ${\rm CH}(f)$ and ${\rm RM}^+(f)$ are positive definite, proving equivalence of the moments.

For $r_2^f=-1$, the finite element space $\mathbb V_{k}^{\div}(\boldsymbol{r}_2;\mathbb S\cap \mathbb T)$ is defined as follows
\begin{align*}
\mathbb V^{\div}_{k}(\boldsymbol{r}_2;\mathbb S\cap \mathbb T):=\{\boldsymbol{\tau}\in L^2(\Omega;\mathbb S\cap\mathbb T)&: \boldsymbol{\tau}|_T\in\mathbb P_{k}(T;\mathbb S\cap \mathbb T) \textrm{ for all } T\in\mathcal T_h, \\
&\;\; \textrm{ and all the DoFs \eqref{eq:divSTdof} are single-valued}\}.
\end{align*}
The single-valued DoFs in \eqref{eq:3dCrdivSTfemdofV}-\eqref{eq:3dCrdivSTfemdofF2} imply $\mathbb V_{k}^{\div}(\boldsymbol{r}_2;\mathbb S\cap \mathbb T)\subset H(\div, \Omega; \mathbb S\cap \mathbb T)$.

\begin{lemma}\label{lem:global-terminal-div}
Let $\boldsymbol r_1\geq(6,0,-1)^{\intercal}$, $\boldsymbol r_2=\boldsymbol r_1\ominus3\geq(3,-1,-1)^{\intercal}$ and $\boldsymbol r_3=\boldsymbol{r}_2\ominus 1$ be smoothness vectors, and $k\geq 2r_2^{\texttt{v}}+2$.
Assume both $\boldsymbol r_1 \ominus 2$ and $\boldsymbol r_2$ satisfy the div-tensor stability condition.
Then we have the $(\div; \mathbb S\cap\mathbb T)$ stability
\begin{equation}\label{eq:divSTstability}
\div \mathbb V^{\div}_{k}(\boldsymbol{r}_2;\mathbb S\cap\mathbb T) = 
\mathbb V_{k-1}(\boldsymbol{r}_3;\mathbb R^3).
\end{equation}
\end{lemma}
\begin{proof}
The forward inclusion follows from the defining smoothness and traces. Given $\boldsymbol v\in\mathbb V_{k-1}(\boldsymbol r_3;\mathbb R^3)$, take $\boldsymbol\tau\in H^1(\Omega;\mathbb S\cap\mathbb T)$ with $\div\boldsymbol\tau=\boldsymbol v$~\cite[(50)]{ArnoldHu2021}. Use~\eqref{eq:divSTdof} for $r_2^{f}=-1$ and the tensor product of~\eqref{eq:Cr3D} for $r_2^{f}\geq0$. With the lower-dimensional derivative data fixed, the equivalent face moments above include ${\rm CH}(f)$ in the normal component and ${\rm RM}^+(f)$ in the tangential component. For $r_2^{f}\geq0$, the same positive weighted pairings apply to the corresponding components of the full face moments.

\enlargethispage{2\baselineskip}
Construct $\boldsymbol\rho_h$ so that the derivatives of $\div\boldsymbol\rho_h$ and $\boldsymbol v$ agree through order $r_2^{\texttt{v}}-1$ at vertices, and their normal derivatives agree through orders $r_2^{e}-1$ on edges and $r_2^{f}-1$ on faces (negative orders impose no conditions), while preserving
\[
\begin{aligned}
(\boldsymbol n^{\intercal}(\boldsymbol\rho_h-\boldsymbol\tau)\boldsymbol n,q)_f&=0,
&&q\in{\rm CH}(f),\\
(\Pi_f(\boldsymbol\rho_h-\boldsymbol\tau)\boldsymbol n,\boldsymbol q)_f&=0,
&&\boldsymbol q\in{\rm RM}^+(f).
\end{aligned}
\]
For a unit vector $\boldsymbol n$, define the traceless normal right inverse
\[
J_n\boldsymbol z=\boldsymbol n\boldsymbol z^{\intercal}+\boldsymbol z\boldsymbol n^{\intercal}
-\tfrac12(\boldsymbol z\cdot\boldsymbol n)(I+\boldsymbol n\boldsymbol n^{\intercal}),
\qquad (J_n\boldsymbol z)\boldsymbol n=\boldsymbol z.
\]
Polynomial divergence liftings of degree one higher prescribe the vertex derivatives. Along each edge, the leading transverse divergence maps symmetric traceless tensor polynomials of normal degree $j+1$ onto vector polynomials of normal degree $j$, as can be verified by direct integration in the two edge-normal coordinates. Successive corrections therefore match the required edge derivatives, with endpoint compatibility as in Lemma~\ref{lem:global-terminal-divdiv}. At the face stage, first fix the traction moments using the equivalent face DoFs. For $r_2^{f}\geq1$, prescribe the normal derivatives of the face correction by
\[
\delta\boldsymbol\tau_0=0,\quad
\delta\boldsymbol\tau_{j+1}=J_n(g_j-\div_f\delta\boldsymbol\tau_j),\quad
g_j=\partial_n^j(\boldsymbol v-\div\boldsymbol\rho_h)|_f,\quad0\leq j\leq r_2^{f}-1.
\]
These corrections match the required derivatives of the divergence and preserve the fixed traction moments, since the face values do not change.
The vertex and edge vanishing orders of the residual ensure that these face corrections preserve all previously prescribed lower-dimensional derivatives.

The resulting residual $r_T=(\boldsymbol v-\div\boldsymbol\rho_h)|_T$ belongs to $\mathbb B_{k-1}(\boldsymbol r_3;\mathbb R^3)$. For $\boldsymbol q\in{\rm CK}$, $\dev\defm\boldsymbol q=0$, while its normal and tangential face components belong to ${\rm CH}(f)$ and ${\rm RM}^+(f)$. Green's formula and the prescribed moments therefore give $r_T\perp{\rm CK}$. Apply Lemma~\ref{lem:terminal-bubble-divST} for $r_2^{e}\geq1$; for $r_2^{e}\leq0$, use the terminal argument of Lemma~\ref{lem:bubbleCEcomplex3d} with $\boldsymbol r_1'=(r_2^{\texttt{v}}+3,3,1)^{\intercal}$, including the auxiliary endpoint when $r_2^{\texttt{v}}=3$. The divergence bubble spaces for $r_2^{e}=-1$ and $r_2^{e}=0$ coincide, since zero face traction forces edge vanishing. This supplies $\boldsymbol b_T$ with $\div\boldsymbol b_T=r_T$. Extending these bubbles by zero gives $\boldsymbol\tau_h=\boldsymbol\rho_h+\sum_T\boldsymbol b_T$ with $\div\boldsymbol\tau_h=\boldsymbol v$.
\end{proof}

\begin{example}\rm 
The space $\mathbb V_k^{\div}(\boldsymbol r_2; \mathbb S\cap\mathbb T)$ and the $(\div; \mathbb S\cap\mathbb T)$ stability \eqref{eq:divSTstability} for $\boldsymbol r_2 = (3,0,-1)^{\intercal}$ and $k\geq 8$ have been constructed recently in~\cite{HuLinShi2023}.
We also devise the less smooth space $\mathbb V_k^{\div}(\boldsymbol r_2; \mathbb S\cap\mathbb T)$ and the $(\div; \mathbb S\cap\mathbb T)$ stability \eqref{eq:divSTstability} for $\boldsymbol r_2 = (3,-1,-1)^{\intercal}$ and $k\geq 8$, for which only the normal--normal edge block is prescribed to be single-valued, whereas full edge continuity is imposed in~\cite{HuLinShi2023}.
\end{example}

\subsection{$H(\cott)$-conforming finite element}\label{subsec:cottSTfem}
When $\boldsymbol{r}_1\geq(8,4,2)^{\intercal}$, define $\mathbb V_{k+3}^{\cott}(\boldsymbol{r}_1;\mathbb S\cap\mathbb T):=\mathbb V_{k+3}(\boldsymbol{r}_1)\otimes (\mathbb S\cap \mathbb T)$.
Then we consider the construction of $H(\cott)$-conforming $\mathbb V_{k+3}^{\cott}(\boldsymbol{r}_1;\mathbb S\cap\mathbb T)$ for $\boldsymbol{r}_1\geq(6,0,-1)^{\intercal}$ with $r_1^f=-1,0,1$.

Take $\mathbb P_{k+3}(T;\mathbb S\cap\mathbb T)$ as the space of shape functions. 
For $r_1^f=-1,0,1$, the DoFs are given by
\begin{subequations}\label{eq:cottSTdofs}
\begin{align}
\nabla^i\boldsymbol{\tau}(\texttt{v}), & \quad i=0,\ldots, r_1^{\texttt{v}}, \label{eq:3dCrcottSTfemdofV1}\\
\int_e \frac{\partial^{j}\boldsymbol{\tau}}{\partial n_1^{i}\partial n_2^{j-i}}:\boldsymbol{q} \dd s, &\quad \boldsymbol{q}\in \mathbb B_{k+3-j}(e; r_1^{\texttt{v}} - j)\otimes (\mathbb S\cap\mathbb T), \label{eq:3dCrcottSTfemdofE1}\\
\notag
&\quad 0\leq i\leq j\leq r_1^e, \\
\int_e \tr_{e,1}^{\cott}(\boldsymbol{\tau})\,q \dd s, &\quad q\in \mathbb B_{k+2}(e;  r_1^{\texttt{v}}-1), \textrm{ if } r_1^{e}=0, \label{eq:3dCrcottSTfemdofE2}\\
\int_e \tr_{e,2}^{\cott}(\boldsymbol{\tau})\cdot\boldsymbol{n}_i\,q \dd s, &\quad q\in \mathbb B_{k+1}(e;  r_1^{\texttt{v}}-2), i=1,2, \textrm{ if } r_1^{e}=0,1, \label{eq:3dCrcottSTfemdofE3}\\
\int_e \tr_{e,3}^{\cott}(\boldsymbol{\tau})\,q \dd s, &\quad q\in \mathbb B_{k+1}(e;  r_1^{\texttt{v}}-2), \textrm{ if } r_1^{e}=0,1, \label{eq:3dCrcottSTfemdofE4}\\
\int_e (\boldsymbol{n}_i^{\intercal}(\cott\boldsymbol{\tau})\boldsymbol{n}_j)\,q \dd s, &\quad q\in \mathbb B_{k}(e;  r_2^{\texttt{v}}), 1\leq i\leq j\leq 2, \textrm{ if } r_2^{e}=-1, \label{eq:3dCrcottSTfemdofE5}\\
\int_f \frac{\partial^{j}\boldsymbol{\tau}}{\partial n^{j}}:\boldsymbol{q} \dd S, &\quad \boldsymbol{q}\in \mathbb B_{k+3-j}(f; \boldsymbol{r}_1 - j)\otimes (\mathbb S\cap\mathbb T), 0\leq j\leq r_1^f,\label{eq:3dCrcottSTfemdofF1}\\
\int_f \tr_{1}^{\cott}(\boldsymbol{\tau}):\boldsymbol{q} \dd S, &\quad \boldsymbol{q}\in \mathbb B_{k+3}(f;\boldsymbol r_1)\otimes (\mathbb S(f)\cap\mathbb T(f)), \textrm{ if } r_1^{f}=-1, \label{eq:3dCrcottSTfemdofF2}\\
\int_f \tr_{2}^{\cott}(\boldsymbol{\tau}):\boldsymbol{q} \dd S, &\quad \boldsymbol{q}\in \sym\curl_f\grad_f\mathbb B_{k+4}(f;\boldsymbol r_0), \textrm{ if } r_1^{f}=-1,0, \label{eq:3dCrcottSTfemdofF3}\\
\int_f \div_f\div_f\tr_{2}^{\cott}(\boldsymbol{\tau})\,q \dd S, &\quad q\in \mathbb B_{k}(f;(\boldsymbol{r}_2)_+)/{\rm CH}(f), \textrm{ if } r_1^{f}=-1,0, \label{eq:3dCrcottSTfemdofF4}\\
\int_f \tr_{3}^{\cott}(\boldsymbol{\tau}):\boldsymbol{q} \dd S, &\quad \boldsymbol{q}\in 
\mathbb B_{k+1}^{\rot_f\rot_f}((\boldsymbol r_1\ominus2)_+;\mathbb S\cap\mathbb T), \label{eq:3dCrcottSTfemdofF5}\\
\int_T \boldsymbol{\tau}:\boldsymbol{q} \dx, &\quad \boldsymbol{q}\in \mathbb B_{k+3}^{\cott}(\boldsymbol{r}_1;\mathbb S\cap\mathbb T), \label{eq:3dCrcottSTfemdofT1}
\end{align}
\end{subequations}
for each $\texttt{v}\in \Delta_{0}(T)$, $e\in \Delta_{1}(T)$ and $f\in \Delta_{2}(T)$, where
\begin{equation*}
\mathbb B_{k+1}^{\rot_f\rot_f}(\boldsymbol r; \mathbb S\cap\mathbb T):= \{\boldsymbol{\tau}\in\mathbb B_{k+1}(f;\boldsymbol r_+)\otimes (\mathbb S(f)\cap\mathbb T(f)): (\boldsymbol{t}\cdot\rot_f\boldsymbol{\tau})|_{\partial f}=0\}.
\end{equation*}
Here $\rot_f$ is applied row-wise to tensor fields.
We can see that $\mathbb B_{k+1}^{\rot_f\rot_f}(\boldsymbol r; \mathbb S\cap\mathbb T)$ is just the rotation of $\mathbb B_{k+1}^{\div_f\div_f}(\boldsymbol r_+; \mathbb S\cap\mathbb T)$.

\begin{lemma}\label{lem:cottSTfemunisol}
Let smoothness vectors $(\boldsymbol{r}_0, \boldsymbol{r}_1, \boldsymbol{r}_2)$ satisfy \eqref{eq:rCE} with $r_1^f=-1,0,1$, and $k\geq 2r_2^{\texttt{v}}+5$. 
DoFs \eqref{eq:cottSTdofs} are unisolvent for $\mathbb P_{k+3}(T;\mathbb S\cap\mathbb T)$.
\end{lemma}
\begin{proof}
By Lemma~\ref{lem:cottSTfemDoFsdim},
the number of DoFs \eqref{eq:cottSTdofs} equals $\dim\mathbb P_{k+3}(T;\mathbb S\cap\mathbb T)$.

Take $\boldsymbol{\tau}\in\mathbb P_{k+3}(T;\mathbb S\cap\mathbb T)$, and assume all the DoFs \eqref{eq:cottSTdofs} vanish. 
The trace $\tr_{1}^{\cott}(\boldsymbol{\tau})$ is determined by DoFs \eqref{eq:3dCrcottSTfemdofV1}-\eqref{eq:3dCrcottSTfemdofE1} and \eqref{eq:3dCrcottSTfemdofF1}-\eqref{eq:3dCrcottSTfemdofF2} on each face, which gives $\tr_{1}^{\cott}(\boldsymbol{\tau})=0$.
Applying integration by parts, we obtain from \eqref{eq:202403201}-\eqref{eq:202403202} and the vanishing DoFs \eqref{eq:3dCrcottSTfemdofV1}-\eqref{eq:3dCrcottSTfemdofE3} that 
\begin{equation*}
(\div_f\div_f\tr_{2}^{\cott}(\boldsymbol{\tau}), q)_f=0, \quad \forall~q\in {\rm CH}(f), \quad \textrm{ if }\; r_1^{f}=-1,0.
\end{equation*}
If $r_1^f=1$, the vertex and edge data together with~\eqref{eq:3dCrcottSTfemdofF1} give $\boldsymbol\tau=\nabla\boldsymbol\tau=0$ on $f$, hence $\tr_2^{\cott}(\boldsymbol\tau)=0$. Otherwise, by~\eqref{eq:20240328}-\eqref{eq:202403202}, the vanishing DoFs \eqref{eq:3dCrcottSTfemdofV1}-\eqref{eq:3dCrcottSTfemdofE3}, \eqref{eq:3dCrcottSTfemdofE5}, and~\eqref{eq:3dCrcottSTfemdofF1}, and the preceding orthogonality,
\[
q_f:=\div_f\div_f\tr_2^{\cott}(\boldsymbol\tau)
=-\boldsymbol n_f^{\intercal}(\cott\boldsymbol\tau)\boldsymbol n_f
\in\mathbb B_k(f;(\boldsymbol r_2)_+)/{\rm CH}(f).
\]
When $r_2^e=-1$, the vertex data imply $q_f|_e\in\mathbb B_k(e;r_2^{\texttt{v}})$ for every $e\subset\partial f$, and~\eqref{eq:3dCrcottSTfemdofE5} then gives $q_f|_e=0$. Thus~\eqref{eq:3dCrcottSTfemdofF4} gives $q_f=0$, and~\eqref{eq:femCdivdivbubblecomplex2d} together with~\eqref{eq:3dCrcottSTfemdofF3} yields $\tr_2^{\cott}(\boldsymbol\tau)=0$.
Thanks to \eqref{eq:202403203}-\eqref{eq:202403204}, the vanishing DoFs \eqref{eq:3dCrcottSTfemdofV1}-\eqref{eq:3dCrcottSTfemdofE1}, \eqref{eq:3dCrcottSTfemdofE3}-\eqref{eq:3dCrcottSTfemdofE4} and the fact $\tr_{1}^{\cott}(\boldsymbol{\tau})=0$ indicate $\tr_{3}^{\cott}(\boldsymbol{\tau})|_f\in(\mathbb B_{k+1}(f;(\boldsymbol r_1\ominus2)_+)\otimes (\mathbb S(f)\cap\mathbb T(f)))$ on face $f$.
Further, by~\eqref{eq:nfecottntr3cott}, we acquire from the vanishing DoFs \eqref{eq:3dCrcottSTfemdofV1}-\eqref{eq:3dCrcottSTfemdofE1}, \eqref{eq:3dCrcottSTfemdofE5} and the fact $\tr_{1}^{\cott}(\boldsymbol{\tau})=0$ that $\tr_{3}^{\cott}(\boldsymbol{\tau})|_f\in\mathbb B_{k+1}^{\rot_f\rot_f}((\boldsymbol r_1\ominus2)_+;\mathbb S\cap\mathbb T)$, which together with the vanishing DoF \eqref{eq:3dCrcottSTfemdofF5} yields $\tr_{3}^{\cott}(\boldsymbol{\tau})=0$.
Consequently, $\boldsymbol{\tau}\in\mathbb B_{k+3}^{\cott}(\boldsymbol{r}_1;\mathbb S\cap\mathbb T)$, and we conclude $\boldsymbol{\tau}=0$ by the vanishing DoF \eqref{eq:3dCrcottSTfemdofT1}.
\end{proof}

For $\boldsymbol{r}_1\geq(6,0,-1)^{\intercal}$ with $r_1^f=-1,0,1$, the finite element space $\mathbb V_{k+3}^{\cott}(\boldsymbol{r}_1;\mathbb S\cap\mathbb T)$ is defined as follows
\begin{align*}
\mathbb V_{k+3}^{\cott}(\boldsymbol{r}_1;\mathbb S\cap\mathbb T):=\{\boldsymbol \tau\in L^2(\Omega;\mathbb S\cap \mathbb T): &\, \boldsymbol \tau|_T\in\mathbb P_{k+3}(T;\mathbb S\cap\mathbb T) \textrm{ for each } T\in\mathcal T_h, \\
&\qquad\qquad\;\;\; \textrm{ DoFs \eqref{eq:cottSTdofs} are single-valued} \}.    
\end{align*}
According to the proof of Lemma~\ref{lem:cottSTfemunisol}, traces $\tr_{1}^{\cott}(\boldsymbol{\tau})|_{f}$, $\tr_{2}^{\cott}(\boldsymbol{\tau})|_{f}$ and $\tr_{3}^{\cott}(\boldsymbol{\tau})|_{f}$ are determined by DoFs \eqref{eq:3dCrcottSTfemdofV1}-\eqref{eq:3dCrcottSTfemdofF5} on face $f$. Since $r_1^e\geq0$, the single-valued vertex DoFs \eqref{eq:3dCrcottSTfemdofV1} and the $j=0$ edge DoFs in \eqref{eq:3dCrcottSTfemdofE1} imply that $\boldsymbol{\tau}|_e$ is single-valued on every interior edge. Hence, Lemma~\ref{lm:cottconforming} gives $\mathbb V_{k+3}^{\cott}(\boldsymbol{r}_1;\mathbb S\cap\mathbb T)\subset H(\cott, \Omega; \mathbb S\cap \mathbb T)$.

\subsection{Finite element conformal elasticity complex}

\begin{theorem}\label{thm:femCEcomplex3d}
Let smoothness vectors $(\boldsymbol r_0, \boldsymbol r_1, \boldsymbol r_2, \boldsymbol r_3)$ be given by~\eqref{eq:rCE}, and $k\geq 2r_2^{\texttt{v}}+5$.
Then the finite element conformal elasticity complex \eqref{femCEcomplex3d} is exact.
\end{theorem}
\begin{proof}
The defining traces show that \eqref{femCEcomplex3d} is a complex. The kernel identity $\mathbb V_{k+3}^{\cott}(\boldsymbol{r}_1;\mathbb S\cap\mathbb T)\cap\ker(\cott)=\dev\defm\mathbb V_{k+4}(\boldsymbol{r}_0;\mathbb R^3)$ follows from Lemma~\ref{lem:continuousCEexact} and polynomial exactness~\eqref{polyconformalElascomplex3d}. The resulting piecewise polynomial $H^1$ potential has the required continuity by the normal-derivative argument in the proof of Theorem~\ref{thm:bubbleCEcomplex3d}, applied to jumps across each face. Thanks to the div stability \eqref{eq:divSTstability}, we only need to prove
\begin{equation}\label{eq:CEfemcomplexdim}
\begin{aligned}
\dim\mathbb V_{k+4}(\boldsymbol{r}_0;\mathbb R^3)&-\dim\mathbb V_{k+3}^{\cott}(\boldsymbol{r}_1;\mathbb S\cap\mathbb T)\\
& +\dim\mathbb V_{k}^{\div}(\boldsymbol{r}_2;\mathbb S\cap\mathbb T)-\dim\mathbb V_{k-1}(\boldsymbol{r}_3;\mathbb R^3)=10.
\end{aligned}
\end{equation}

Let $C_i$, $i=0,1,2,3$, denote the alternating entitywise DoF count in~\eqref{eq:CEfemcomplexdim} on one $i$-dimensional simplex. Then \eqref{eq:CEfemcomplexdim} is equivalent to
\begin{equation}\label{eq:CEfemcomplexdimC}
C_0|\Delta_0(\mathcal{T}_h)|+C_1|\Delta_1(\mathcal{T}_h)|+C_2|\Delta_2(\mathcal{T}_h)|+C_3|\Delta_3(\mathcal{T}_h)|=10.
\end{equation}
By the bubble complex \eqref{polybubbleCEcomplex3d}, $C_3=-10$.
We compute $C_0$ and $C_1$ as follows
\begin{align*}
C_0&=3{r_0^{\texttt{v}}+3\choose3}-5{r_1^{\texttt{v}}+3\choose3}+5{r_2^{\texttt{v}}+3\choose3}-3{r_3^{\texttt{v}}+3\choose3}=10, \\
C_1&= \frac{1}{2}(r_0^e+1)(r_0^e+2)(3k+2r_0^e-6r_0^{\texttt{v}}+9) \\*
&\quad -\frac{5}{6}(r_1^e+1)(r_1^e+2)(3k+2r_1^e-6r_1^{\texttt{v}}+6) -[r_1^e=0](4k-8r_1^{\texttt{v}}+15) \\*
&\quad -[r_1^e=1](3k-6r_1^{\texttt{v}}+12)+\frac{5}{6}(r_2^e+1)(r_2^e+2)(3k+2r_2^e-6r_2^{\texttt{v}}-3) \\*
&\quad -\frac{1}{2}(r_3^e+1)(r_3^e+2)(3k+2r_3^e-6r_3^{\texttt{v}}-6) =-10.
\end{align*}
Notice that
\begin{align*}
4C_0+6C_1+4C_2+C_3&=\dim\mathbb P_{k+4}(T;\mathbb R^3)-\dim\mathbb P_{k+3}(T;\mathbb S\cap\mathbb T)\\
&\quad +\dim\mathbb P_{k}(T;\mathbb S\cap\mathbb T)-\dim\mathbb P_{k-1}(T;\mathbb R^3)=10,
\end{align*}
in which we have used the polynomial conformal elasticity complex \eqref{polyconformalElascomplex3d}.
Hence, $C_2=10$.
Therefore, \eqref{eq:CEfemcomplexdimC} follows from Euler's formula \eqref{eulerformula}.
\end{proof}

\begin{example}\rm 
Taking $\boldsymbol{r}_0=(7,1,0)^{\intercal}$, $\boldsymbol{r}_1=\boldsymbol{r}_0-1$, $\boldsymbol{r}_2=\boldsymbol{r}_1\ominus3$, $\boldsymbol{r}_3=\boldsymbol{r}_2\ominus1$, and $k\geq11$, we obtain the finite element conformal elasticity complex
\begin{equation*}
{\rm CK} \hookrightarrow
\begin{pmatrix}
	 7\\
	 1\\
	 0
\end{pmatrix}
\xrightarrow{\dev\defm}
\begin{pmatrix}
 6\\
 0\\
 -1
\end{pmatrix}
\xrightarrow{\cott}
\begin{pmatrix}
	3\\
	-1\\
	-1
\end{pmatrix}
\xrightarrow{\div} 
\begin{pmatrix}
	 2\\
	 -1\\
	 -1
\end{pmatrix} \to 0,
\end{equation*}
which is less smooth than the complex in \cite{HuLinShi2023}.
\end{example}

\begin{remark}\label{rem:femCEtilde}\rm
Let $\boldsymbol{r}_0, \boldsymbol{r}_1, \boldsymbol{r}_2$ and $\boldsymbol{r}_3$ be smoothness vectors satisfying
\begin{equation*}
\boldsymbol{r}_0\geq(7,1,0)^{\intercal},  \quad 
\boldsymbol{r}_1=\boldsymbol{r}_0-1\geq(6,0,-1)^{\intercal},\quad \boldsymbol{r}_2\geq\boldsymbol{r}_1\ominus3,\quad \boldsymbol{r}_3\geq\boldsymbol{r}_2\ominus1.
\end{equation*}
Assume also that the terminal pair satisfies \eqref{eq:divSTstability}, with target smoothness $\boldsymbol r_2\ominus1$. This holds in the proved range of Lemma~\ref{lem:global-terminal-div}. Restricting the divergence to the target with smoothness $\boldsymbol r_3$ preserves this surjectivity. The tilde operation in \cite{ChenHuang2025,ChenHuang2024}, applied to complex \eqref{femCEcomplex3d}, then gives the following exact complex for sufficiently large $k$:
\begin{align*}
{\rm CK}\xrightarrow{\subset} \mathbb V_{k+4}(\boldsymbol{r}_0;\mathbb R^3)&\xrightarrow{\dev\defm} \mathbb V_{k+3}^{\cott}(\boldsymbol{r}_1,\boldsymbol{r}_2;\mathbb S\cap\mathbb T)\xrightarrow{\cott} \mathbb V_{k}^{\div}(\boldsymbol{r}_2,\boldsymbol{r}_3;\mathbb S\cap\mathbb T) \\
\notag
&\xrightarrow{\div} \mathbb V_{k-1}(\boldsymbol{r}_3;\mathbb R^3)\xrightarrow{}0,
\end{align*}
where
\begin{align*}
\mathbb V_{k+3}^{\cott}(\boldsymbol{r}_1,\boldsymbol{r}_2;\mathbb S\cap\mathbb T)&:=\{\boldsymbol{\tau}\in\mathbb V_{k+3}^{\cott}(\boldsymbol{r}_1;\mathbb S\cap\mathbb T): \cott\boldsymbol{\tau}\in\mathbb V_{k}^{\div}(\boldsymbol{r}_2;\mathbb S\cap\mathbb T)\}, \\
\mathbb V_{k}^{\div}(\boldsymbol{r}_2,\boldsymbol{r}_3;\mathbb S\cap\mathbb T)&:=\{\boldsymbol{\tau}\in\mathbb V_{k}^{\div}(\boldsymbol{r}_2;\mathbb S\cap\mathbb T): \div\boldsymbol{\tau}\in\mathbb V_{k-1}(\boldsymbol{r}_3;\mathbb R^3)\}.
\end{align*}
Taking $\boldsymbol{r}_0=(7,1,0)^{\intercal}$, $\boldsymbol{r}_1=\boldsymbol{r}_0-1$, $\boldsymbol{r}_2=(3,0,-1)^{\intercal}$, $\boldsymbol{r}_3=\boldsymbol{r}_2\ominus1$, and $k\geq11$, we recover the complex in~\cite{HuLinShi2023}:
\begin{equation*}
{\rm CK} \hookrightarrow
\begin{pmatrix}
	 7\\
	 1\\
	 0
\end{pmatrix}
\xrightarrow{\dev\defm}
\begin{pmatrix}
 6\\
 0\\
 -1
\end{pmatrix}
\xrightarrow{\cott}
\begin{pmatrix}
	3\\
	0\\
	-1
\end{pmatrix}
\xrightarrow{\div} 
\begin{pmatrix}
	 2\\
	 -1\\
	 -1
\end{pmatrix} \to 0.
\end{equation*}
\end{remark}

\appendix 

\section{Details for the continuous conformal elasticity construction}\label{apdx:continuous-exactness}

\begin{proof}[Proof of Lemma~\ref{lem:reducedelasticityexact}]
By $\div S\boldsymbol{\sigma}=2\vskw\curl\boldsymbol{\sigma}$,
\begin{equation*}
\div\widetilde{H}(\div; \mathbb S)=\div S H(\curl;\mathbb S)=2\vskw\curl H(\curl;\mathbb S)\subseteq L^2(\mathbb R^3),
\end{equation*}
which means $\widetilde{H}(\div; \mathbb S)\subseteq H(\div; \mathbb S)$.
Using
\[
H^1(\mathbb S)=SH^1(\mathbb S)\subseteq SH(\curl;\mathbb S)=\widetilde{H}(\div; \mathbb S),
\]
we obtain $H^1(\mathbb S)\subseteq\widetilde{H}(\div; \mathbb S)\subseteq H(\div; \mathbb S)$.
Since $\div H^1(\mathbb S)=L^2(\mathbb R^3)$ (cf. elasticity complex (34) in~\cite{ArnoldHu2021}), it follows that $\div\widetilde{H}(\div; \mathbb S)=L^2(\mathbb R^3)$.

Recall the elasticity complex~\cite[page 1232]{ArnoldAwanouWinther2008}
\begin{equation}\label{elascomplex3d}
{\rm RM}\xrightarrow{\subset} H^{1}(\mathbb R^3)\xrightarrow{\defm} H(\inc; \mathbb S)\xrightarrow{\inc} H(\div; \mathbb S) \xrightarrow{\div} L^2(\mathbb R^3)\xrightarrow{}0.
\end{equation}
Applying the tilde operation in~\cite[Section 2.3]{ChenHuang2025} to~\eqref{elascomplex3d} gives the exact complex
\begin{equation}\label{elascomplex3dtilde}
{\rm RM}\xrightarrow{\subset} H^{1}(\mathbb R^3)\xrightarrow{\defm} \widetilde{H}(\inc; \mathbb S)\xrightarrow{\inc} \widetilde{H}(\div; \mathbb S) \xrightarrow{\div} L^2(\mathbb R^3)\xrightarrow{}0,
\end{equation}
where
\[
\widetilde{H}(\inc; \mathbb S):=\{\boldsymbol{\tau}\in H(\inc; \mathbb S): \inc\boldsymbol{\tau}\in \widetilde{H}(\div; \mathbb S)\}.
\]
For $\boldsymbol{\tau}\in \widetilde{H}(\div;\mathbb S)\cap\ker(\div)\subseteq H(\div;\mathbb S)\cap\ker(\div)$, the identity $\inc H^2(\mathbb S)= H(\div;\mathbb S)\cap\ker(\div)$ (cf. elasticity complex (34) in~\cite{ArnoldHu2021}) gives $\boldsymbol{\sigma}\in H^2(\mathbb S)$ such that $\inc\boldsymbol{\sigma}=\boldsymbol{\tau}$. Hence $\boldsymbol{\sigma}\in \widehat{\widetilde{H}}(\inc; \mathbb S)$ and
\[
\widetilde{H}(\div; \mathbb S)\cap\ker(\div)= \inc\widehat{\widetilde{H}}(\inc; \mathbb S).
\]
The hat operation in~\cite[Lemma 2.1]{ChenHuang2025}, applied to~\eqref{elascomplex3dtilde}, now yields~\eqref{elascomplex3dtildehat}.
\end{proof}

\begin{proof}[Proof of Lemma~\ref{lem:continuousCEexact}]
Combine the elasticity complex~\eqref{elascomplex3dtildehat} and the Hessian complex~\cite{PaulyZulehner2020,ChenHuang2022d} to obtain
\begin{equation*}
\begin{tikzcd}
{\rm RM} \longrightarrow H^{1,2}(\div)
\arrow{r}{\defm}
&
\widehat{\widetilde{H}}(\inc; \mathbb S)
 \arrow{r}{\inc}
 &
\widetilde{H}(\div; \mathbb S)
 \arrow{r}{\div}
 & L^2(\mathbb R^3) \longrightarrow 0
 \\
\mathbb{P}_1 \longrightarrow H^2
 \arrow[ur,swap,"\iota"'] \arrow{r}{\hess}
 &
H(\curl;\mathbb S)
 \arrow[ur,swap,"S"'] \arrow{r}{\curl}
 &
H(\div; \mathbb T)
 \arrow[ur,swap,"-2\vskw"'] \arrow{r}{\div} \arrow[r]
 & L^2(\mathbb R^3) \longrightarrow 0.
\end{tikzcd}
\end{equation*}
The BGG construction induces the exact complex
\begin{equation*}
{\rm CK}\xrightarrow{\subset} H^{1,2}(\div)\xrightarrow{\dev\defm} H(\cott^+; \mathbb{S}\cap\mathbb{T})\xrightarrow{\cott} H(\div; \mathbb{S}\cap\mathbb{T}) \xrightarrow{\div} L^2(\mathbb R^3)\xrightarrow{}0,
\end{equation*}
where
\[
H(\cott^+; \mathbb{S}\cap\mathbb{T}):=\{\boldsymbol{\tau}\in H(\inc;\mathbb{S}\cap\mathbb{T}): \cott\boldsymbol{\tau}\in L^{2}(\mathbb{S}\cap\mathbb{T})\}.
\]
Since $H(\cott^+; \mathbb{S}\cap\mathbb{T})\subseteq H(\cott; \mathbb{S}\cap\mathbb{T})$, this gives
\[
\div H(\div; \mathbb{S}\cap\mathbb{T})=L^2(\mathbb R^3),\qquad
H(\div; \mathbb{S}\cap\mathbb{T})\cap\ker(\div)=\cott H(\cott; \mathbb{S}\cap\mathbb{T}).
\]

It remains to show $H(\cott; \mathbb{S}\cap\mathbb{T})\cap\ker(\cott)=\dev\defm H^1(\mathbb R^3)$.
For $\boldsymbol{\tau}$ in the kernel on the left, \eqref{eq:incandcott} gives $\sym\curl\boldsymbol{\tau}\in H^{-1}(\mathbb{S})\cap\ker(\inc)$. The elasticity complex~(34) in~\cite{ArnoldHu2021} yields $\sym\curl\boldsymbol{\tau}=\defm\boldsymbol{w}$ for some $\boldsymbol{w}\in L^2(\mathbb R^3)$. Moreover,
\[
\div\boldsymbol{w}=\tr(\defm\boldsymbol{w})=\tr(\curl\boldsymbol{\tau})=2\div(\vskw\boldsymbol{\tau})=0.
\]
Exactness of the de Rham complex gives $\boldsymbol{w}=\curl\boldsymbol{u}$ for some $\boldsymbol{u}\in H^1(\mathbb R^3)$. Therefore
\[
\sym\curl\boldsymbol{\tau}=\defm(\curl\boldsymbol{u})
=\sym\curl(\nabla\boldsymbol{u})
=\sym\curl(\dev(\nabla\boldsymbol{u})).
\]
By exactness of the div\,div complex, $\boldsymbol{\tau}=\dev(\nabla\boldsymbol{u})+\dev\grad\boldsymbol{v}$ for some $\boldsymbol{v}\in H^1(\mathbb R^3)$. Finally,
\[
\boldsymbol{\tau}=\sym\boldsymbol{\tau}
=\dev\defm(\boldsymbol{u}+\boldsymbol{v})
\in\dev\defm H^1(\mathbb R^3).
\]
\end{proof}

\section{The number of DoFs for $H(\cott;\mathbb S\cap\mathbb T)$-conforming finite element}\label{apdx:uinsol}
In this appendix, we count the number of DoFs \eqref{eq:cottSTdofs} for the $H(\cott;\mathbb S\cap\mathbb T)$-conforming finite element.

\begin{lemma}\label{lem:cottSTfemDoFsdim}
Let smoothness vectors $(\boldsymbol{r}_0, \boldsymbol{r}_1, \boldsymbol{r}_2)$ satisfy \eqref{eq:rCE} with $r_1^f=-1,0,1$, and $k\geq 2r_2^{\texttt{v}}+5$. 
The number of DoFs \eqref{eq:cottSTdofs} equals $\dim\mathbb P_{k+3}(T;\mathbb S\cap\mathbb T)$.
\end{lemma}
\begin{proof}
\step 1 Consider case $r_1^f=1$, then $\boldsymbol{r}_1\geq(7,3,1)^{\intercal}$ as $\boldsymbol{r}_0\geq(8,4,2)^{\intercal}$. Consequently, DoFs \eqref{eq:3dCrcottSTfemdofE2}-\eqref{eq:3dCrcottSTfemdofE5} and DoFs \eqref{eq:3dCrcottSTfemdofF2}-\eqref{eq:3dCrcottSTfemdofF4} disappear. The number of DoFs \eqref{eq:3dCrcottSTfemdofV1}-\eqref{eq:3dCrcottSTfemdofF5} is $5\dim\mathbb P_{k+3}(T)-5\dim\mathbb B_{k+3}(\boldsymbol{r}_1)+8\dim\mathbb B_{k+1}(f;\boldsymbol r_1-2)$. By \eqref{eq:cottbubblerf1},
\begin{equation*}
\dim\mathbb B_{k+3}^{\cott}(\boldsymbol{r}_1;\mathbb S\cap\mathbb T)=5\dim\mathbb B_{k+3}(\boldsymbol{r}_1)-8\dim\mathbb B_{k+1}(f;\boldsymbol r_1-2).
\end{equation*} 
Thus, the number of DoFs \eqref{eq:cottSTdofs} equals $\dim\mathbb P_{k+3}(T;\mathbb S\cap\mathbb T)$.

\step 2 Consider case $r_1^f=-1,0$ and $r_1^e\geq3$. Let $\bar{\boldsymbol{r}}_1=(r_1^{\texttt{v}}, r_1^e, 1)^{\intercal}$ and $\bar{\boldsymbol{r}}_0=\bar{\boldsymbol{r}}_1+1$.
Thanks to the bubble complex \eqref{eq:femCdivdivbubblecomplex2d} and the bubble complex \eqref{polybubbleCEcomplex3d}, the difference between the number of DoFs \eqref{eq:cottSTdofs} with smoothness vector $\boldsymbol{r}_1$ and that with smoothness vector $\bar{\boldsymbol{r}}_1$ is
\begin{align*}
&4\dim\mathbb B_{k+2}^{\div_f\div_f}(\boldsymbol{r}_1-1;\mathbb S\cap\mathbb T)-20\dim\mathbb B_{k+2}(f;\boldsymbol{r}_1-1) \\
&-12[r_1^f=-1]\dim\mathbb B_{k+3}(f;\boldsymbol{r}_1) +3\dim\mathbb B_{k+4}(\boldsymbol{r}_0)-3\dim\mathbb B_{k+4}(\bar{\boldsymbol{r}}_0),
\end{align*}
by $\dim\mathbb B_{k+2}^{\div_f\div_f}(\boldsymbol{r}_1-1;\mathbb S\cap\mathbb T)=2\dim\mathbb B_{k+2}(f;\boldsymbol{r}_1-1)$, which equals $0$.
Again, the number of DoFs \eqref{eq:cottSTdofs} equals $\dim\mathbb P_{k+3}(T;\mathbb S\cap\mathbb T)$.

\step 3 First consider $r_1^f=0$, $r_1^e=1,2$, and $r_1^{\texttt{v}}\geq7$. Then $\bar{\boldsymbol{r}}_1=(r_1^{\texttt{v}},3,0)^{\intercal}$ and $\bar{\boldsymbol{r}}_0=\bar{\boldsymbol{r}}_1+1$ are valid smoothness vectors.
Notice that
\begin{equation*}
\dim\mathbb B_{k+3}(f;\boldsymbol{r}_1)-\dim\mathbb B_{k+3}(f;\bar{\boldsymbol{r}}_1) = 3\dim\mathbb B_{k}(e;r_2^{\texttt{v}}) + 3[r_1^e=1]\dim\mathbb B_{k+1}(e;r_1^{\texttt{v}}-2),
\end{equation*}
\begin{equation*}
\dim\mathbb B_{k+4}(f;\boldsymbol{r}_0)-\dim\mathbb B_{k+4}(f;\bar{\boldsymbol{r}}_0) = 3\dim\mathbb B_{k}(e;r_2^{\texttt{v}}) + 3[r_1^e=1]\dim\mathbb B_{k+1}(e;r_1^{\texttt{v}}-2),
\end{equation*}
\begin{equation*}
\dim\mathbb B_{k+1}^{\rot_f\rot_f}((\boldsymbol r_1\ominus2)_+;\mathbb S\cap\mathbb T)-\dim\mathbb B_{k+1}^{\rot_f\rot_f}(\bar{\boldsymbol r}_1-2;\mathbb S\cap\mathbb T) = 3\dim\mathbb B_{k}(e;r_2^{\texttt{v}}).
\end{equation*}
By bubble complex \eqref{polybubbleCEcomplex3d}, we have
\begin{align*}
\dim\mathbb B_{k+3}^{\cott}(\boldsymbol{r}_1;\mathbb S\cap\mathbb T)-\dim\mathbb B_{k+3}^{\cott}(\bar{\boldsymbol{r}}_1;\mathbb S\cap\mathbb T)&= 3\dim\mathbb B_{k+4}(\boldsymbol{r}_0)- 3\dim\mathbb B_{k+4}(\bar{\boldsymbol{r}}_0) \\
&=18\dim\mathbb B_{k}(e;r_2^{\texttt{v}}).
\end{align*}
Thanks to the bubble complex \eqref{eq:femCdivdivbubblecomplex2d} and the bubble complex \eqref{polybubbleCEcomplex3d}, using the last four identities, the difference between the number of DoFs \eqref{eq:cottSTdofs} with smoothness vector $\boldsymbol{r}_1$ and that with smoothness vector $\bar{\boldsymbol{r}}_1$ is
\begin{align*}
 -102\dim\mathbb B_{k}(e;r_2^{\texttt{v}}) - 72[r_1^e=1]\dim\mathbb B_{k+1}(e;r_1^{\texttt{v}}-2)& \quad\quad \textrm{ (difference on edges) }\\
 + 84\dim\mathbb B_{k}(e;r_2^{\texttt{v}}) + 72[r_1^e=1]\dim\mathbb B_{k+1}(e;r_1^{\texttt{v}}-2)& \quad\quad \textrm{ (difference on faces) } \\
 + 18\dim\mathbb B_{k}(e;r_2^{\texttt{v}})& =0.
\end{align*}

For $r_1^{\texttt{v}}=6$, use instead the valid reference $\boldsymbol r_1=(6,2,0)^{\intercal}$, $\boldsymbol r_0=(7,3,1)^{\intercal}$. Counting \eqref{eq:cottSTdofs} by vertices, edges, faces, and the cell gives, for $k\geq11$,
\[
\begin{aligned}
N_V&=1680,&N_E&=198k-1686,\\
N_F&=18k^2-138k+116,&N_T&=(5k^3-33k^2+10k-60)/6.
\end{aligned}
\]
Here the cell count follows from the proved bubble complex \eqref{polybubbleCEcomplex3d}. Their sum is $5\binom{k+6}{3}$. Lowering the edge parameter from $2$ to $1$ changes the edge, face, and cell counts by $-72(k-8)$, $72(k-8)$, and $0$, respectively. Hence both lowest-vertex cases have the required count.

\step 4 Consider case $r_1^f=-1$ and $r_1^e=1,2$.
Let $\bar{\boldsymbol{r}}_1=(r_1^{\texttt{v}}, r_1^e, 0)^{\intercal}$ and $\bar{\boldsymbol{r}}_0=\bar{\boldsymbol{r}}_1+1$.
Thanks to the bubble complex \eqref{polybubbleCEcomplex3d}, the difference between the number of DoFs \eqref{eq:cottSTdofs} with smoothness vector $\boldsymbol{r}_1$ and that with smoothness vector $\bar{\boldsymbol{r}}_1$ is
\begin{equation*}
-12\dim\mathbb B_{k+3}(f;\boldsymbol{r}_1) +3\dim\mathbb B_{k+4}(\boldsymbol{r}_0)-3\dim\mathbb B_{k+4}(\bar{\boldsymbol{r}}_0)=0.
\end{equation*}

\step 5 Finally, for $\boldsymbol r_1=(r_1^{\texttt{v}},0,-1)^{\intercal}$, compare with $\bar{\boldsymbol r}_1=(r_1^{\texttt{v}},1,-1)^{\intercal}$. Put $d=\dim\mathbb B_{k+2}(e;r_1^{\texttt{v}}-1)$. The total vertex, edge, face, and cell DoF counts on $T$ differ by
\[
(0,-54d,36d,18d),
\]
respectively, with the cell difference following from~\eqref{polybubbleCEcomplex3d}. Their sum is zero.
Therefore, the number of DoFs \eqref{eq:cottSTdofs} equals $\dim\mathbb P_{k+3}(T;\mathbb S\cap\mathbb T)$.
\end{proof}

\bibliographystyle{abbrv}
\bibliography{paper}

\begin{thebibliography}{10}

\bibitem{Arnold2018}
D.~N. Arnold.
\newblock {\em Finite element exterior calculus}, volume~93 of {\em CBMS-NSF Regional Conference Series in Applied Mathematics}.
\newblock Society for Industrial and Applied Mathematics (SIAM), Philadelphia, PA, 2018.

\bibitem{ArnoldAwanouWinther2008}
D.~N. Arnold, G.~Awanou, and R.~Winther.
\newblock Finite elements for symmetric tensors in three dimensions.
\newblock {\em Math. Comp.}, 77(263):1229--1251, 2008.

\bibitem{ArnoldFalkWinther2006}
D.~N. Arnold, R.~S. Falk, and R.~Winther.
\newblock Finite element exterior calculus, homological techniques, and applications.
\newblock {\em Acta Numer.}, 15:1--155, 2006.

\bibitem{ArnoldFalkWinther2010}
D.~N. Arnold, R.~S. Falk, and R.~Winther.
\newblock Finite element exterior calculus: from {H}odge theory to numerical stability.
\newblock {\em Bull. Amer. Math. Soc. (N.S.)}, 47(2):281--354, 2010.

\bibitem{ArnoldHu2021}
D.~N. Arnold and K.~Hu.
\newblock Complexes from complexes.
\newblock {\em Found. Comput. Math.}, 21(6):1739--1774, 2021.

\bibitem{Beig1997}
R.~Beig.
\newblock T{T}-tensors and conformally flat structures on {$3$}-manifolds.
\newblock In {\em Mathematics of gravitation, {P}art {I} ({W}arsaw, 1996)}, volume 41, Part I of {\em Banach Center Publ.}, pages 109--118. Polish Acad. Sci. Inst. Math., Warsaw, 1997.

\bibitem{BeigChrusciel2020}
R.~Beig and P.~T. Chru\'{s}ciel.
\newblock On linearised vacuum constraint equations on {E}instein manifolds.
\newblock {\em Classical Quantum Gravity}, 37(21):215012, 14, 2020.

\bibitem{BerchenkoKoganGawlik2025}
Y.~Berchenko-Kogan and E.~S. Gawlik.
\newblock Finite element spaces of double forms.
\newblock {\em Numer. Math.}, https://doi.org/10.1007/s00211-026-01579-7, 2026.

\bibitem{BonizzoniHuKanschatSap2025}
F.~Bonizzoni, K.~Hu, G.~Kanschat, and D.~Sap.
\newblock Discrete tensor product {BGG} sequences: {S}plines and finite elements.
\newblock {\em Math. Comp.}, 94(352):517--549, 2025.

\bibitem{ChenChenGaoHuangEtAl2025}
C.~Chen, L.~Chen, T.~Gao, X.~Huang, and H.~Wei.
\newblock Implementation and basis construction for smooth finite element spaces.
\newblock {\em arXiv preprint arXiv:2507.19732}, 2025.

\bibitem{ChenChenHuangWei2024}
C.~Chen, L.~Chen, X.~Huang, and H.~Wei.
\newblock Geometric decomposition and efficient implementation of high order face and edge elements.
\newblock {\em Commun. Comput. Phys.}, 35(4):1045--1072, 2024.

\bibitem{ChenHuHuang2018}
L.~Chen, J.~Hu, and X.~Huang.
\newblock Multigrid methods for {H}ellan-{H}errmann-{J}ohnson mixed method of {K}irchhoff plate bending problems.
\newblock {\em J. Sci. Comput.}, 76(2):673--696, 2018.

\bibitem{ChenHuang2018}
L.~Chen and X.~Huang.
\newblock Decoupling of mixed methods based on generalized {H}elmholtz decompositions.
\newblock {\em SIAM J. Numer. Anal.}, 56(5):2796--2825, 2018.

\bibitem{ChenHuang2020}
L.~Chen and X.~Huang.
\newblock Finite elements for divdiv-conforming symmetric tensors.
\newblock {\em arXiv preprint arXiv:2005.01271}, 2020.

\bibitem{ChenHuang2021}
L.~Chen and X.~Huang.
\newblock Geometric decompositions of the simplicial lattice and smooth finite elements in arbitrary dimension.
\newblock {\em arXiv preprint arXiv:2111.10712}, 2021.

\bibitem{ChenHuang2022d}
L.~Chen and X.~Huang.
\newblock Discrete {H}essian complexes in three dimensions.
\newblock In {\em The virtual element method and its applications}, volume~31 of {\em SEMA SIMAI Springer Ser.}, pages 93--135. Springer, Cham, 2022.

\bibitem{ChenHuang2022b}
L.~Chen and X.~Huang.
\newblock A finite element elasticity complex in three dimensions.
\newblock {\em Math. Comp.}, 91(337):2095--2127, 2022.

\bibitem{ChenHuang2022a}
L.~Chen and X.~Huang.
\newblock Finite elements for div- and divdiv-conforming symmetric tensors in arbitrary dimension.
\newblock {\em SIAM J. Numer. Anal.}, 60(4):1932--1961, 2022.

\bibitem{ChenHuang2022}
L.~Chen and X.~Huang.
\newblock Finite elements for {${\rm div\,div}$} conforming symmetric tensors in three dimensions.
\newblock {\em Math. Comp.}, 91(335):1107--1142, 2022.

\bibitem{ChenHuang2024}
L.~Chen and X.~Huang.
\newblock Finite element de {R}ham and {S}tokes complexes in three dimensions.
\newblock {\em Math. Comp.}, 93(345):55--110, 2024.

\bibitem{ChenHuang2024b}
L.~Chen and X.~Huang.
\newblock {$H ({\rm div})$}-conforming finite element tensors with constraints.
\newblock {\em Results Appl. Math.}, 23:Paper No. 100494, 33, 2024.

\bibitem{ChenHuang2024c}
L.~Chen and X.~Huang.
\newblock Tangential-normal decompositions of the second family of finite element differential forms.
\newblock {\em arXiv preprint arXiv:2410.20408}, 2024.

\bibitem{ChenHuang2024a}
L.~Chen and X.~Huang.
\newblock Finite element complexes in two dimensions (in {C}hinese).
\newblock {\em Sci. Sin. Math.}, 55(8):1593--1626, 2025.

\bibitem{ChenHuang2025b}
L.~Chen and X.~Huang.
\newblock Hybridizable symmetric stress elements on the barycentric refinement in arbitrary dimensions.
\newblock {\em Math. Comp.}, https://doi.org/10.1090/mcom/4180, 2025.

\bibitem{ChenHuang2025a}
L.~Chen and X.~Huang.
\newblock A new div-div-conforming symmetric tensor finite element space with applications to the biharmonic equation.
\newblock {\em Math. Comp.}, 94(351):33--72, 2025.

\bibitem{ChenHuang2025}
L.~Chen and X.~Huang.
\newblock Complexes from complexes: finite element complexes in three dimensions.
\newblock {\em Math. Comp.}, 95(359):1083--1142, 2026.

\bibitem{ChenHuangSunTian2026}
L.~Chen, X.~Huang, Y.~Sun, and S.~Tian.
\newblock Hybridizable staggered discontinuous {Galerkin} methods for polyharmonic equations on polytopes.
\newblock {\em arXiv preprint arXiv:2607.00831}, 2026.

\bibitem{ChenHuangZhang2025}
L.~Chen, X.~Huang, and C.~Zhang.
\newblock Distributional finite element curl div complexes and application to quad curl problems.
\newblock {\em SIAM J. Numer. Anal.}, 63(3):1078--1104, 2025.

\bibitem{Christiansen2011}
S.~H. Christiansen.
\newblock On the linearization of {R}egge calculus.
\newblock {\em Numer. Math.}, 119(4):613--640, 2011.

\bibitem{ChristiansenGopalakrishnanGuzmanHu2024}
S.~H. Christiansen, J.~Gopalakrishnan, J.~Guzm\'an, and K.~Hu.
\newblock A discrete elasticity complex on three-dimensional {A}lfeld splits.
\newblock {\em Numer. Math.}, 156(1):159--204, 2024.

\bibitem{ChristiansenHuHu2018}
S.~H. Christiansen, J.~Hu, and K.~Hu.
\newblock Nodal finite element de {R}ham complexes.
\newblock {\em Numer. Math.}, 139(2):411--446, 2018.

\bibitem{ChristiansenHu2023}
S.~H. Christiansen and K.~Hu.
\newblock Finite element systems for vector bundles: elasticity and curvature.
\newblock {\em Found. Comput. Math.}, 23(2):545--596, 2023.

\bibitem{Dain2006}
S.~Dain.
\newblock Generalized {K}orn's inequality and conformal {K}illing vectors.
\newblock {\em Calc. Var. Partial Differential Equations}, 25(4):535--540, 2006.

\bibitem{GennesProst1993}
P.-G. de~Gennes and J.~Prost.
\newblock {\em The Physics of Liquid Crystals}.
\newblock Clarendon Press, Oxford, 2 edition, 1993.

\bibitem{FeireislNovotny2017}
E.~Feireisl and A.~Novotn\'{y}.
\newblock {\em Singular limits in thermodynamics of viscous fluids}.
\newblock Advances in Mathematical Fluid Mechanics. Birkh\"{a}user/Springer, Cham, 2017.
\newblock Second edition.

\bibitem{FuehrerHeuerNiemi2019}
T.~F\"{u}hrer, N.~Heuer, and A.~H. Niemi.
\newblock An ultraweak formulation of the {K}irchhoff-{L}ove plate bending model and {DPG} approximation.
\newblock {\em Math. Comp.}, 88(318):1587--1619, 2019.

\bibitem{GopalakrishnanLedererSchoeberl2020}
J.~Gopalakrishnan, P.~L. Lederer, and J.~Sch\"{o}berl.
\newblock A mass conserving mixed stress formulation for {S}tokes flow with weakly imposed stress symmetry.
\newblock {\em SIAM J. Numer. Anal.}, 58(1):706--732, 2020.

\bibitem{GuoHuLin2025}
Y.~Guo, J.~Hu, and T.~Lin.
\newblock Discretizing linearized {E}instein-{B}ianchi system by symmetric and traceless tensors.
\newblock {\em arXiv preprint arXiv:2508.04560}, 2025.

\bibitem{GuoHuLin2025a}
Y.~Guo, J.~Hu, and T.~Lin.
\newblock Discretizing linearized {E}instein-{B}ianchi system by symmetric and traceless tensors.
\newblock In {\em The 11th Peking University Workshop on Numerical Methods for Partial Differential Equations}, July 2025.

\bibitem{GuzmanHuang2026}
J.~Guzm\'{a}n and X.~Huang.
\newblock Low-regularity finite element elasticity complexes with hybridizable stresses on tetrahedral {Alfeld} splits.
\newblock {\em arXiv preprint arXiv:2607.01933}, 2026.

\bibitem{Hatcher2002}
A.~Hatcher.
\newblock {\em Algebraic topology}.
\newblock Cambridge University Press, Cambridge, 2002.

\bibitem{HuLiang2021}
J.~Hu and Y.~Liang.
\newblock Conforming discrete {G}radgrad-complexes in three dimensions.
\newblock {\em Math. Comp.}, 90(330):1637--1662, 2021.

\bibitem{HuLiangLin2024}
J.~Hu, Y.~Liang, and T.~Lin.
\newblock Finite element grad grad complexes and elasticity complexes on cuboid meshes.
\newblock {\em J. Sci. Comput.}, 99(2):Paper No. 50, 29, 2024.

\bibitem{HuLiangMa2022}
J.~Hu, Y.~Liang, and R.~Ma.
\newblock Conforming finite element divdiv complexes and the application for the linearized {E}instein-{B}ianchi system.
\newblock {\em SIAM J. Numer. Anal.}, 60(3):1307--1330, 2022.

\bibitem{HuLiangMaZhang2024}
J.~Hu, Y.~Liang, R.~Ma, and M.~Zhang.
\newblock A family of conforming finite element divdiv complexes on cuboid meshes.
\newblock {\em Numer. Math.}, 156(4):1603--1638, 2024.

\bibitem{HuLinWu2024}
J.~Hu, T.~Lin, and Q.~Wu.
\newblock A construction of {$C^r$} conforming finite element spaces in any dimension.
\newblock {\em Found. Comput. Math.}, 24(6):1941--1977, 2024.

\bibitem{HuMaZhang2021}
J.~Hu, R.~Ma, and M.~Zhang.
\newblock A family of mixed finite elements for the biharmonic equations on triangular and tetrahedral grids.
\newblock {\em Sci. China Math.}, 64(12):2793--2816, 2021.

\bibitem{HuLin2025}
K.~Hu and T.~Lin.
\newblock Finite element form-valued forms: Construction.
\newblock {\em arXiv preprint arXiv:2503.03243}, 2025.

\bibitem{HuLinShi2023}
K.~Hu, T.~Lin, and B.~Shi.
\newblock Finite elements for symmetric and traceless tensors in three dimensions.
\newblock {\em Math. Comp.}, https://doi.org/10.1090/mcom/4233, 2026.

\bibitem{HuLinZhang2025}
K.~Hu, T.~Lin, and Q.~Zhang.
\newblock Distributional {H}essian and divdiv complexes on triangulation and cohomology.
\newblock {\em SIAM J. Appl. Algebra Geom.}, 9(1):108--153, 2025.

\bibitem{HuangTang2025}
X.~Huang and Z.~Tang.
\newblock Robust and optimal mixed methods for a fourth-order elliptic singular perturbation problem.
\newblock {\em J. Sci. Comput.}, 105(3):Paper No. 72, 29, 2025.

\bibitem{HuangTang2026}
X.~Huang and Z.~Tang.
\newblock A lowest-order robust mixed finite element method with a third-order tensor variable for strain gradient elasticity.
\newblock {\em arXiv preprint arXiv:2607.17134}, 2026.

\bibitem{HuangZhangZhouZhu2024}
X.~Huang, C.~Zhang, Y.~Zhou, and Y.~Zhu.
\newblock New low-order mixed finite element methods for linear elasticity.
\newblock {\em Adv. Comput. Math.}, 50(2):Paper No. 17, 31, 2024.

\bibitem{HuangZhao2026}
X.~Huang and X.~Zhao.
\newblock Symmetric-tensor distributional mixed method for fourth-order elliptic singular perturbation problem.
\newblock {\em arXiv preprint arXiv:2606.02188}, 2026.

\bibitem{Iverson1962}
K.~E. Iverson.
\newblock A programming language.
\newblock In {\em Proceedings of the May 1-3, 1962, spring joint computer conference}, pages 345--351, 1962.

\bibitem{LewintanMuellerNeff2021}
P.~Lewintan, S.~M\"{u}ller, and P.~Neff.
\newblock Korn inequalities for incompatible tensor fields in three space dimensions with conformally invariant dislocation energy.
\newblock {\em Calc. Var. Partial Differential Equations}, 60(4):Paper No. 150, 46, 2021.

\bibitem{Nedelec1980}
J.-C. N{\'e}d{\'e}lec.
\newblock Mixed finite elements in {${\bf R}^{3}$}.
\newblock {\em Numer. Math.}, 35(3):315--341, 1980.

\bibitem{NeffJeong2009}
P.~Neff and J.~Jeong.
\newblock A new paradigm: the linear isotropic {C}osserat model with conformally invariant curvature energy.
\newblock {\em ZAMM Z. Angew. Math. Mech.}, 89(2):107--122, 2009.

\bibitem{PaulyZulehner2020}
D.~Pauly and W.~Zulehner.
\newblock The div{D}iv-complex and applications to biharmonic equations.
\newblock {\em Appl. Anal.}, 99(9):1579--1630, 2020.

\bibitem{QuennevilleBelair2015}
V.~Quenneville-Belair.
\newblock {\em A New Approach to Finite Element Simulations of General Relativity}.
\newblock ProQuest LLC, Ann Arbor, MI, 2015.
\newblock Thesis (Ph.D.)--University of Minnesota.

\bibitem{RaviartThomas1977}
P.-A. Raviart and J.~M. Thomas.
\newblock A mixed finite element method for 2nd order elliptic problems.
\newblock In {\em Mathematical aspects of finite element methods ({P}roc. {C}onf., {C}onsiglio {N}az. delle {R}icerche ({C}.{N}.{R}.), {R}ome, 1975)}, pages 292--315. Lecture Notes in Math., Vol. 606. Springer, Berlin, 1977.

\bibitem{Sokolnikoff1946}
I.~S. Sokolnikoff.
\newblock {\em Mathematical Theory of Elasticity}.
\newblock McGraw-Hill, New York, 1946.

\bibitem{CapHu2024}
A.~\v{C}ap and K.~Hu.
\newblock {BGG} sequences with weak regularity and applications.
\newblock {\em Found. Comput. Math.}, 24(4):1145--1184, 2024.

\bibitem{WilliamsHong2024}
D.~M. Williams and Q.~Hong.
\newblock Generalized {K}orn's inequalities for piecewise {$H^1$} and {$H^2$} vector fields.
\newblock {\em Math. Comp.}, 93(350):2587--2609, 2024.

\end{thebibliography}
\end{document}